\documentclass[10pt,a4paper,oneside,reqno]{amsart}

\usepackage{setspace}
\usepackage[totalwidth=14cm,totalheight=22cm]{geometry}
\usepackage[T1]{fontenc}
\usepackage[dvips]{graphicx}
\usepackage{epsfig}
\usepackage{amsmath,amssymb,amscd,amsthm}
\usepackage{subfigure}
\usepackage[latin1]{inputenc}
\usepackage{latexsym}
\usepackage{graphics}
\usepackage{ae}
\usepackage{enumitem}
\usepackage[all]{xy}
\usepackage{scalerel}
\usepackage{cancel}
\usepackage{mathrsfs}
\usepackage[dvipsnames]{xcolor}
\usepackage[bookmarks=true, colorlinks=true, urlcolor=Blue,
linkcolor=BrickRed, citecolor=MidnightBlue, pagebackref=true]{hyperref}
\usepackage[normalem]{ulem}
\usepackage{nicefrac}
\usepackage{xfrac}

\makeatletter
\newtheorem*{rep@theorem}{\rep@title}
\newcommand{\newreptheorem}[2]{%
\newenvironment{rep#1}[1]{%
 \def\rep@title{#2 \ref{##1}}%
 \begin{rep@theorem}}%
 {\end{rep@theorem}}}
\makeatother

\makeatletter
\newtheorem*{rep@cor}{\rep@title}
\newcommand{\newrepcor}[2]{%
\newenvironment{rep#1}[1]{%
 \def\rep@title{#2 \ref{##1}}%
 \begin{rep@cor}}%
 {\end{rep@cor}}}
\makeatother

\makeatletter
\newtheorem*{rep@prop}{\rep@title}
\newcommand{\newrepprop}[2]{%
\newenvironment{rep#1}[1]{%
 \def\rep@title{#2 \ref{##1}}%
 \begin{rep@prop}}%
 {\end{rep@prop}}}
\makeatother

\newtheorem{cor}{Corollary}[section]

\newtheorem{theorem}[cor]{Theorem}

\newtheorem{prop}[cor]{Proposition}

\newrepcor{cor}{Corollary}
\newreptheorem{theorem}{Theorem}
\newrepprop{prop}{Proposition}

\newtheorem{lemma}[cor]{Lemma}

\theoremstyle{definition}
\newtheorem{defi}[cor]{Definition}
\theoremstyle{remark}
\newtheorem{remark}[cor]{Remark}
\newtheorem*{remark*}{Remark}
\newtheorem{example}[cor]{Example}

\newtheorem*{notation*}{Notation}

\newlist{steps}{enumerate}{1}
\setlist[steps, 1]{itemsep=8pt,leftmargin=0cm,itemindent=.5cm,labelwidth=\itemindent,labelsep=0cm,align=left,label = \textbf{\emph{Step \arabic*}:\,}}

\newcommand{\R}{{\mathbb R}}
\newcommand{\Z}{{\mathbb Z}}
\newcommand{\dev}{\mathrm{dev}}

\newcommand{\Hyp}{\mathbb{H}}

\newcommand{\AdS}{\mathbb{A}\mathrm{d}\mathbb{S}}
\newcommand{\HP}{\mathsf{HP}}
\newcommand{\HS}{\mathsf{HS}}
\newcommand{\dS}{\mathrm{d}\mathbb{S}}
\newcommand{\RP}{\mathsf{P}}
\newcommand{\SP}{\mathsf{S}}
\newcommand{\Aff}{\mathsf{A}}

\newcommand{\PGL}{\mathrm{PGL}}
\newcommand{\GL}{\mathrm{GL}}

\renewcommand{\r}{\mathfrak{r}}

\newcommand{\Isom}{\mathrm{Isom}}
\newcommand{\Aut}{\mathrm{Aut}}
\renewcommand{\O}{\mathrm{O}}
\newcommand{\X}{\mathbb{X}}

\newcommand{\p}[1]{\ensuremath{\boldsymbol{#1^+} }}
\newcommand{\m}[1]{\ensuremath{\boldsymbol{#1^-} } }
\renewcommand{\l}[1]{\ensuremath{\boldsymbol{#1}} }

\begin{document}

\title[Four-dimensional geometric transition]{Geometric transition from hyperbolic to anti-de Sitter structures in dimension four}

\author[Stefano Riolo]{Stefano Riolo}
\address{Stefano Riolo: Dipartimento di Matematica\\ Universit\`a di Bologna \newline Piazza di Porta San Donato 5\\ 40126 Bologna\\ Italy}
\email{stefano.riolo@unibo.it}

\author[Andrea Seppi]{Andrea Seppi}
\address{Andrea Seppi: CNRS and Universit\'e Grenoble Alpes \newline 100 Rue des Math\'ematiques\\ 38610 Gi\`eres\\ France} \email{andrea.seppi@univ-grenoble-alpes.fr}

\thanks{The authors were partially supported by FIRB 2010 project ``Low dimensional geometry and topology'' (RBFR10GHHH003), and are members of the national research group GNSAGA.
The first author was supported by the Mathematics Department of the University of Pisa (research fellowship ``Deformazioni di strutture iperboliche in dimensione quattro''), and by the Swiss National Science Foundation (project no.~PP00P2-170560).
}


\begin{abstract}
We provide the first examples of geometric transition from hyperbolic to anti-de Sitter structures in dimension four, in a fashion similar to Danciger's three-dimensional examples.
The main ingredient is a deformation of hyperbolic 4-polytopes, discovered by Kerckhoff and Storm, eventually collapsing to a 3-dimensional ideal cuboctahedron. We show the existence of a similar family of collapsing anti-de Sitter polytopes, and join the two deformations by means of an opportune half-pipe orbifold structure. The desired examples of geometric transition are then obtained by gluing copies of the polytope.
\end{abstract}

\maketitle

\section{Introduction}

In this paper we provide explicit examples of geometric transition in dimension four. Before stating the main result (Theorem \ref{teo: main} below), we begin with some motivational preliminaries in dimension three.

\subsection*{Degeneration and transition}

In his famous notes \cite{thurstonnotes}, Thurston introduced a phenomenon called \emph{degeneration} of hyperbolic structures. Several contribuitions have then been given on this topic \cite{Hthesis,P98,P01,P02,MR2140265,P07,P13,Kozai_thesis,LM2,LM1,Kozai}, which plays an important role in the proof of the celebrated Orbifold Theorem \cite{BLP,CHK}. 

As an example, for some closed hyperbolic 3-orbifolds $\mathcal X$, singular along a knot $\Sigma\subset\mathcal X$ with cone angle $\frac{2\pi}m$, the following holds. There is a path $\theta\mapsto\mathcal X_\theta$ of hyperbolic cone-manifold structures on $\mathcal X$ with singular locus $\Sigma$ and cone angle $\theta\in\left[\frac{2\pi}m,2\pi\right)$, such that $\mathcal X_\theta$ collapses to a lower-dimensional orbifold as $\theta\to2\pi$. This holds, for instance, when $\mathcal X$ is an exceptional Dehn filling of the figure-eight knot complement admitting a Seifert fibration $\mathcal X\to\mathcal N$ with base a hyperbolic 2-orbifold $\mathcal N$. As $\theta\to2\pi$, the cone-manifold $\mathcal X_\theta$ collapses to $\mathcal N$, whose hyperbolic structure is said to \emph{regenerate} to 3-dimensional hyperbolic structures. 

The familiar idea of going from spherical to hyperbolic geometry, through Euclidean geometry, was known since Klein \cite{AP}. This is a continuous process inside projective geometry, seen as a common ``ambient'' geometry. This phenomenon, called \emph{geometric transition}, has been recently studied in greater generality by Cooper Danciger and Wienhard \cite{CDW} (see also \cite{trettel_thesis}) through the notion of \emph{limit geometry}. For example, among others, Euclidean geometry is a limit of both spherical and hyperbolic geometries inside projective geometry.

Let us come back to our hyperbolic cone 3-manifolds $\mathcal X_\theta$ collapsing to the hyperbolic 2-orbifold $\mathcal N$. The work of Danciger \cite{danciger,dancigertransition,dancigerideal} shows that in many such cases the hyperbolic structure of $\mathcal N$ regenerates to anti-de Sitter (AdS for short, the Lorentzian analogue of hyperbolic geometry) structures on $\mathcal X$, where the singular locus $\Sigma$ is a spacelike geodesic. Moreover, the two deformations are joined continuously via projective geometry so as to have geometric transition. To this purpose, Danciger introduced the so called \emph{half-pipe} (HP for short) geometry, which is a limit geometry \cite{CDW} inside projective geometry of both hyperbolic and anti-de Sitter geometry. Half-pipe space naturally identifies with the space of spacelike hyperplanes in Minkowski space $\R^{1,n-1}$, and its group of transformations, which is a Chabauty limit of both $\Isom(\Hyp^n)$ and $\Isom(\AdS^n)$, is isomorphic to $\Isom(\R^{1,n-1})$ by means of this duality. Suitable projective transformations are used to ``rescale'' the hyperbolic and AdS metric along the direction of collapse, thus obtaining geometric transition via half-pipe geometry.

In \cite[Theorem 1.1]{dancigertransition}, Danciger provides an infinite class of Seifert 3-manifolds $\mathcal X$ (unit tangent bundles of some hyperbolic 2-orbifolds) supporting such a kind of geometric transition. Also, \cite[Theorem 1.2]{dancigertransition} is a regeneration result of half-pipe structures under a cohomological condition: the 1-dimensionality of the twisted cohomology group $H^1_{\mathrm{Ad}\,\rho}(\pi_1(\mathcal X\smallsetminus\Sigma),\mathfrak{so}(1,2))$, where $\rho\colon\pi_1(\mathcal X\smallsetminus\Sigma)\to\Isom(\Hyp^2)$ is the representation associated to the degenerate structure and $\mathrm{Ad}\colon\Isom(\Hyp^2)\to\mathrm{Aut}(\mathfrak{so}(1,2))$ is the adjoint representation.

\subsection*{Examples of transition in dimension four}

It seems natural to ask whether this phenomenon is purely three-dimensional, or if it can happen also in higher dimension, where hyperbolic structures are typically more rigid. In this paper we answer affirmatively in dimension four. We indeed build some examples of geometric transition from hyperbolic to AdS structures. The construction is explicitly obtained by gluing copies of a hyperbolic or AdS collapsing 4-polytope. 

The study of deformations of 4-dimensional hyperbolic cone-manifolds is quite recent, and in general very little is known on this topic. Recently, Martelli and the first author \cite[Theorem 1.2]{MR} provided the first example of degeneration of hyperbolic cone structures on a 4-manifold to a 3-dimensional hyperbolic structure. We show that in this case there is geometric transition from hyperbolic to AdS structures, and provide an infinite class of such examples. The existence of such a phenomenon is a  novelty in dimension four. Precisely, we show the following:

\begin{theorem} \label{teo: main}
Let $\mathcal{N}$ be a hyperbolic 3-manifold that finitely orbifold-covers the ideal right-angled cuboctahedron. There exists a $C^1$ family $\{\sigma_t\}_{t\in\left(-\epsilon,\epsilon\right]}$ of simple projective cone-manifold structures on the 4-manifold
$$\mathcal X=\mathcal{N}\times S^1,$$
singular along a compact foam $\Sigma\subset\mathcal X$, such that $\sigma_t$ is conjugated to a cusped, finite-volume,
\begin{itemize}
\item hyperbolic orbifold structure with cone angles $\pi$ as $t=\epsilon$,
\item hyperbolic cone structure with decreasing cone angles $\alpha_t\in[\pi,2\pi)$ as $t>0$,
\item half-pipe structure with spacelike singularity as $t=0$,
\item anti-de Sitter structure with spacelike singularity of increasing magnitude $\beta_t\in (-\infty,0)$ as $t<0$.
\end{itemize}
As $t\to0^+$ (resp. $t\to0^-$), we have $\alpha_t\to2\pi$ (resp. $\beta_t\to0$) and the induced hyperbolic (resp. AdS) structures on $\mathcal X\smallsetminus\Sigma$ degenerate to the complete hyperbolic structure of $\mathcal{N}$.
\end{theorem}

Similarly to Danciger's \cite[Theorem 1.1]{dancigertransition}, but in higher dimension, there is a circle bundle over a hyperbolic orbifold (a 2-orbifold in his case, a 3-manifold in ours), and geometric transition from hyperbolic to AdS singular structures on the total space of the bundle with collapse to the base. Let us briefly explain some terminology used in the statement of Theorem \ref{teo: main}.

The \emph{cuboctahedron}, drawn in Figure \ref{fig:cuboct}, is a well-known uniform polyhedron whose \emph{ideal} hyperbolic counterpart $\mathcal C\subset\Hyp^3$ is \emph{right-angled}. As such, the polyhedron $\mathcal C$ can be seen as a cusped hyperbolic 3-orbifold.

\begin{figure}
\includegraphics[scale=.605]{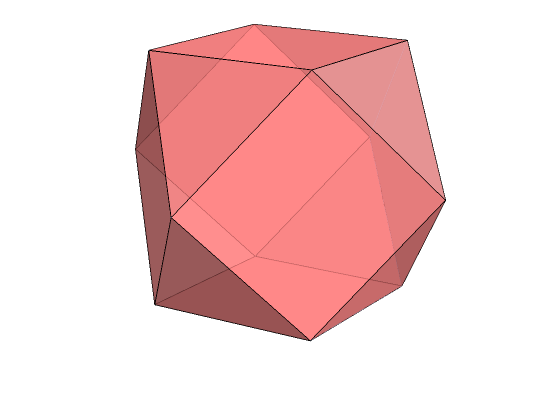}

\vspace{-.8cm}

\caption[The cuboctahedron.]{\footnotesize A cuboctahedron in $\R^3$. The ideal right-angled cuboctahedron in $\Hyp^3$ can be seen as a cusped hyperbolic 3-orbifold. }\label{fig:cuboct}
\end{figure}

Roughly speaking, \emph{simple projective cone-manifolds} (Definition \ref{def:proj_simple_cone-mfd}) are singular real projective manifolds locally modelled on the double of a simple polytope in projective space. The singular locus $\Sigma\subset \mathcal{X}$ of an $n$-dimensional simple projective cone-manifold $\mathcal{X}$ is an $(n-2)$-complex with generic singularities: {if} $n=1,2,3$ or $4$, {the set} $\Sigma$ is empty, discrete, a trivalent graph or a foam, respectively. A \emph{foam} is a 2-complex locally modelled on the cone over the 1-skeleton of the tetrahedron; see Figure \ref{fig: Sigma_local}. Our singular locus is not a surface, as it has edges and vertices. However foams are quite natural objects in dimension four (like trivalent graphs in 3-manifolds). To the best of our knowledge, it is not known whether there can even exist deformations of 4-dimensional, finite-volume, hyperbolic cone-manifolds with singular locus an embedded surface.

\begin{figure}
\includegraphics[scale=.18]{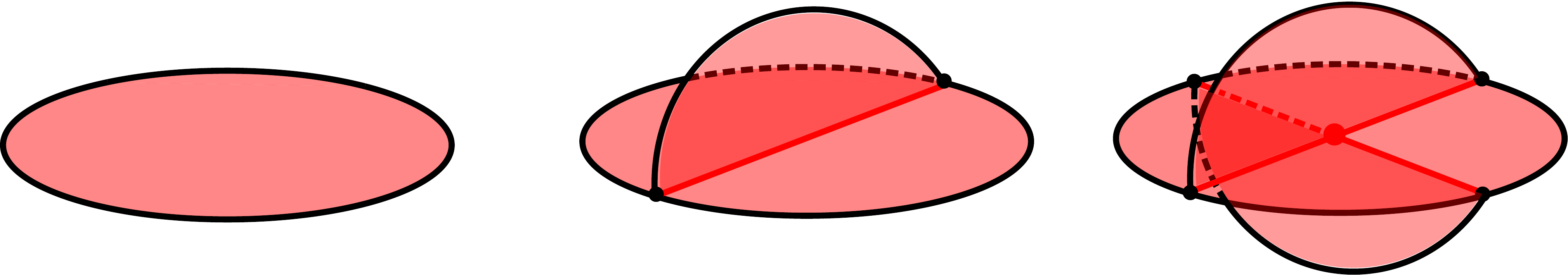}
\caption[Local models of a foam.]{\footnotesize In red, the local models of a foam $\Sigma$, seen as open cones over some graphs (drawn in black). From left to right, a neighbourhood in $\Sigma$ of a point in a 2-, 1-, and 0-stratum of $\Sigma$, respectively. Note that the third local model includes the other two. A foam in a 4-manifold is somehow the analogue of a trivalent graph in a 3-manifold.}
\label{fig: Sigma_local}
\end{figure}

The holonomy of a meridian $\gamma\in\pi_1(\mathcal X\smallsetminus\Sigma)$ of a 2-stratum of $\Sigma$ has a totally geodesic 2-plane as fixed point set. We have a rotation in $\Hyp^4$ of angle $\alpha_t$ when $t>0$, and a Lorentz boost in $\AdS^4$ of magnitude $\beta_t$ as $t<0$. In the half-pipe case, we have a transformation that can be interpreted as an infinitesimal rotation (resp. boost) in $\Hyp^4$ (resp. $\AdS^4$).

It is worth remarking that the cone-manifolds of Theorem \ref{teo: main} are non-compact, but of finite volume. (See \cite[Chapter 5]{surveyseppifillastre} and \cite{barbotfillastre} for the notion of volume in half-pipe geometry.) Nevertheless, the singularity $\Sigma$ is compact, or in other words, it does not enter into the ends of the cone-manifolds. These ends are (non-singular) \emph{cusps} in a suitable sense: while for hyperbolic manifolds this notion is well-established, we propose here an analogue definition for AdS and half-pipe manifolds (Definition \ref{defi cusp}). As a direct consequence of our methods, we achieve a nice description of the geometry of the cusps. A section of the cusps will indeed naturally support a geometric transition from Euclidean to Minkowski (non-singular) structures --- going through an intermediate geometry which is a ``flat version'' of half-pipe geometry and is the so-called \emph{Galilean geometry} \cite{yaglom}. The curious reader might want to have a preliminary look at Figure \ref{fig: cusp transition2} at page \pageref{fig: cusp transition2}.

Also the \emph{links} (the ``spheres of directions'') of the points of $\mathcal X$ naturally carry a geometric transition, which enlightens the structure of $\Sigma$ itself, and will be described in Figures \ref{fig:transition link 1}, \ref{fig:transition link 2} and \ref{fig:transition link 3} at page \pageref{fig:transition link 1}.

Finally, we remark that the statement of Theorem \ref{teo: main} can be made slightly more general by our methods, just assuming that $\mathcal N$ is a \emph{cuboctahedral manifold}, namely a hyperbolic manifold tessellated by ideal right-angled cuboctahedra. We however chose to keep the statement in this simpler version. See Remark \ref{rem: cuboctahedral_manifolds} for more details.

\subsection*{The proof: extending Kerkhoff and Storm's construction}

The essential ingredient for the proof of Theorem \ref{teo: main}
is a deforming 4-polytope ${\mathcal P}_t\subset\Hyp^4$ parametrised by $t\in(0,1]$, introduced by Kerckhoff and Storm \cite{KS}. For a particular choice of the 3-manifold $\mathcal N$, the hyperbolic cone structures $\sigma_t$ that degenerate were shown to exist by Martelli and the first author \cite[Theorem 1.2]{MR} by gluing eight copies of ${\mathcal P}_t$.

A fundamental property {of ${\mathcal P}_t$} is that most of its dihedral angles are right for all values of $t$, while the remaining dihedral angles are all equal and tend to $\pi$ as $t\to 0$, i.e. when ${\mathcal P}_t$ collapses to the aforementioned cuboctahedron. The presence of many right angles is essential in order to glue copies of $\mathcal P_t$ without creating a too complicated singular locus.

To prove Theorem \ref{teo: main}, we first show that the path of hyperbolic polytopes extends for negative times $t\in(-1,0)$ to a path of AdS polytopes with the same combinatorics of $\mathcal P_t\subset\Hyp^4$ with $t\in(0,\epsilon]$, and sharing similar properties on the dihedral angles and on the collapse. A remarkable difference is that, since the AdS metric is Lorentzian, some of the bounding hyperplanes are spacelike, and some others timelike.

The construction is however quite complicated and involves several computations. To prove that the combinatorics of the AdS polytopes remains constant, we needed to implement a \textsc{Sage} \cite{sagemath} worksheet to prove Lemma \ref{lem: vertices}. The proof of the analogous property on the hyperbolic side \cite{KS,MR} circumvented this amount of computations relying on Vinberg's theory of hyperbolic polytopes with non-obtuse dihedral angles. 

By opportunely rescaling ${\mathcal P}_t$ inside projective space along the direction of collapse, as suggested by the work of Danciger, we show that the resulting path of rescaled projective polytopes extends as $t=0$ to a half-pipe 4-polytope. This whole deformation can be interpreted as a geometric transition of ``cone-orbifold'' structures. More precisely, the subset
$${\mathcal P}_t^\times\subset{\mathcal P}_t$$
obtained by removing the ridges (the codimension-2 faces) with non-constant dihedral angles has a natural structure of hyperbolic (when $t>0$) or AdS (when $t<0$) orbifold. To show that these structures are linked by geometric transition, we construct an opportune half-pipe orbifold structure on the ``rescaled limit'' of $\mathcal P^\times_t$ as $t\to0$. 

Then, inspired by \cite{MR}, we glue several copies of $\mathcal P_t$ in the following way. Any $d$-sheeted orbifold cover $\mathcal N\to\mathcal C$ of the the ideal right-angled cuboctahedron naturally induces a way to pair certain facets of $d$ copies of $\mathcal P_t$. When $t<0$, these facets are precisely the timelike facets of the AdS polytope. The resulting space is homeomorphic to $\mathcal N\times[0,1]$, and its two boundary components contain all the ridges of the copies of $\mathcal P_t$ with non-constant dihedral angle. The final step is to double this manifold, thus obtaining $\mathcal X=\mathcal N\times S^1$ with a structure of hyperbolic, or AdS, cone-manifold. The singular locus $\Sigma$ consists of the union of the copies of the ridges with non-constant dihedral angle.

We would like to stress here a particular caveat of this construction. The fact that the polytope $\mathcal P_t$, suitably rescaled, converges when $t\to 0$ to a half-pipe polytope is \emph{not} sufficient to produce {a} half-pipe {orbifold} structure on the rescaled limit of $\mathcal P_t^\times$. Indeed, in contrast with the hyperbolic or AdS case, a hyperplane in half-pipe space does not uniquely determine a half-pipe reflection: there is a one-parameter family of reflections which fix a non-spacelike (i.e. \emph{degenerate}) hyperplane. This counterintuitive phenomenon, which often occurs in the realm of real projective geometry, highlights the fact that half-pipe geometry is neither Riemannian, nor pseudo-Riemannian. Hence finding the ``half-pipe glueings'' is somehow subtler, and will be achieved by analysing the behaviour of the holonomy representations of the hyperbolic and AdS structures infinitesimally, near the collapse.

This analysis of the holonomy representations ``nearby''  the collapse, which is important in our construction of half-pipe structures, is one of the motivations of our work \cite{transition_char_var}. In general, a half-pipe structure is never rigid, because one can always conjugate with a transformation which ``stretches'' the degenerate direction, and obtain a new structure equivalent to the initial one \emph{as a real projective structure}, but inequivalent \emph{as a half-pipe structure}. We discover \emph{a posteriori} in \cite{transition_char_var} that such ``stretchings'' are the only possible deformations of the HP orbifold structure we found, which is therefore essentially unique.

\subsection*{Organisation of the paper}

We first develop some tools which will be useful in the following. In Section \ref{section geometric transition} we recall the relevant notions of geometric structures and geometric transition in any dimension. We introduce AdS and HP cusps in Section \ref{sec:cusps}, hyperplanes, rotations and reflections in Section \ref{sec:half-spaces}, and cone-manifolds in Section \ref{sec:poly-conemfds}.

Then, we construct our examples of geometric transition. More precisely, in Section \ref{sec:warm-up} we study some examples in dimension three, which are of fundamental importance to the understanding of the four-dimensional construction. The latter is developed in Section \ref{sec 4dim}, which provides the proof of Theorem \ref{teo: main}.

\subsection*{Acknowledgments}

We are grateful to Francesco Bonsante and Bruno Martelli for interesting discussions, useful advices, and encouragement. We also thank Jeffrey Danciger and Gye-Seon Lee for interest in this work and related discussions. We owe to Fran\c{c}ois Fillastre and Ivan Izmestiev the observation that the transitional geometry from Euclidean to Minkowski geometry is called Galilean geometry.

We thank the mathematics departments of Pavia, Luxembourg and Neuch\^atel, for the warm hospitality during the respective visits while part of this work was done. The stage of this collaboration was set during the workshop ``Moduli spaces'', held in Ventotene in September 2017: we are grateful to the organisers for this opportunity.


\section{Geometric transition from $\Hyp^n$ to $\AdS^n$} \label{section geometric transition} \label{section geometric transition} \label{sec prel}

In this first part of the paper, we recall the relevant notions of geometric structures and geometric transition in any dimension, and we develop some tools which will be useful in the following.

We start by recalling the basic definitions concerning projective structures, in particular hyperbolic, anti-de Sitter and half-pipe structures, and geometric transition.

\subsection{(G,X)-structures}

Recall that, given a Lie group $G$ of analytic diffeomorphisms of a manifold $X$, a $(G,X)$-\emph{structure} $\mathscr P$ on a smooth manifold $\mathcal{M}$ consists of an atlas 
$$\mathscr A=\{\varphi^U\colon U\to X\,|\,U\in\mathscr U\}$$
where $\mathscr U$ is an open covering of $\mathcal{M}$, the maps $\varphi$ are diffeomorphisms onto their images, and the transition functions are restrictions of elements of $G$.

Let $\widetilde{\mathcal{M}}\to \mathcal{M}$ the universal covering. It is well-known that a $(G,X)$-structure on $\mathcal{M}$ is equivalent to the data of a \emph{developing map} 
$$\dev\colon\widetilde{\mathcal{M}}\to X~,$$
which is a local diffeomorphism, and a \emph{holonomy representation}
$$\rho\colon\pi_1\mathcal{M}\to G~,$$
satisfying the condition that $\dev$ is equivariant for the holonomy representation. The pair $(\dev,\rho)$ is well-defined up to the action of $G$ on such pairs, where $G$ is acting on local diffeomorphisms from $\widetilde{\mathcal{M}}$ to $X$ by post-composition and on $G$-valued representations by conjugation.

We say that a family $\mathscr P_t$ of $(G,X)$-structures on $\mathcal{M}$ is $C^k$ if it admits a family of pairs $(\dev_t,\rho_t)$ such that $t\mapsto\dev_t\in C^\infty(\widetilde{\mathcal{M}},X)$ is continuous  for the $C^k$-norm on any compact set of $\widetilde{\mathcal{M}}$, and {$t\mapsto\rho_t(\gamma)\in G$} is $C^k$ for every $\gamma\in\pi_1\mathcal M$.

\subsection{Real projective structures} \label{sec: proj structures}

In this paper, we are interested in \emph{real projective structures} on manifolds --- namely, structures locally modelled on the real projective space $\RP^n$. 

We denote by $\Aut(\RP^n)$ the group of projective transformations of $\RP^n$, which is identified to $\PGL_{n+1}(\R)$. 

\begin{defi}
A \emph{real projective structure} on an $n$-{manifold} $\mathcal{M}$ is an $(\Aut(\RP^n),\RP^n)$-structure.
A \emph{real projective manifold} is a manifold endowed with a real projective strucure.
\end{defi}

The goal of this paper is to produce families of real projective structures on a fixed smooth manifold. Our structures will be obtained by gluing several copies of a projective polytope. In general, (convex) polytopes are conveniently defined as the intersection of some half-spaces (see Section \ref{section polytopes}). Since in $\RP^n$ there is no notion of half-space, we will work with its double cover, namely the \emph{projective sphere}
$$\SP^n=\{x\in\R^{n+1}\smallsetminus\{0\}\}/\R_{>0}~,$$
where the group $\R_{>0}$ acts by multiplication.
We will always use the notation
$$x=(x_0,\ldots,x_n)\in\R^{n+1},\quad[x]=[x_0:\ldots:x_n]\in\SP^n.$$

The projective sphere $\SP^n$ is clearly endowed with a real projective structure induced by the double covering $\SP^n\to \RP^n$. 
We will denote by $\Aut(\SP^n)$ the group of projective automorphisms of $\SP^n$, which is
the double cover of $\Aut(\RP^n)$ induced by $\SP^n\to\RP^n$.

{An \emph{affine chart} is a subset of $\SP^n$ defined by an equation of the form $\alpha(x)>0$, for some nonzero linear form $\alpha\in\R^{n+1,*}$.} Throughout the paper, we will mostly consider the following affine chart:
$$\Aff^n=\{[x_0:\ldots:x_n]\in\SP^n\,|\,x_0>0\}~.$$
We will also denote the affine coordinates of $\Aff^n$ by 
$$\left(y_1,\ldots,y_n\right)=\left(\frac{x_1}{x_0},\ldots,\frac{x_n}{x_0}\right).$$

\begin{remark}
To be precise, in this paper we will construct families of $(\Aut(\SP^n),\SP^n)$-structures. Actually, for the structures we will construct, the restriction of the projection $\SP^n\to \RP^n$ on the image of the developing map will be injective. Thus the  $(\Aut(\SP^n),\SP^n)$-structures we will construct will be automatically equivalent to $(\Aut(\RP^n),\RP^n)$-structures.
\end{remark}

Our deformations of projective structures interpolate between \emph{hyperbolic} and \emph{anti-de Sitter} structures, going through \emph{half-pipe} structures. These are introduced in the following sections.

\subsection{Hyperbolic structures}

We introduce the hyperbolic $n$-space as follows:
$$\Hyp^n=\{[x]\in\SP^n\,|\,q_1(x)<0\,,\,x_0>0\}~,$$
where $q_1$ is the quadratic form
$$q_1(x)=-x_0^2+x_1^2+\ldots+x_n^2~.$$
Observe that $\Hyp^n$ is well-defined as a subset of $\SP^n$, since both conditions $$q_1(x)<0\qquad\text{and}\qquad x_0>0$$
are invariant under multiplication by a positive number. 
By construction, $\Hyp^n$ is contained in the affine chart $\Aff^n$ defined in Section \ref{sec: proj structures}, and is expressed in affine coordinates as the unit ball
$\{y_1^2+\ldots+y_n^2<1\}$.

The \emph{boundary at infinity} of $\Hyp^n$ is its topological frontier:
$$\partial\Hyp^n=\{[x]\in\SP^n\,|\,q_1(x)=0\,,\,x_0>0\}~,$$
which in affine coordinates is the sphere $\{y_1^2+\ldots+y_n^2=1\}$.

\begin{remark}
The symbol $\partial$ denotes the topological frontier of a subset of $\SP^n$. (Sometimes, at the end of the paper, it also denotes the boundary of an abstract manifold, but there will not be risk of ambiguity.)
\end{remark}

It is well-known that $\Hyp^n$ carries a Riemannian metric of constant sectional curvature $-1$, which is obtained by pulling-back the standard bilinear form $b_1$ of signature $(-,+,\ldots,+)$ on $\R^{n+1}$ (whose associated quadratic form is $q_1$) via the immersion $\sigma\colon\Hyp^n\to \R^{n+1}$ which maps the class $[x]$ to the unique positive multiple of $x$ such that $q_1(x)=-1$.

It then turns out that the group $\Isom(\Hyp^n)$ of isometries of $\Hyp^n$, endowed with the Riemannian metric $\sigma^*b_1$ as above, coincides with the subgroup of $\Aut(\SP^n)$ which preserves $\Hyp^n\subset\SP^n$. The group $\Isom(\Hyp^n)$ is also identified to an index two subgroup of $\O(q_1)$, the group of linear isometries of the quadratic form $q_1$. In conclusion, we have the following definition:

\begin{defi}
A \emph{hyperbolic structure} on an $n$-{dimensional} manifold $\mathcal{M}$ is an $(\Isom(\Hyp^n),\Hyp^n)$-structure.
\end{defi}

As a consequence of the above discussion, a hyperbolic structure on $\mathcal{M}$ is a particular case of real projective structure, as we can consider it as a $\SP^n$-valued atlas with transition functions in $\Aut(\SP^n)$.

\subsection{Anti-de Sitter structures}

Let us now introduce the anti-de Sitter $n$-space, in a somewhat parallel way to $\Hyp^n$. We define it as:
$$\AdS^n=\{[x]\in\SP^n\,|\,q_{-1}(x)<0\}~,$$
where now $q_{-1}$ is the quadratic form of signature $(-,+,\ldots,+,-)$:
$$q_{-1}(x)=-x_0^2+x_1^2+\ldots+x_{n-1}^2-x_n^2~.$$
The \emph{boundary at infinity} of $\AdS^n$ is then naturally defined as 
$$\partial\AdS^n=\{[x]\in\SP^n\,|\,q_{-1}(x)=0\}~.$$

\begin{remark} \label{rem: ads double covering}
Anti-de Sitter space is more often defined as the image of what we defined $\AdS^n$ through the double covering $\SP^n\to\RP^n$. 
Nevertheless, the polytopes we will construct are contained in the affine chart $\Aff^n=\{x_0>0\}$ (although $\AdS^n$ is not), hence this choice will make no substantial difference with the more frequent definition.
\end{remark}

As already said, $\AdS^n$ is not contained in a single affine chart. However, we can easily describe its intersection with 
$\Aff^n$ as the internal region of the one-sheeted hyperboloid
$$\AdS^n\cap\Aff^n=\{y_1^2+\ldots+y_{n-1}^2-y_n^2<1\},$$
while $\partial\AdS^n\cap\Aff^n$ is the one-sheeted hyperboloid $\{y_1^2+\ldots+y_{n-1}^2-y_n^2=1\}$ itself.

Similarly to $\Hyp^n$, the space $\AdS^n$ is endowed with a metric, which is now Lorentzian, obtained as the pull-back the standard bilinear form $b_{-1}$ of signature $(-,+,\ldots,+,-)$ on $\R^{n+1}$ by the immersion $\sigma\colon\AdS^n\to \R^{n+1}$ mapping the class $[x]$ to the unique positive multiple of $x$ such that $q_{-1}(x)=-1$. Again, the group $\Isom(\AdS^n)$ of isometries of $\AdS^n$ coincides with the subgroup of $\Aut(\SP^n)$ preserving $\AdS^n$, and is identified to $\O(q_{-1})$.

We give the following definition of anti-de Sitter structure, which will be another particular case of projective structure:

\begin{defi}
An \emph{anti-de Sitter} (or \emph{AdS}) \emph{structure} on an $n$-{dimensional} manifold $\mathcal{M}$ is an $(\Isom(\AdS^n),\AdS^n)$-structure.
\end{defi}

\subsection{Half-pipe structures}\label{subsec:hp structures}

In \cite{danciger}, Danciger introduced half-pipe geometry as a limit (in the sense of \cite{CDW}; see Section \ref{sec:rescaled_limit}) of both hyperbolic and anti-de Sitter geometries inside real projective geometry. Its definition is the following. Let us denote by $q_0$ the degenerate quadratic form on $\R^{n+1}$:
$$q_{0}(x)=-x_0^2+x_1^2+\ldots+x_{n-1}^2~.$$
Then we define
$$\HP^n=\{[x]\in\SP^n\,|\,q_{0}(x)<0\,,\,x_0\geq 0\}~,$$
This is again well-defined by homogeneity of the two conditions, and the \emph{boundary at infinity} of half-pipe space is:
$$\partial\HP^n=\{[x]\in\SP^n\,|\,q_{0}(x)=0\,,\,x_0\geq 0\}~.$$
By construction, $\HP^n$ is contained in the affine chart $\Aff^n=\{x_0> 0\}$, where it is represented as a solid cylinder, defined by the equation $y_1^2+\ldots+y_{n-1}^2<1$ in affine coordinates. Its boundary at infinity is topologically a sphere, consisting of the frontier of the solid cylinder in $\Aff^n$ and two additional points at infinity. 

In analogy with the hyperbolic and anti-de Sitter construction, we can introduce a degenerate metric on $\HP^n$ by means of the embedding of $\sigma\colon\HP^n\to\R^{n+1}$ sending $[x]$ to the unique positive multiple of $x$ on which $q_0$ takes value $-1$. Then one pulls-back the degenerate bilinear form $b_0$ of signature $(-,+,\ldots,+,0)$. The symmetric 2-tensor $\sigma^*b_0$ obtained in this way actually corresponds to the splitting $\HP^n=\Hyp^{n-1}\times\R$, where $\sigma^*b_0$ coincides with the hyperbolic metric when restricted to the first factor, and is zero whenever one of the two arguments is in the $\R$ factor.

One would be tempted to define the
transformation
group of $\HP^n$ as the group $\Aut(\HP^n)<\Aut(\SP^n)$ of projective transformations that preserve $\HP^n\subset\SP^n$.
However,
we are interested in a more rigid 
geometry,
which will be the 
limit of both hyperbolic and anti-de Sitter geometry.
One then \emph{defines} the
group of half-pipe transformations as:
$$G_{\HP^n}=\{A\in\O(q_0)\,|\,A(e_n)=\pm e_n\,,\,(A(e_0))_0>0\}~.$$
Here we used $e_0,\ldots,e_n$ to denote the standard basis of $\R^{n+1}$. 
The last condition means that the first coordinate of $A(e_0)$ in the standard basis is positive. Together with the  fact that $A\in\O(q_0)$, this implies that $A$ preserves $\HP^n\subset\SP^n$. 

As a consequence of the definition, one sees that elements of $G_{\HP^n}$ are of the form:
\begin{equation} \label{eq form isometries hp}
A=\left[
\begin{array}{ccc|c}
  &&& 0 \\
  
   & \widehat A & & \vdots \\
  &&& 0 \\
    \hline  
    \star&\ldots&\star   & \pm1
\end{array}
\right]~,
\end{equation}
where $\widehat A$ is a linear isometry for the quadratic form of signature $(-,+,\ldots,+)$, and the stars denote the entries of any vector in $\R^n$. The square brackets denote the projective class of a matrix in $\GL_{n+1}(\R)/\R_{>0}$. 

\begin{remark}
In contrast with $\Hyp^n$ and $\AdS^n$ (where in place of the strict inclusions there are equalities), we have
$$G_{\HP^n}\subsetneq\Aut(\HP^n)\subsetneq\Isom(\HP^n),$$
where by $\Isom(\HP^n)$ we denote the group of self-homeomorphisms of $\HP^n$ which preserve the degenerate symmetric 2-tensor $\sigma^*b_0$.
Indeed, from \eqref{eq form isometries hp}, the condition that $e_n$ is eigenvector with eigenvalue $\pm1$ implies that $A$ cannot ``stretch'' in the degenerate direction. Moreover, the group $\Isom(\HP^n)$ is infinite-dimensional, and so it cannot even embed into $\Aut(\SP^n)$. 
\end{remark}
%

This finally enables us to provide the definition of half-pipe structures, which is the third special type of projective structures of our interest:

\begin{defi}
A \emph{half-pipe} (or \emph{HP}) \emph{structure} on an $n$-{manifold} $\mathcal{M}$ is a $(G_{\HP^n},\HP^n)$-structure.
\end{defi}

\subsection{Relation with Minkowski geometry} \label{sec motivations hp}
There are two main motivations behind this definition of half-pipe geometry. One motivation
is that half-pipe geometry is transitional from hyperbolic to anti-de Sitter geometry, as explained in detail in Section \ref{sec:rescaled_limit} below.

The other motivation comes from the fact that $\HP^n$ is naturally the dual space of Minkowski space $\R^{1,n-1}$, which is the vector space $\R^{n}$ endowed with the quadratic form 
$$\widehat q(\widehat x)=-x_0^2+x_1^2+\ldots+x_{n-1}^2~,$$
where $\widehat x=(x_0,\ldots,x_{n-1})$. 
Indeed, if $\widehat b$ denotes the bilinear form of $\R^{1,n-1}$ whose associated quadratic form is $\widehat q$, then  any spacelike hyperplane in $\R^{1,n-1}$ writes as 
\begin{equation} \label{eq in remark motivation}
\{p\in\R^{1,n-1}\,|\,\widehat b(p,\widehat x)=a\}~,
\end{equation}
where $\widehat x$ is a future-directed normal vector to the hyperplane, hence satisfying 
$$\widehat q(\widehat x)=-x_0^2+x_1^2+\ldots+x_{n-1}^2<0\qquad\text{and}\qquad x_0>0~,$$
and $a\in\R$. This means that the class of the pair $(\widehat x,a)$ belongs to $\HP^n$.
Moreover, two pairs $(\widehat x,a)$ and $(\widehat x',a')$ determine the same spacelike hyperplane if and only if they differ by multiplication by a positive number. 

In conclusion, $\HP^n$ parameterises the spacelike hyperplanes in $\R^{1,n-1}$. Similarly, $\partial\HP^n$ consists of a cylinder (homeomorphic to $S^{n-2}\times\R$) which naturally parametrises lightlike hyperplanes of $\R^{1,n-1}$, plus two additional points at infinity. Moreover, we have:

%
\begin{lemma} \label{lemma isomorphism halfpipe and minkowski}
The action of $\Isom(\R^{1,n-1})$ on the set of spacelike hyperplanes of $\R^{1,n-1}$ induces a group isomorphism between $\Isom(\R^{1,n-1})$ and $G_{\HP^n}$.
\end{lemma}

Although this fact has already been observed, for instance in \cite{surveyseppifillastre} and \cite{barbotfillastre}, we provide a complete proof since the explicit computation of the isomorphism will be useful in the remainder of the paper.
\begin{proof}
Our purpose is to construct a group isomorphism 
$$\phi\colon\Isom(\R^{1,n-1})\to G_{\HP^n}~.$$
Let us first check that it is well-defined, i.e. its image is actually composed of elements of $G_{\HP^n}$. It will then be obvious from the definition that $\phi$ is a group homomorphism, and that it is injective, since clearly only the identity element of $\Isom(\R^{1,n-1})$ fixes all spacelike hyperplanes.

For this purpose, we denote by $(\widehat A,v)$  an isometry of $\R^{1,n-1}$, of the form
 $$p\mapsto \widehat A\cdot p+v~,$$ 
 for $\widehat A\in \O(\widehat q)$, and let us compute its action on $\HP^n$. We need to distinguish two cases. If $\widehat A$ is future-preserving, namely $(\widehat A(e_0))_0>0$, then
 the hyperplane parameterised (up to positive multiples) by $(\widehat x,a)$, namely 
 $$P=\{p\in\R^{1,n-1}\,|\,\widehat b(p,\widehat x)=a\}$$
   is mapped to the hyperplane 
 $$\widehat A\cdot P+v=\{q\in\R^{1,n-1}\,|\,\widehat b(q,\widehat A\cdot\widehat x)=a+\widehat b(v,\widehat A\cdot \widehat x)\}~,$$
 which is  
 parameterised by $(\widehat A\cdot \widehat x, a+\widehat b(v,\widehat A\cdot \widehat x))$. From \eqref{eq form isometries hp}, this shows that $(\widehat A,v)$ corresponds to the following element of $G_{\HP^n}$:
 $$\phi(\widehat A,v)=\left[
\begin{array}{ccc|c}
  &&& 0 \\
  
   & \widehat A & & \vdots \\
  &&& 0 \\
    \hline  
   \ldots &v^TJ\widehat A & \ldots  & 1
\end{array}
\right]~,$$
 where $J=\mathrm{diag}(-1,1,\ldots,1)$ is the matrix which represents the bilinear form $\widehat b$.
 Similarly, one checks that the induced action of $(-\widehat A,v)$, with $(\widehat A(e_0))_0>0$, gives the following element of 
 $G_{\HP^n}$:
 $$\phi(-\widehat A,v)=\left[
\begin{array}{ccc|c}
  &&& 0 \\
  
   & \widehat A & & \vdots \\
  &&& 0 \\
    \hline  
   \ldots &v^TJ\widehat A & \ldots  & -1
\end{array}
\right]~.$$
It thus follows from \eqref{eq form isometries hp} that $\phi$ is well-defined and surjective, and this concludes the proof.
\end{proof}

%

\subsection{Rescaled limits and geometric transition} \label{sec:rescaled_limit}
Let us consider the family $\r_t\in \Aut(\SP^n)$, depending on the real parameter $t\neq0$, defined by:
$$\r_t=\left[\mathrm{diag}\left(1,\ldots,1,\frac{1}{t}\right)\right]\in \GL_{n+1}(\R)/\R_{>0}~.$$
Let us denote by $q_t$ the quadratic form:
$$q_t(x)=-x_0^2+x_1^2+\ldots+x_{n-1}^2+\mathrm{sign}(t)t^2x_n^2~.$$
(Observe that the notation is consistent with the definitions of $q_{-1},q_0,q_1$ in the previous sections.) Then it follows that, for $t>0$, $\r_{t}(\Hyp^n)$ is the domain $\X_t^n$ in $\SP^n$ defined as follows:
$$\X_t^n=\{[x]\in\SP^n\,|\,q_t(x)<0\,,\,x_0>0\}~.$$
In particular, $\X_1^n=\Hyp^n$. Moreover, if we endow $\X_t^n$ with a Riemannian metric induced by the quadratic form $q_t$ as we did for $\Hyp^n$, then $\r_t$ is an isometry between $\Hyp^n$ and $\X_t^n$, and therefore conjugates the isometry group of $\Hyp^n$ to the isometry group of $\X_t^n$:
$$\Isom(\X_t^n)=\r_t\Isom(\Hyp^n)\r_t^{-1}~.$$ The following lemma says that half-pipe geometry is a ``limit'' of hyperbolic geometry:

\begin{lemma}[\cite{CDW,surveyseppifillastre}] \label{lemma convergence hyp}
When $t\to 0^+$, the closure $\overline{\X_t^n}$ converges to $\overline{\HP^n}$ in the Hausdorff topology of $\SP^n$, and the groups $\Isom(\X_t^n)$ converge to $G_{\HP^n}$ in the Chabauty topology on closed subgroups of $\Aut(\SP^n)$.
\end{lemma}

\begin{figure}
\centering
\begin{minipage}[c]{.2\textwidth}
\centering
\includegraphics[scale=0.32]{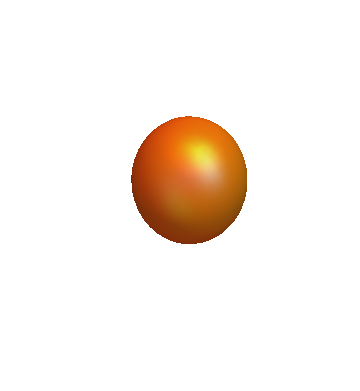}
\end{minipage}%
\begin{minipage}[c]{.2\textwidth}
\centering
\includegraphics[scale=0.3]{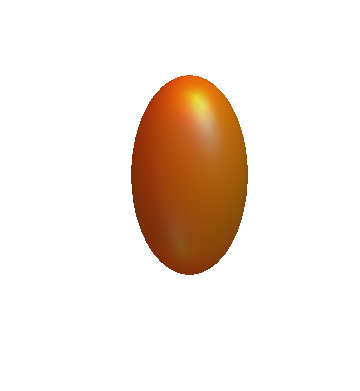}
\end{minipage}%
\begin{minipage}[c]{.2\textwidth}
\centering
\includegraphics[scale=0.28]{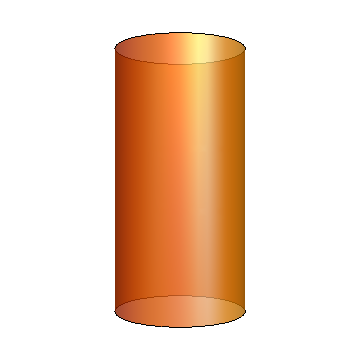}
\end{minipage}%
\begin{minipage}[c]{.2\textwidth}
\centering
\includegraphics[scale=0.27]{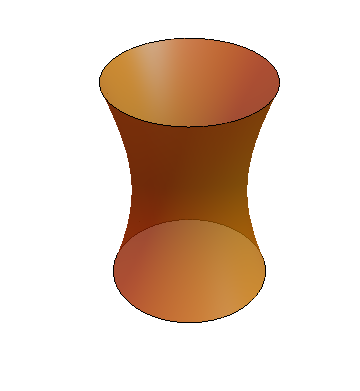}
\end{minipage}%
\begin{minipage}[c]{.2\textwidth}
\centering
\includegraphics[scale=0.25]{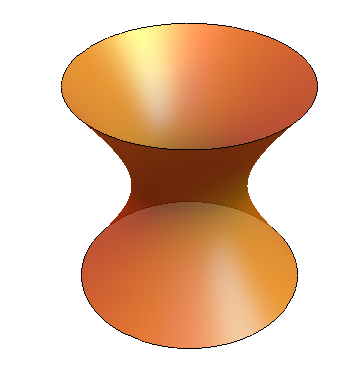}
\end{minipage}
\caption
[Transition from $\Hyp^n$ to $\AdS^n$ in an affine chart.]{\footnotesize The transition from the ball model of $\Hyp^n$ to the hyperboloid model of $\AdS^n$, in an affine chart (for $n=3$).} 
\label{fig:transition}
\end{figure}

If we project to $\RP^n$, the situation is exactly the same for $t<0$. However, in our setting there is a small difference due to the fact that we defined $\AdS^n$ as the double cover of what is usually defined as anti-de Sitter space inside $\RP^n$ (recall Remark \ref{rem: ads double covering}). In particular, $\AdS^n$ is invariant by the antipodal map
$a=[\mathrm{diag}(-1,\ldots,-1)]\in\Aut(\SP^n)$
(which is in the centre of $\Aut(\SP^n)$), while $\HP^n$ is contained in the affine chart $\{x_0>0\}$ and is therefore clearly not invariant by
$a$.
Nevertheless, if we define, for $t<0$:
 $$\X_t^n=\{[x]\in\SP^n\,|\,q_t(x)<0\}~,$$
 then again $\r_{|t|}$ defines an isometry between $\AdS^n$ and $\X_t^n$ endowed with the Lorentzian metric induced by $q_t$, so that
 $$\Isom(\X_t^n)=\r_{|t|}\Isom(\AdS^n)\r_{|t|}^{-1}~.$$
Similarly to the hyperbolic case, we have:
 
\begin{lemma}[\cite{CDW,surveyseppifillastre}] \label{lemma convergence ads}
When $t\to 0^-$, $\overline{\X_t^n\cap \Aff^n}$ converges to $\overline{\HP^n}$, and the groups $\Isom(\X_t^n)$ converge to
a central ${\Z/2\Z}$-extension of $G_{\HP^n}$ by means of the antipodal map.
\end{lemma}

See also Figure \ref{fig:transition}. Motivated by this construction, we have the following definition of \emph{geometric transition}:

\begin{defi}{\cite[Definition 3.8]{dancigertransition}} \label{defi transition}
Given an $n$-dimensional manifold $\mathcal{M}$, a \emph{geometric transition} on $\mathcal{M}$ from hyperbolic to anti-de Sitter structures is a continuous path $\{\mathscr P_t\}_{t\in(-\epsilon,\epsilon)}$ of real projective structures on $\mathcal{M}$ such that $\mathscr P_t$ is conjugate to a hyperbolic structure for $t>0$, to a half-pipe structure for $t=0$, and to an anti-de Sitter structure for $t<0$.
\end{defi}

In fact, the geometric transitions we construct in this paper will be $C^1$ deformations of geometric structures.

\begin{remark}
Theorem \ref{teo: main} shows that there exists a cusped hyperbolic 3-manifold $\mathcal{N}$, a foam $\Sigma\subset\mathcal X=\mathcal{N}\times S^1$, and a geometric transition on the 4-manifold 
$\mathcal{M}=\mathcal X\smallsetminus\Sigma$ (see Corollary \ref{cor: transition manifold}).
As in \cite{danciger,dancigertransition}, the geometric structures of $\mathcal M$ extend to 
$\mathcal X$ with some special kinds of singularities along $\Sigma$. Such singular structures will be described in Section \ref{sec cone-mfds}.
\end{remark}

\subsection{A recipe to construct examples} \label{sec recipe}
A direct way to construct examples of geometric transition from hyperbolic to anti-de Sitter structures, which is essentially the strategy we will use in this paper, is the following. Observe that there is an isometric embedding $\iota$ of $\Hyp^{n-1}$ into both $\Hyp^n$ and $\AdS^n$, which is given by:
\begin{equation}\label{eq: iota}
\iota([x_0:\ldots:x_{n-1}])=[x_0:\ldots:x_{n-1}:0]~.
\end{equation}
Hence $\iota(\Hyp^{n-1})$ is a totally geodesic hyperplane in $\Hyp^n$ (resp. $\AdS^n$), whose image is $\Hyp^n\cap\{x_n=0\}$ (resp. $\AdS^n\cap\{x_n=0\}\cap \{x_0>0\}$), and the embedding $\iota$ extends to an embedding of $\partial\Hyp^{n-1}$ into $\partial\Hyp^{n}$ (resp. $\partial\AdS^{n}$). The same formula \eqref{eq: iota} defines also a copy of $\Hyp^{n-1}$ inside $\HP^n$.

In fact, observe that the subgroup 
\begin{equation} \label{eq:subgroup G0}
G_0=\left\{\left[
\begin{array}{ccc|c}
  &&& 0 \\
  
   & \widehat A & & \vdots \\
  &&& 0 \\
    \hline  
    0&\ldots&0   & \pm1
\end{array}
\right]\bigg|\ \widehat A\in\O(\widehat q), (\widehat A(e_0))_0>0\right\}~,
\end{equation}
is simultaneously a subgroup of $\Isom(\Hyp^n)$, $\Isom(\AdS^n)$ and $G_{\HP^n}$ in $\Aut(\SP^n)$, composed precisely of those elements of $\Isom(\Hyp^n)$, $\Isom(\AdS^n)$ and $G_{\HP^n}$ which preserve the image of $\iota$. The group $G_0$ is isomorphic to $\Isom(\Hyp^{n-1})\times {\Z/2\Z}$. The reflection $r=\mathrm{diag}(1,\ldots,1,-1)$, along the hyperplane $\iota(\Hyp^{n-1})$ of $\Hyp^n$, $\AdS^n$, or $\HP^n$ is indeed central. The group $G_0$ is also isomorphic to $\O(\widehat q)$, the isomorphism being given by
\begin{equation} \label{eq:isomorphism G0}
r^{i}\begin{bmatrix} \widehat A & 0 \\ 0 & 1\end{bmatrix}\in G_0\mapsto (-1)^i\widehat A\in\O(\widehat q)
\end{equation}
for $i=0,1$. We will sometimes implicitly use this isomorphism in the paper.

With these premises, one then constructs a continuous family of projective structures $\mathscr P_t$ on a manifold $\mathcal{M}$, with pairs developing map-holonomy $(\dev_t,\rho_t)$ such that:
\begin{itemize}
\item For $t>0$, $\dev_t$ takes values in $\Hyp^n$ and $\rho_t$ in $\Isom(\Hyp^n)$;
\item For $t<0$, $\dev_t$ takes values in $\AdS^n$ and $\rho_t$ in $\Isom(\AdS^n)$;
\item When $t\to 0^\pm$, $\dev_t$ converges to a submersion $d_0$ with image in $\iota(\Hyp^{n-1})$, which is $h_0$-equivariant for a representation $h_0$ of $\pi_1 \mathcal{M}$ into the subgroup $G_0$ preserving $\iota(\Hyp^{n-1})$.
\end{itemize}

Then by applying the projective transformations $\r_{|t|}$, the pair $(\r_{|t|}\circ\dev_t,\r_{|t|}\rho_t\r_{|t|}^{-1})$ determines
a path of projective structures which, by construction, are conjugate to a hyperbolic structure when $t>0$, and to an anti-de Sitter structure when $t<0$.
If the pair $(\r_{|t|}\circ\dev_t,\r_{|t|}\rho_t\r_{|t|}^{-1})$ converges to a pair developing map-holonomy $(\dev_0,\rho_0)$, then, as a consequence of Lemma \ref{lemma convergence hyp} and Lemma \ref{lemma convergence ads}, this will determine a half-pipe structure on $\mathcal{M}$.


\section{Geometry of the cusps} \label{sec:cusps}

Since the cone-manifolds of Theorem \ref{teo: main} are cusped, in this section we introduce the notion of cusp in AdS and half-pipe manifolds. 

\subsection{Horospheres} \label{sec: horospheres}

Let us start with the notion of horosphere in the three geometries of our interest.
 
\begin{defi}
A \emph{horosphere} in $\Hyp^n$ is a smooth surface $H\subset\Hyp^n$ which is orthogonal to all the geodesics with the same endpoint $p\in \partial\Hyp^n$. A \emph{horosphere} in $\AdS^n$ is a smooth timelike surface $H\subset\AdS^n$ which is orthogonal to all the spacelike geodesics with the same endpoint $p\in \partial\AdS^n$.
\end{defi}

Since there is no notion of orthogonality in $\HP^n$, the half-pipe notion is slightly different. 

\begin{defi}
A \emph{horosphere} in $\HP^n$ is the union of all the degenerate lines going through a hyperbolic horosphere $\widehat{H}$ contained in a spacelike hyperplane.
\end{defi}

\begin{figure}
\centering
\begin{minipage}[c]{.33\textwidth}
\centering
\includegraphics[scale=0.5]{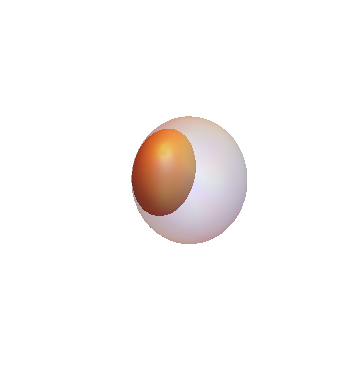}
\end{minipage}%
\begin{minipage}[c]{.25\textwidth}
\centering
\includegraphics[scale=0.45]{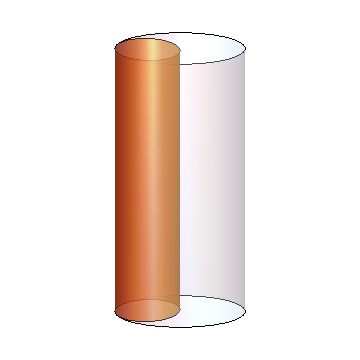}
\end{minipage}%
\begin{minipage}[c]{.5\textwidth}
\centering
\includegraphics[scale=0.42]{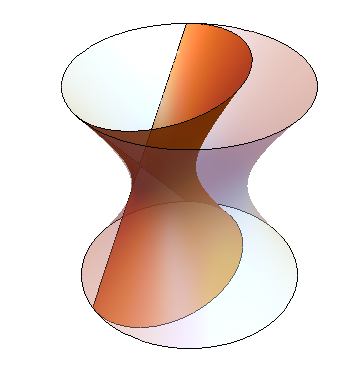}
\end{minipage}
\caption[Hyperbolic, half-pipe, and AdS horospheres.]{\footnotesize Hyperbolic, half-pipe, and anti-de Sitter horospheres in an affine chart.} 
\label{fig:horospheres}
\end{figure}

See Figure \ref{fig:horospheres} to visualise the horospheres in the affine models of $\Hyp^n$, $\HP^n$ and $\AdS^n$.
In each of the three geometries, we call \emph{boundary at infinity} of a horosphere $H$ {the set}
$$\partial_\infty H=\overline H\smallsetminus H$$ 
(here $\overline H$ denotes the closure of $H$ in $\SP^n$), which consists of a single point in $\partial\Hyp^n$ for the hyperbolic case, a pair of antipodal points in $\partial\AdS^n$ for the AdS case, and of a closed interval in $\partial\HP^n$ for the half-pipe geometry. In the latter case, the endpoints of such interval are the projectivisation in $\SP^n$ of the degenerate direction.

In fact, in terms of the duality with $\R^{1,n-1}$, a horosphere in $\HP^n$ can be described as the space of all spacelike hyperplanes in $\R^{1,n-1}$ whose normal vector (which is a point of $\Hyp^{n-1}$) lies in a horosphere of dimension $n-2$. Hence the boundary at infinity of a half-pipe horosphere consists of a degenerate (vertical) line, which corresponds to all lightlike hyperplanes in $\R^{1,n-1}$ containing the same lightlike direction, plus the two additional points which lie outside the affine chart $\Aff^n$.

\subsection{Metric expressions and upper half-space models} \label{sec transition horospheres}\label{sec upper halfspace}
Let us now give a parameterisation of horospheres and recover their Euclidean, Minkowski, or Galilean geometry. Recall that we defined 
$$\X^n_t=\{[x]\in\SP^n\,|\,q_t(x)<0\}~,$$
where 
$$q_t(x)=-x_0^2+x_1^2+\ldots+x_{n-1}^2+t|t|x_n^2~.$$
For $t\neq 0$, $\X_t^n$ is endowed with a pseudo-Riemannian metric (Riemannian for $t>0$ and Lorentzian for $t<0$) of constant curvature $-1$, given by pulling back the bilinear form $b_t$ on $\R^{n+1}$ associated to $q_t$ by the embedding that sends $[x]$ to the unique positive multiple of $x$ on which $q_t$ takes the value $-1$. Let us thus consider the embedding 
$$\eta_t\colon\R^{n-1}\to\R^{n+1}$$
given by 
$$\eta_t(y_2,\ldots,y_n)=(f_t(y_2,\ldots,y_n)+1,f_t(y_2,\ldots,y_n),y_2,\ldots,y_n)~,$$
where the function $f_t\colon\R^{n-1}\to\R$ is given by
$$f_t(y_2,\ldots,y_n)=\frac{1}{2}(y_2^2+\ldots+y_{n-1}^2+t|t|y_n^2)~.$$
 (Note that $f_t$ is determined by the condition $q_t\circ\eta_t\equiv -1$.)
By pulling back the bilinear form $b_t=-dx_0^2+dx_1^2+\ldots+t|t|dx_n^2$, we obtain
$$\eta_t^*b_t=dy_2^2+\ldots+dy_{n-1}^2+t|t|dy_n^2~.$$
In particular we obtained:
\begin{itemize}
\item for $t=1$, a horosphere in $\Hyp^n$ is isometric to Euclidean space $\R^{n-1}$;
\item for $t=-1$ a horosphere in $\AdS^n$ is isometric to Minkowski space $\R^{1,n-2}$;
\item for $t=0$ a horosphere in $\HP^n$ is isometric to $\R^{n-1}$, endowed with a degenerate metric of signature $+\ldots+0$ (the pull-back of the degenerate metric of $\HP^n$).
\end{itemize}
In all cases,
we have $$p=[1:1:0:\ldots:0]\in\partial_\infty\,\eta_t(\R^{n-1}).$$

In fact, applying the projective transformation $\r_t$ one sees that the half-pipe horosphere which is the image of $\eta_0$ is the rescaled limit of hyperbolic and AdS horospheres.


With a little more effort, we can use the embeddings $\eta_t$ to obtain an upper half-plane model for the spaces $\X_t$. Let us define a  parameterization $\zeta_t\colon\R_{>0}\times\R^{n-1}\to\X_t^n$ (this is only a local parameterization if $t<0$, see Remark \ref{rmk upper halfspace ads} below):
\begin{equation}\label{eq parameterization halfspace model}
\zeta_t(y_1,\ldots,y_n)=\left(
\begin{array}{ccccc}
  \frac{1}{2}\left(y_1+\frac{1}{y_1}\right) &\frac{1}{2}\left(y_1-\frac{1}{y_1}\right) &0&\ldots&  0 \\
  \frac{1}{2}\left(y_1-\frac{1}{y_1}\right) &\frac{1}{2}\left(y_1+\frac{1}{y_1}\right) &0&\ldots&  0 \\

 0 & 0 & 1 &  & \vdots \\

 \vdots &  && \ddots  &  \\

  0 &\ldots & \ldots & &1
\end{array}
\right)\eta_t\! \left(\frac{y_2}{y_1},\ldots,\frac{y_n}{y_1}\right)~.
\end{equation}
A tedious but elementary computation shows that 
$$\zeta_t^*b_t=\frac{dy_1^2+\ldots+dy_{n-1}^2+t|t|dy_n^2}{y_1^2}~.$$

\begin{remark}
To explain how the expression \eqref{eq parameterization halfspace model} is obtained, let us observe that
 the big matrix in Equation \eqref{eq parameterization halfspace model} is an isometry for every $\X_t^n$, translating along the geodesic $x_2=\ldots=x_n=0$, which is orthogonal to the horosphere parameterised by $\eta_t$.
\end{remark}

Hence $\zeta_t$ provides a upper half-space model for $\Hyp^n$ ($t=1$), $\AdS^n$ ($t=-1$) and $\HP^n$ ($t=0$). It is moreover evident that the multiplication of the last coordinate by $|t|$ provides an isometry between the upper half-space model for $\X_t^n$ and $\Hyp^n$ (if $t>0$) and for $\X_t^n$ and $\AdS^n$ (if $t<0$). The (spacelike, for the AdS and HP case) geodesics with endpoint at infinity $p=[1:1:0:\ldots:0]$ are represented by vertical lines, as expected.

\begin{remark} \label{rmk upper halfspace ads}
The upper half-space model for $\AdS^3$ has been described in \cite[Appendix A]{danciger}, although obtained in a different way. We remark here that, with our definition of $\AdS^n$, the upper half-space model only covers a part of $\AdS^n$ --- roughly speaking, half of $\AdS^n$. A little trick to visualise a larger portion is to allow $y_1\in\R\smallsetminus \{0\}$ in the parameterization $\zeta_{-1}$. In this way, one gets a parameterization of the complement of the lightlike hyperplane $\{x_0=x_1\}$ as the union of the upper half-space and the lower half-space. However, the boundary at infinity is somewhat more complicated to describe in this model. 
\end{remark}

\subsection{Cusps}
From the upper half-space models we constructed, we see that 
every isometry of $\Hyp^n$ or $\AdS^n$ which preserves a horosphere $H$ acts on $H$ by isometries for 
its intrinsic (Euclidean or Minkowski) metric. Conversely, every 
isometry of $H$ extends uniquely to a global isometry. 

{Given a horosphere $H$ of $\mathbb{X}^n_t$, we thus define:}
 \begin{itemize}
\item for $t=1$, the subgroup $P_1:=\mathrm{Stab}_{\Hyp^n}(H)\cong\Isom(\R^{n-1})$ of $\Isom(\Hyp^n)$; 
\item for $t=-1$, the subgroup $P_{-1}:=\mathrm{Stab}_{\AdS^n}\cong\Isom(\R^{1,n-2})$ of $\Isom(\AdS^n)$; 
\item for $t=0$, the subgroup $P_{0}:=\mathrm{Stab}_{\HP^n}(H)\cap \mathrm{Stab}_{\HP^n}(p)$ of $G_{\HP^n}$, 
for some $p\in\partial_\infty H$ that is not a vertex of the closed interval $\partial_\infty H$.
\end{itemize}
The third point needs some explication. First, recall that the boundary at infinity of a half-pipe horosphere does not consist of a single point, but of a closed interval. It is for this reason that we need to specify that $P_0$ is the stabiliser of \emph{both} $H$ and $p\in\partial_\infty H$. Moreover, the condition that $p$ is an interior point of $\partial_\infty H$ means that $p$ is not one of the two points which lie outside the affine chart $\Aff^n$ ({the two points} at infinity in Figure \ref{fig:horospheres}). In other words, in the usual duality with Minkowski space, $p$ corresponds to a lightlike hyperplane in $\R^{1,n-1}$.

{The geometry $(G,X)=(P_0,H)$ of a h}alf-pipe horosphere 
$H$ is called \emph{Galilean geometry} \cite{yaglom}. 
It can be checked that $P_0$ is the limit of $\r_{|t|}P_1\r_{|t|}^{-1}$ and $\r_{|t|}P_{-1}\r_{|t|}^{-1}$. In other words, Galilean geometry is transitional between Euclidean and Minkowski geometry.

\begin{remark} \label{rem:trans-link-ideal}
It turns out that $P_0$ is isomorphic to the semidirect product $\Isom(\R^{n-2})\ltimes\R^{n-1}$. (An explicit geometric interpretation can be given using the duality with Minkowski space.) We omit the precise details here, as we will explain concretely some examples of interest for this paper --- see Example \ref{ex:hp torical cusp} below. 
\end{remark}

We are now ready for the definition of cusp in each of the three cases.

\begin{defi} \label{defi cusp}
A \emph{cusp} in a hyperbolic (resp. anti-de Sitter, half-pipe) manifold is a region isometric to the quotient of $\{y_1>1\}$ in the  upper half-space model, by a subgroup $\Gamma$ of $P_1$ (resp. $P_{-1}$, $P_0$) acting properly and co-compactly on $H=\{y_1=1\}$.
\end{defi}

\begin{remark} \label{rem: finite volume}
By a standard computation, one sees that a cusp in a hyperbolic or anti-de Sitter manifold has finite volume. For half-pipe geometry, there is a canonical volume form as well \cite{surveyseppifillastre}, and in the same way it is immediate to check that the volume of a cusp is finite also in this case.
\end{remark}

\begin{example}\label{ex:hp torical cusp}
Simple examples are \emph{toric cusps}, where $\Gamma\cong\Z^{n-1}$ lies in the  normal subgroup $\R^{n-1}$ of $P_1$, $P_{-1}$, or $P_0$, and the section $\sfrac{H}{\Gamma}$ is a Euclidean, Minkowski, or Galilean $(n-1)$-torus, respectively. In hyperbolic geometry, $\Z^{n-1}$ acts by translations on a Euclidean horosphere, where the standard generators of $\Z^{n-1}$ are linearly independent translations. Exactly the same construction goes through for the AdS case for actions on Minkowski horospheres.

Let us now provide a similar example for half-pipe geometry. Recall that a horosphere $H$ in $\HP^n$ is the product of a horosphere $\widehat H$ in $\Hyp^{n-1}$ and the degenerate direction, in the standard decomposition $\HP^n\cong \Hyp^{{n-1}}\times\R$. We let $\Z^{n-1}$ act on $H$ in the following way. The first $n-2$ standard generators $\gamma_1,\ldots,\gamma_{n-2}$ act on $\widehat H$ by translation as above. We now define the action of the remaining generator $\gamma_{n-1}$. Suppose $\widehat H$ is obtained as the intersection of the hyperboloid in $\R^{1,n-1}$ with a lightlike hyperplane parallel to $w^\perp$, where $w$ is a lightlike vector in $\R^{1,n-1}$ and $w^\perp$ is its orthogonal complement with respect to the Minkowski bilinear form $\widehat b$. Then we let $\gamma_{n-1}$ act by translation by (a multiple of) $w$.

We now check, by means of the duality with Minkowski geometry (see Section \ref{sec motivations hp}), that this action of $\Z^{n-1}$ is faithful. Indeed, in the dual Minkowski picture, $\gamma_1,\ldots,\gamma_{n-2}$ act as
$$x\mapsto \widehat A\cdot x~,$$
where $\widehat A$ is a linear isometry such that $\widehat A\cdot w=w$. The generator $\gamma_{n-1}$ acts as 
$$x\mapsto x+\lambda w$$
for some $\lambda\in\R$. Since $\widehat A\cdot w=w$, it is clear that these actions commute. The resulting group $\Gamma<P_0$ thus provides an example of toric cusp in a half-pipe manifold.
\end{example}

The examples of geometric transition we construct in the second part of the paper will be examples of hyperbolic/anti-de Sitter/half-pipe manifolds with cusps. We will describe the geometry of the cusps, and their transition, in terms of the geometric structures induced on the quotient of a horosphere, as in Section \ref{sec: geometry polytope}. We will therefore be able to visualise the corresponding transition from Euclidean to Minkowski structures of codimension one.

\section{Half-spaces, reflections and rotations} \label{sec:half-spaces}

In this section, we describe the behaviour of hyperplanes and half-spaces under geometric transition and introduce projective reflections.
Finally, we introduce rotations, boosts and their infinitesimal analogues in half-pipe geometry.

\subsection{Dual projective sphere} \label{subsec:proj sphere}

Let us introduce the necessary notation.

{A \emph{hyperplane} (resp. \emph{subspace}) $H\subset\SP^n$ of the projective sphere is the image through the quotient map $\R^{n+1}\smallsetminus\{0\}\to\SP^n$ of a linear hyperplane (resp. \emph{subspace})} of $\R^{n+1}$. A \emph{half-space} $\l H\subset\SP^n$ of the projective sphere is the closure of one of the two connected components of $\SP^n\smallsetminus H$, where $H$ is a hyperplane. In other words, a half-space is the closure of an affine chart (see Section \ref{sec: proj structures}).

The \emph{dual projective sphere} is defined as 
$$\SP^{n,*}=\{\alpha\in\R^{n+1,*}\smallsetminus \{0\}\}/{\R_{>0}}~,$$
where $\R^{n+1,*}$ is the vector space of linear forms on $\R^{n+1}$. We will use coordinates with respect to the dual basis of the standard basis, namely the basis $\{e_0^*,\ldots,e_n^*\}$ defined by $e_i^*(x_0,\ldots,x_n)=x_i$. We will also denote elements of $\SP^{n,*}$ by
$$(\alpha)=(\alpha_0:\ldots:\alpha_n)\in\SP^{n,*}$$
if $\alpha=\alpha_0e_0^*+\ldots+\alpha_ne_n^*$.
The dual projective sphere $\SP^{n,*}$ is identified to the space of half-spaces in $\SP^n$, by associating to the class of a linear form $\alpha$ the half-space defined by:
$$\l H=\{[x]\in\SP^n\,|\,\alpha(x)\leq 0\,\}~.$$
By a small abuse of notation, we will sometimes denote a half-space $\l H\subset\SP^n$ with the corresponding point $(\alpha)\in\SP^{n,*}$ of the dual sphere.

The following elementary lemma will be useful to control the behaviour of half-spaces under geometric transition.

\begin{lemma} \label{lemma rescale halfspace}
Let $\l H=(\alpha_0:\ldots:\alpha_n)$ be a half-space in $\SP^n$. Then for every $t>0$, the half-space $\r_t\l H$ has coordinates
$$\r_t\l H=\left(\frac{\alpha_0}{t}:\ldots:\frac{\alpha_{n-1}}{t}:\alpha_n\right)~.$$
\end{lemma}
\begin{proof}
Recall that the projective transformation $\r_t$ is defined by 
$$\r_t=\left[\mathrm{diag}\left(1,\ldots,1,\frac{1}{t}\right)\right]~.$$
Given $[x]=[x_0:\ldots:x_n]$, then $[x]\in \r_t \l H$ if and only if $[x]=\r_t([x'])$ for some $[x']\in\l H$; namely:
$$[x_0:\ldots:x_{n-1}:x_n]=\left[x'_0:\ldots:x_{n-1}':\frac{x'_n}{t}\right]~,$$
where $[x'_0:\ldots:x'_{n-1}:x'_n]$ satisfies the defining condition for $\l H$, namely:
$$\alpha_0x'_0+\ldots+\alpha_nx'_n\leq 0~.$$
By multiplying by $1/t$, this is equivalent to
$$\frac{\alpha_0}{t}x_0+\ldots+\frac{\alpha_{n-1}}{t}x_{n-1}+\alpha_nx_n\leq 0~.$$
This proves the claim.
\end{proof}

\begin{defi}
A \emph{half-space} (resp. \emph{hyperplane}, \emph{subspace}) in $\X^n_t$ is a nonempty intersection of $\X_t^n$ with a half-space (resp. hyperplane, subspace) of $\SP^n$.
\end{defi}

\subsection{Hyperplanes in $\Hyp^n$ and $\AdS^n$} \label{sec geom halfspaces}

We shall now provide a geometric description of hyperplanes and half-spaces in $\Hyp^n$, $\AdS^n$ and $\HP^n$. Let us start with the hyperbolic space. We have the following simple lemma.

\begin{lemma} \label{lemma hyperplane hyp}
Given a half-space $\l H=(\alpha)$ of $\SP^n$, $\partial\l H$ intersects $\Hyp^n$ if and only if 
$$q_1(\alpha_0,\ldots,\alpha_n)>0~.$$
\end{lemma}

\begin{proof}
It is well-known that a hyperplane in $\R^{n+1}$ intersects $\Hyp^n$ if and only if its orthogonal complement for the bilinear form $b_1$ (whose associated quadratic form is $q_1$), seen as a line in Minkowski space $\R^{1,n}$, is spacelike. Compared to our lemma, there is only one small caveat: when choosing the dual basis of $\R^{n+1,*}$ as $\{e_0^*,\ldots,e_n^*\}$, we are essentially using the standard Euclidean product to identify $\R^{n+1,*}$ with $\R^{n+1}$. On the other hand, taking the orthogonal complement in $\R^{1,n}$ corresponds to choosing the basis $\{-e_0^*,e_1^*,\ldots,e_{n-1}^*,e_n^*\}$. However, the two choices differ by the following change of coordinates
$$(\alpha_0,\alpha_1,\ldots,\alpha_n)\mapsto (-\alpha_0,\alpha_1,\ldots,\alpha_n)$$
which is an isometry for the quadratic form $q_1$. Hence Lemma \ref{lemma hyperplane hyp} follows.
\end{proof}

By the same reason, we can also use the usual formulae to compute the dihedral angle between two half-spaces:

\begin{lemma} \label{lemma angle hyp}
Given $\alpha,\alpha'$ such that $q_1(\alpha),q_1(\alpha')<0$ let $\l H=(\alpha)$ and $\l H'=(\alpha')$ be the corresponding half-spaces. The hyperplanes $\partial\l H$ and $\partial\l H'$ intersect transversely in $\Hyp^n$ if and only if 
$${|b_1(\alpha,\alpha')|}<{\sqrt{|q_1(\alpha)|}\sqrt{|q_1(\alpha')|}}~.$$
In this case, the dihedral angle $\theta$ between the half-spaces $\l H$ and $\l H'$ satisfies 
$$\cos\theta=-\frac{b_1(\alpha,\alpha')}{\sqrt{|q_1(\alpha)|}\sqrt{|q_1(\alpha')|}}~.$$
\end{lemma}

One can also find similar formulae for the anti-de Sitter case. It turns out that every hyperplane $\partial\l H$ in $\SP^n$ intersects $\AdS^n$ non-trivially. Moreover, recall that a hyperplane in anti-de Sitter space is called \emph{spacelike}, \emph{timelike} or \emph{lightlike} if the induced bilinear form is positive definite, indefinite or degenerate, respectively. One then has the following characterisation:

\begin{lemma} \label{lemma hyperplane ads}
Let $\l H=(\alpha_0:\ldots:\alpha_n)$ be a half-space of $\SP^n$. Then $\partial\l H\cap\AdS^n$ is:
\begin{itemize}
\item Spacelike if $q_{-1}(\alpha_0,\ldots,\alpha_n)<0$. In this case, $\partial\l H\cap\AdS^n$ consists of two disconnected totally geodesic copies of $\Hyp^{n-1}$.
\item Timelike if $q_{-1}(\alpha_0,\ldots,\alpha_n)>0$. In this case, $\partial\l H\cap\AdS^n$ consists of a totally geodesic copy of $\AdS^{n-1}$.
\item Lightlike if $q_{-1}(\alpha_0,\ldots,\alpha_n)=0$.
\end{itemize}
\end{lemma}

One can then compute angles between hyperplanes by a direct formula in terms of the bilinear form $b_{-1}$. We provide here the formula for the case of two spacelike hyperplanes.

We recall that the \emph{angle} between two spacelike hyperplanes in a Lorentzian space is a number $\varphi\in [0,+\infty)$, which is defined as the distance in a copy of $\Hyp^{n-1}$ between the two points corresponding to the two future unit normal vectors in the tangent space at an intersection point. This notion of angle is used in Lemma \ref{lem: AdS angle} below. See also Figure \ref{fig:AdSplanes}.

\begin{lemma} \label{lem: AdS angle}
Given $\alpha,\alpha'$ such that $q_{-1}(\alpha),q_{-1}(\alpha')<0$ let $\l H=(\alpha)$ and $\l H'=(\alpha')$ be the corresponding half-spaces. The hyperplanes $\partial\l H$ and $\partial\l H'$ intersect transversely in $\AdS^n$ if and only if 
$${|b_{-1}(\alpha,\alpha')|}>{\sqrt{|q_{-1}(\alpha)|}\sqrt{|q_{-1}(\alpha')|}}~.$$
In this case, the angle $\varphi$ between the hyperplanes $\partial\l H$ and $\partial\l H'$ satisfies the equation:
$$\cosh\varphi=\frac{|b_{-1}(\alpha,\alpha')|}{\sqrt{|q_{-1}(\alpha)|}\sqrt{|q_{-1}(\alpha')|}}~,$$
where the sign of $b_{-1}(\alpha,\alpha')$
is negative if $\l H\cap \l H'$ contains timelike segments with endpoints in $\partial\l H\cap\partial\l H'$, and positive otherwise.
\end{lemma}

We also need to briefly analyse the situation for the intersection between two timelike hyperplanes. 
See also Figure \ref{fig:AdSplanes_timelike}.

\begin{lemma} \label{lem: AdS intersection timelike}
Given $\alpha,\alpha'$ such that $q_{-1}(\alpha),q_{-1}(\alpha')>0$ let $\l H=(\alpha)$ and $\l H'=(\alpha')$ be the corresponding half-spaces. The hyperplanes $\partial\l H$ and $\partial\l H'$ always intersect in $\AdS^n$. Moreover:
\begin{itemize}
\item The intersection is spacelike (i.e. a totally geodesic copy of $\Hyp^{n-2}$) if and only if $${|b_{-1}(\alpha,\alpha')|}>{\sqrt{|q_{-1}(\alpha)|}\sqrt{|q_{-1}(\alpha')|}}~.$$
\item The intersection is timelike (i.e. a totally geodesic copy of $\AdS^{n-2}$) if and only if $${|b_{-1}(\alpha,\alpha')|}<{\sqrt{|q_{-1}(\alpha)|}\sqrt{|q_{-1}(\alpha')|}}~.$$
\end{itemize} 
\end{lemma}

\begin{figure}
\includegraphics[scale=0.35]{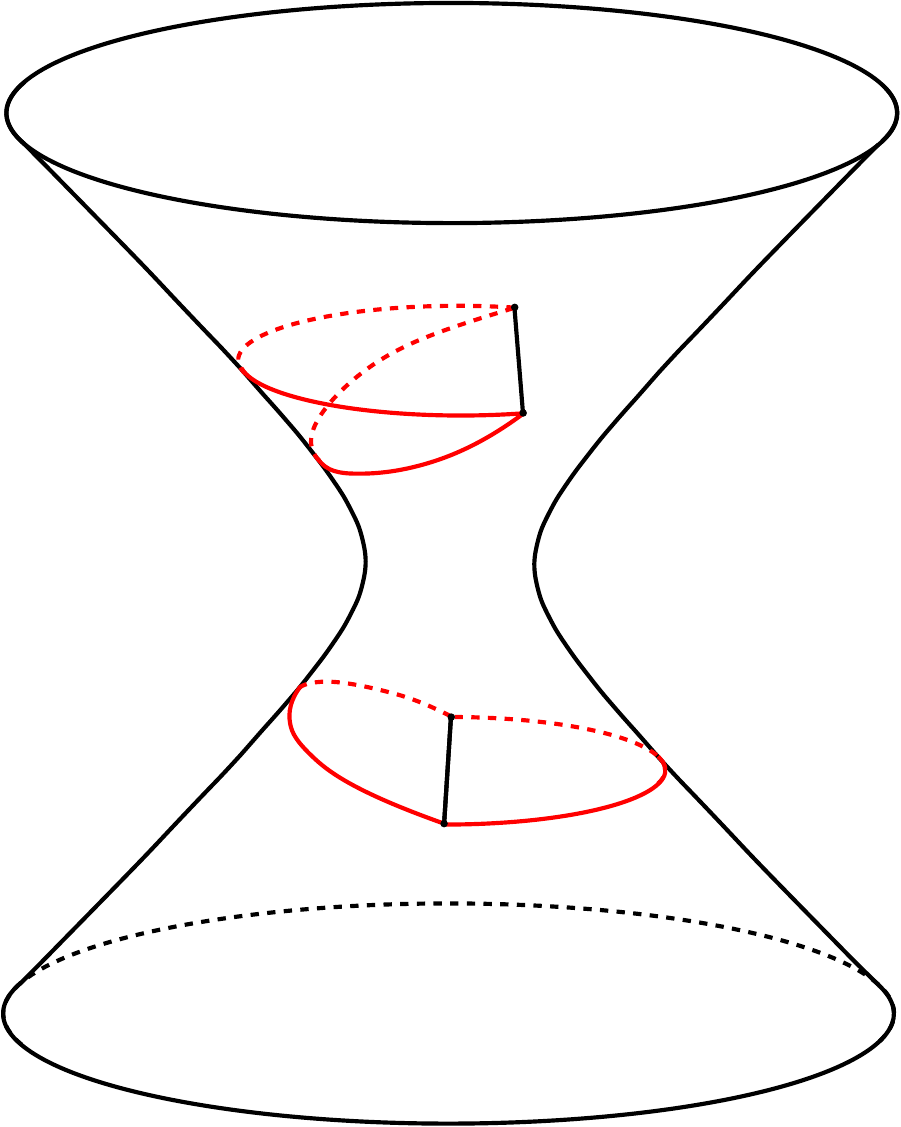}
\caption[Intersecting spacelike hyperplanes in $\AdS^n$.]{\footnotesize In an affine chart for anti-de Sitter space, the two possibilities (above and below in the same figure) for the configuration of $\l H$ and $\l H'$ as in Lemma \ref{lem: AdS angle}.}
\label{fig:AdSplanes}
\end{figure}

\begin{figure}
\begin{minipage}[c]{.4\textwidth}
\centering
\includegraphics[scale=.35]{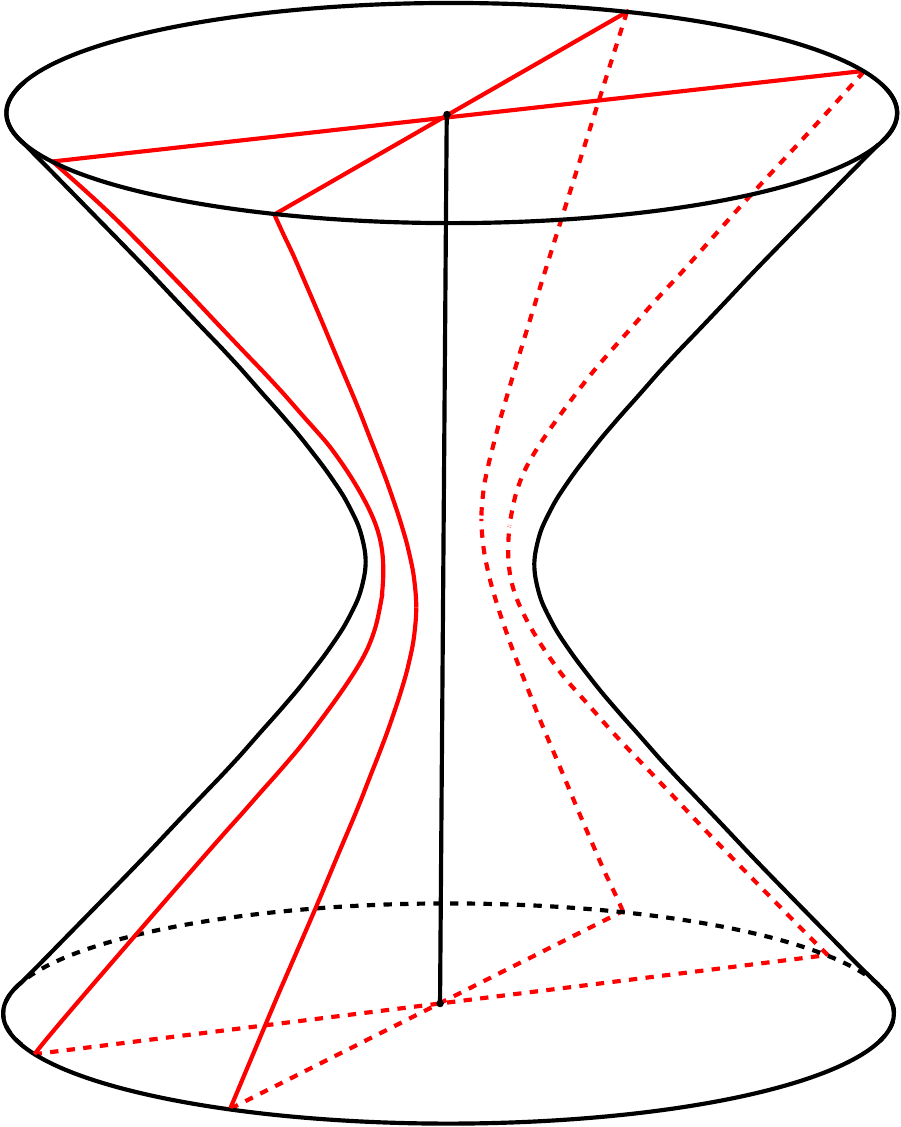}
\end{minipage}
\hspace{1cm}
\begin{minipage}[c]{.4\textwidth}
\centering
\includegraphics[scale=.35]{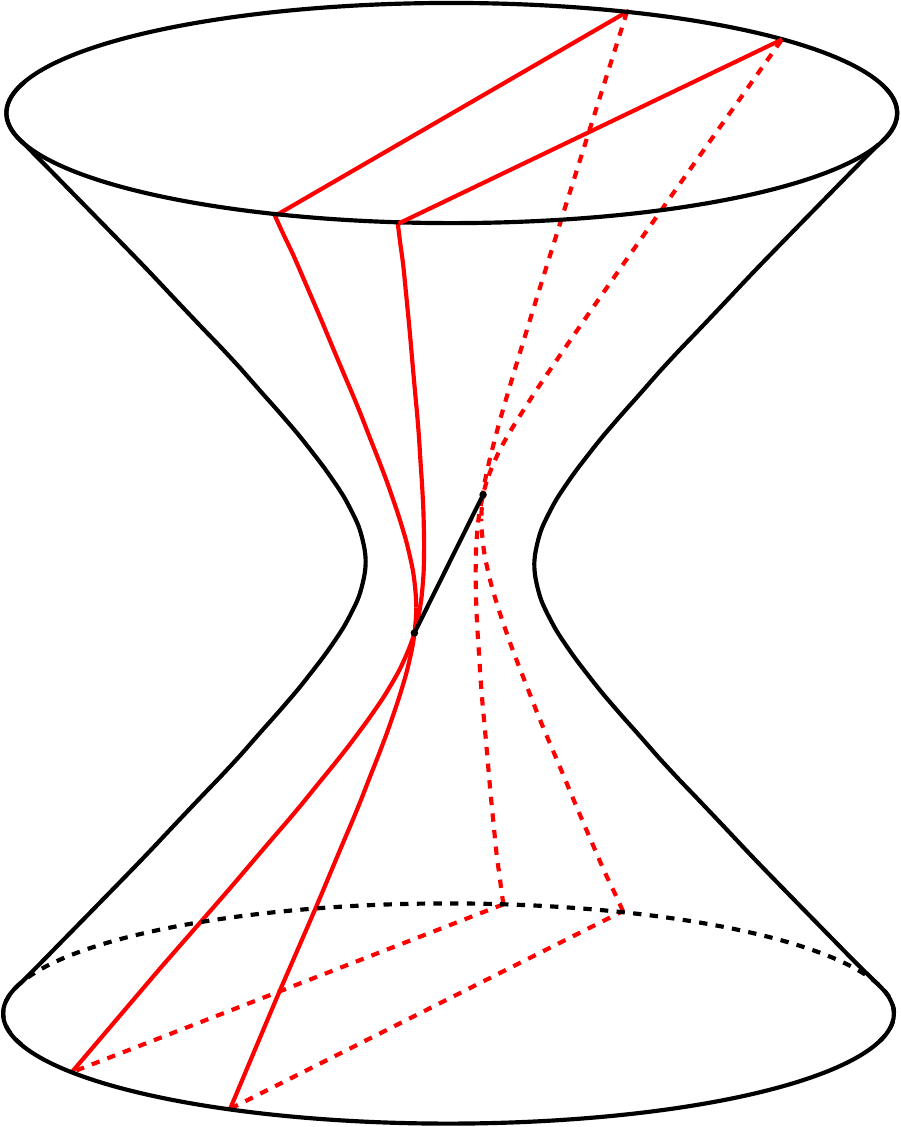}
\end{minipage}
\caption[Intersecting timelike hyperplanes in $\AdS^n$.]{\footnotesize The two possibilities for the intersection of two timelike planes in $\AdS^3$: a timelike (left) or spacelike (right) line.}
\label{fig:AdSplanes_timelike}
\end{figure}

\subsection{Hyperplanes in half-pipe geometry}

Let us now move to the case of hyperplanes in half-pipe space. In this case, we have two types of hyperplanes: \emph{spacelike} hyperplanes, for which the induced bilinear form is positive definite (these are isometrically embedded copies of $\Hyp^{n-1}$), and \emph{degenerate} hyperplanes, for which the induced bilinear form is indeed degenerate. These two types are detected by the following lemma.

\begin{figure}
\includegraphics[scale=0.5]{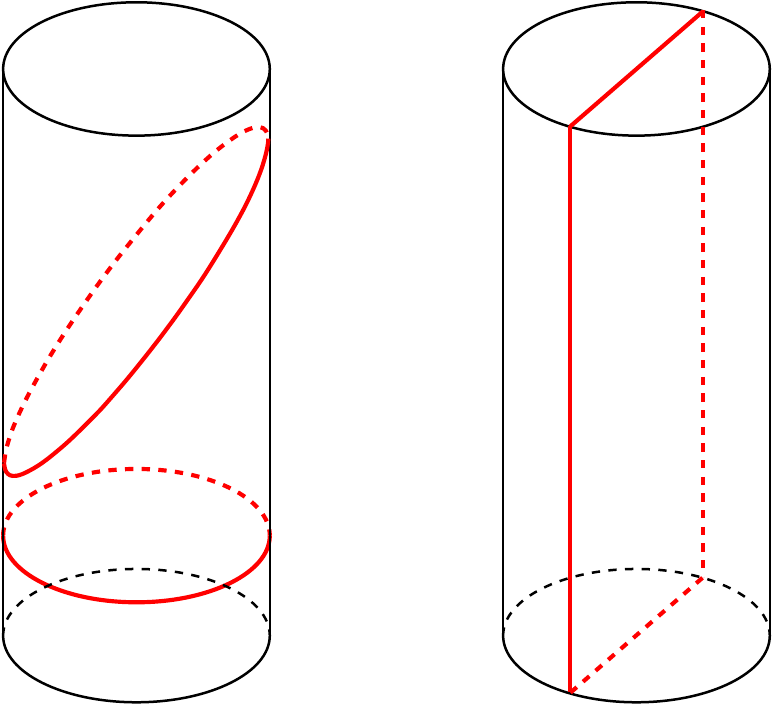}
\caption[Hyperplanes in half-pipe space.]{\footnotesize Hyperplanes in the affine (cylindric) model of $\HP^n$: on the left, two spacelike hyperplanes, on the right, a degenerate hyperplane.}
\label{fig:HPplanes}
\end{figure}

\begin{lemma} \label{lemma hyperplane hp}
A half-space $\l H=(\alpha_0:\ldots:\alpha_n)$ in $\SP^n$ intersects $\HP^n$ if and only if $\alpha_n\neq 0$ or $q_0(\alpha_0,\ldots,\alpha_n)>0$. In this case, the hyperplane $\partial\l H\cap\HP^n$ is:
\begin{itemize}
\item Spacelike if $\alpha_n\neq 0$. In this case, $\partial\l H\cap\HP^n$ consists of a copy of $\Hyp^{n-1}$.
\item Degenerate if $\alpha_n= 0$ and $q_{0}(\alpha_0,\ldots,\alpha_n)>0$. In this case, $\partial\l H\cap\HP^n$ consists of a copy of $\HP^{n-1}$.
\end{itemize}
\end{lemma}

\begin{proof}
If $\alpha_n=0$, then $\l H$ is of the form 
$$\l H=\{(\widehat x:x_n)\,|\,\widehat x\in\widehat{\l H},x_n\in\R\}~,$$where $\widehat{\l H}$ is a half-space of $\SP^{n-1}$ defined by the condition $\alpha_0x_0+\ldots+\alpha_{n-1}x_{n-1}\leq 0$. Hence in the affine chart $\Aff^n$, $\l H$ is the product of a half-space in $\Aff^{n-1}\subset\Aff^n$ and $\R$. Since $\HP^n$ can be regarded in $\Aff^n$ as the product of $\Hyp^{n-1}\times\R$, the condition that $\l H$ intersects $\HP^n$ is equivalent to the condition that $\widehat{\l H}$ intersects $\Hyp^{n-1}$, which by Lemma \ref{lemma hyperplane hyp} is $-\alpha_0^2+\alpha_1^2+\ldots+\alpha_{n-1}^2>0$, or equivalently,  
$$q_0(\alpha_0,\ldots,\alpha_n)>0~.$$
In this case, for each $\widehat x\in\widehat{\l H}$, the ``vertical'' line $\{(\widehat x,t)\,|\,t\in \R\}$ is in $\partial\l H$ and its tangent space coincides with the kernel of the degenerate form $\sigma^*b_0$ (see Section \ref{subsec:hp structures}). Hence $\partial\l H\cap\HP^n$ is degenerate, namely it is vertical in the affine chart $\Aff^n$. See Figure \ref{fig:HPplanes}, on the right. 

The other case is thus $\alpha_n\neq 0$. In this case, the sign of $\alpha_n$ determines whether $\l H$ is unbounded in the positive or negative $x_n$ direction. In fact, up to multiplying by a positive number, we can assume 
$$\l H=(\alpha_0:\ldots:\alpha_{n-1}:\pm 1)~.$$
For instance, if $\alpha_n=1$, we get that $\l H$ is defined by the condition
$$ \alpha_0x_0+\ldots \alpha_{n-1}x_{n-1}+x_n\leq 0~,$$
and therefore $\partial\l H$ is given by the condition 
\begin{equation} \label{eq minkowski duality}
\alpha_0x_0+\ldots \alpha_{n-1}x_{n-1}+x_n=0~.
\end{equation}
This shows that $\partial\l H\cap\HP^n$ is transverse to the degenerate direction, and it is therefore of spacelike type.
\end{proof}

\begin{remark}
This discussion also gives a deeper insight into the duality between $\HP^n$ and Minkowski space. Recall that, as mentioned in Section \ref{sec motivations hp}, any point of $\HP^n$ corresponds to a spacelike hyperplane in Minkowski space $\R^{1,n-1}$. Dually, any spacelike hyperplane $\partial\l H$ of $\HP^n$ corresponds to a point in $\R^{1,n-1}$, which turns out to be the intersection point of all the spacelike hyperplanes associated to points of $\partial\l H$. Such a dual point is easily computed: if the hyperplane $\partial\l H$ in $\HP^n$ is determined by the equation
$$\alpha_0x_0+\ldots +\alpha_{n-1}x_{n-1}+x_n=0~,$$
then its dual point is 
$$p=(\alpha_0,-\alpha_1,\ldots,-\alpha_{n-1})~.$$
Moreover, this correspondence is again natural with respect to the action of the isometry groups $G_{\HP^n}$ and $\Isom(\R^{1,n-1})$.
\end{remark}

As a consequence of the above remark, the condition that two spacelike hyperplanes $H,H'\subset\HP^n$ intersect is equivalent to the condition that the two dual points $p$ and $p'$ in $\R^{1,n-1}$ belong to the same spacelike hyperplane. Indeed, every point in the intersection $H\cap H'$ corresponds to a spacelike hyperplane in $\R^{1,n-1}$ which contains both $p$ and $p'$. See also Figure \ref{fig:sheafHP}. This shows the following:

\begin{lemma} \label{lemma intersection half pipe}
Given $\widehat \alpha,\widehat \alpha'$ such that $\widehat  q(\widehat  \alpha),\widehat  q(\widehat \alpha')<0$, consider the half-spaces $\l H=(\widehat \alpha:1)$ and $\l H'=(\widehat  \alpha':1)$. The hyperplanes $\partial\l H$ and $\partial\l H'$ intersect transversely in $\HP^n$ if and only if $\widehat\alpha-\widehat\alpha'$ is a spacelike segment in Minkowski space.
\end{lemma}

\begin{figure}
\includegraphics[scale=0.5]{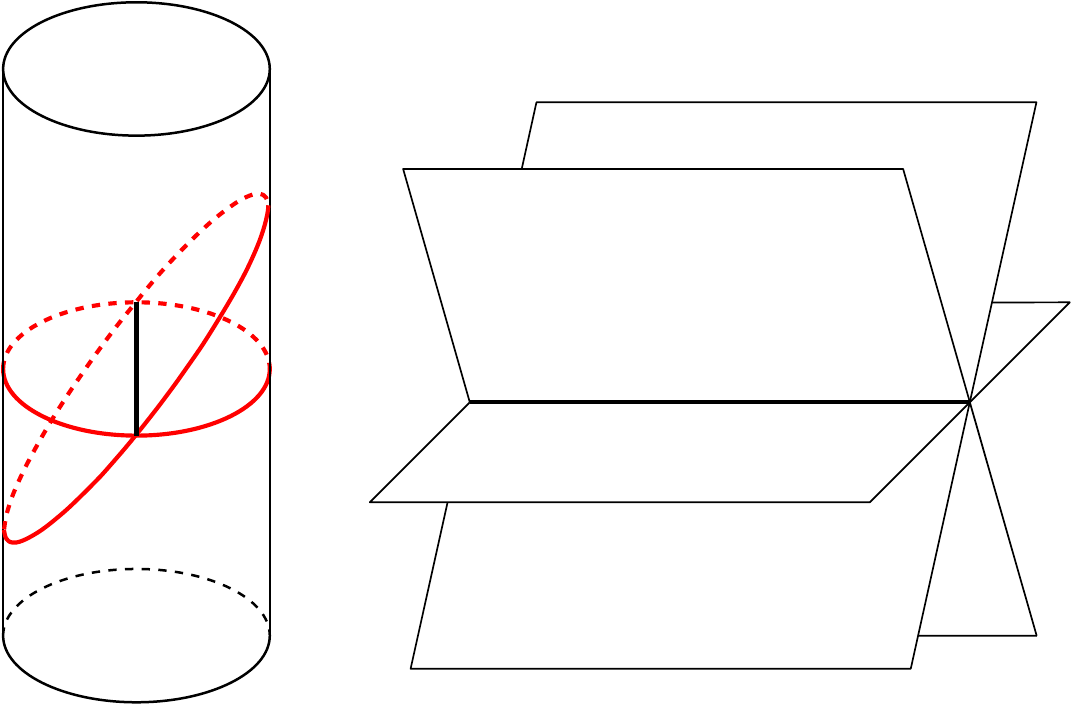}
\caption[Sheaf of hyperplanes in $\HP^n$ and duality with $\R^{1,n-1}$.]{\footnotesize A sheaf of hyperplanes in $\HP^n$ (on the left) corresponds to the points of $\R^{1,n-1}$ lying on a spacelike line $\ell$ (on the right). Viceversa, the intersection of the sheaf is in bijection with the spacelike hyperplanes containing $\ell$.}
\label{fig:sheafHP}
\end{figure}

Recall that we denote by $\widehat q$ the quadratic form of signature $(-,+,\ldots,+)$ on $\R^n$, viewed as the subspace of $\R^{n+1}$ defined by the vanishing of the last coordinate.

In \cite{dancigertransition} a notion of angle between spacelike hyperplanes in $\HP^n$ was introduced, which is the \emph{infinitesimal version} of dihedral angles in $\Hyp^n$ and $\AdS^n$ under the geometric transition we described above. In this setting, it is easy to give a definition of dihedral angle between two spacelike hyperplanes:

\begin{defi} \label{def: angle HP}
Given two HP spacelike half-spaces defined by $\l H=(\widehat \alpha:1)$ and $\l H'=(\widehat  \alpha':1)$ (i.e. such that $\widehat  q(\widehat  \alpha),\widehat  q(\widehat \alpha')<0$), the \emph{angle} between $\partial \l H$ and $\partial \l H'$ is the number 
$$\psi=\sqrt{\widehat q(\widehat\alpha-\widehat\alpha')}\in [0,+\infty)~.$$
\end{defi}
In other words, the angle is defined as the length of the segment connecting the two dual points, which is spacelike by Lemma \ref{lemma intersection half pipe}. In Section \ref{sec rotations} below we show that this notion actually coincides with the infinitesimal version of the angles between hyperplanes in $\Hyp^n$ and $\AdS^n$.

\subsection{Rotations, boosts, and their infinitesimal analogues} \label{sec rotations}

In this section we briefly introduce rotations and their analogues in AdS geometry (boosts) and in half-pipe geometry (infinitesimal rotations). This will be relevant for our Theorem \ref{teo: main}, because the holonomy of the geometric structures on a peripheral loop around the singular locus $\Sigma$ will consists of these elements.

\begin{defi}
A \emph{rotation} (resp. \emph{boost} or \emph{infinitesimal rotation}) is a non-trivial orientation-preserving element of $\Isom(\Hyp^n)$ (resp. $\Isom(\AdS^n)$ or $G_{\HP^n}$) which fixes point-wise a co-dimension two subspace (which is required to be spacelike, for $\AdS^n$ and $\HP^n$).
\end{defi}

If $r$ is such a rotation (resp. boost or infinitesimal rotation), the \emph{angle} associated to $r$ is defined as the angle between $H$ and $r(H)$, where $H$ is any (spacelike) hyperplane containing $\mathrm{Fix}(r)$.

\begin{remark} \label{remark:reflections}
To motivate the existence of infinitesimal rotations in $\HP^n$, recall Lemma \ref{lemma intersection half pipe} and Figure \ref{fig:sheafHP}. Two intersecting spacelike hyperplanes correspond precisely to two points $\widehat \alpha,\widehat \alpha'\in\R^{1,n-1}$ which are spacelike separated. Hence any translation in $\R^{1,n-1}$ in the direction of $\widehat \alpha-\widehat \alpha'$ induces an infinitesimal rotation $r$ of $\HP^n$, which fixes the points of $\HP^n$ corresponding to spacelike hyperplanes $P$ containing $\widehat \alpha-\widehat \alpha'$. 

Observe that $r$ also fixes pointwise an entire degenerate hyperplane in $\HP^n$, corresponding to all the translates of the hyperplanes $P$ as above. This is a qualitative difference with respect to hyperbolic and AdS geometries.
\end{remark}

We now briefly show that the infinitesimal angle in half-pipe geometry is exactly the infinitesimal version of angle in $\Hyp^n$ and $\AdS^n$. Let $r_t$ be a smooth family of rotations of angle $\theta(t)$ with $\theta(0)=0$. Up to isometries, we assume that 

$$r_t=\left(
\begin{array}{cccc}
  1 &0 &\ldots&  0 \\

 0 &\ddots &  & \vdots \\

 \vdots &  &  \cos\theta(t)& \sin\theta(t) \\

  0 &\ldots & -\sin\theta(t) & \cos\theta(t)
\end{array}
\right)~.$$

That is, $r_t$ is a rotation which fixes the codimension two totally geodesic subspace defined by $x_{n-1}=x_n=0$, and sends the hyperplanes $x_n=0$ to another hyperplane forming an angle $\theta(t)$. By a direct computation, one sees that
$$\lim_{t\to 0}\r_t r_t\r_t^{-1}=\left(
\begin{array}{cccc}
  1 &0 &\ldots&  0 \\

 0 &\ddots &  & \vdots \\

 \vdots &  &  1& 0 \\

  0 &\ldots & -\dot\theta & 1
\end{array}
\right)~,$$
which is a half-pipe infinitesimal rotation, corresponding under the usual isomorphism with $\Isom(\R^{1,n-1})$ to a translation of the vector $(0,\ldots,0,-\dot\theta)$. Hence in the limit the angle of the infinitesimal rotation is $|\dot\theta|$, by Definition \ref{def: angle HP}. The computation can be done analogously for a boost in anti-de Sitter space by replacing sin and cos by sinh and cosh, respectively.

This argument also explains the name ``infinitesimal rotation'', introduced in \cite{dancigertransition}.

\subsection{Reflections along hyperplanes} \label{sec: reflections}

More generally, the holonomy group of our geometric structures will be generated by compositions of reflections through the bounding hyperplanes of the polytopes that we will glue. 
In this section, we describe half-pipe  limits of reflections in $\Hyp^n$ or $\AdS^n$.

\begin{defi} \label{def:reflection} 
A \emph{(projective) reflection} is a non-trivial involution $r\in\Aut(\SP^n)$ that fixes point-wise a hyperplane.
\end{defi}

Note that in $\Isom(\Hyp^n)$ and $\Isom(\AdS^n)$ there is a unique reflection fixing a given hyperplane. This turns out not to be true in half-pipe geometry. Since this point will be very relevant in the following, let us explain this phenomenon more precisely.

Let $H$ be a degenerate hyperplane in $\HP^n$, as in the second point of Lemma \ref{lemma hyperplane hp}. In the duality with Minkowski space, $H$ corresponds to the set of all the spacelike hyperplanes $P$ in $\R^{1,n-1}$ having normal vector in a totally geodesic hyperplane of $\Hyp^{n-1}$, which is of the form $w^\perp\cap \Hyp^{n-1}$, where $w$ is some spacelike vector in $\R^{1,n-1}$ and $w^\perp$ is its orthogonal complement for the Minkowski product. Now, every reflection of $\R^{1,n-1}$ in a spacelike hyperplane orthogonal to $w$ leaves (set-wise) invariant each such hyperplane $P$. See Figure \ref{fig:arg_reflections}.

\begin{figure}
\includegraphics[scale=0.4]{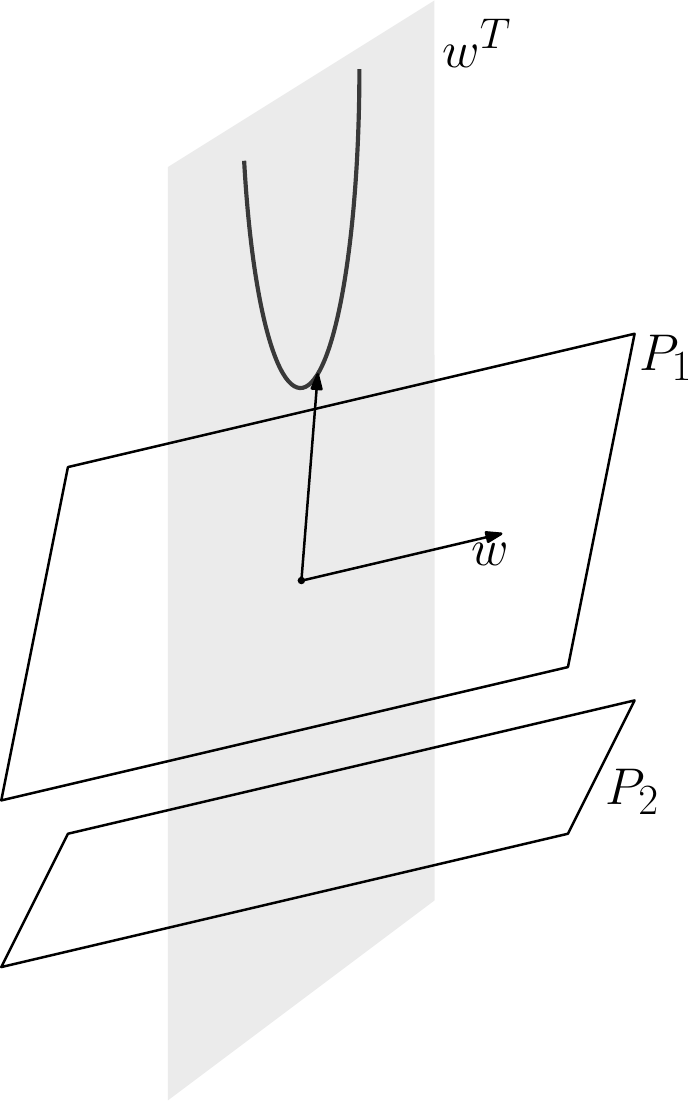}
\caption[Non-uniqueness of half-pipe reflections.]{\footnotesize The argument of Proposition \ref{prop:hp reflections deg}: Minkowski reflections in hyperplanes parallel to $w^\perp$ (in grey) leave set-wise invariant every spacelike hyperplane (like $P_1$ and $P_2$ in the figure) having normal vector in $w^\perp\cap \Hyp^{n-1}$ (this intersection is pictured as a hyperbola here). They all induce half-pipe reflections fixing a degenerate hyperplane, as in Figure \ref{fig:HPplanes} on the right.} \label{fig:arg_reflections}
\end{figure}

In summary, the above argument shows the following proposition, which is a remarkable difference with respect to hyperbolic and AdS geometry.

\begin{prop} \label{prop:hp reflections deg}
Given any degenerate hyperplane in $\HP^n$, there is a one-parameter family of reflections in $G_{\HP^n}$  which fix the hyperplane pointwise.
\end{prop}

On the other hand, for spacelike hyperplanes uniqueness of the half-pipe reflection holds:

\begin{prop}\label{prop:hp reflections space}
Given any spacelike hyperplane in $\HP^n$, there is a unique reflection in $G_{\HP^n}$  which fix the hyperplane pointwise.
\end{prop}

To see this, recall that a spacelike hyperplane $H$ in $\HP^n$ corresponds to all the spacelike hyperplanes $P$ in $\R^{1,n-1}$ which contain a given point $p$, and we can assume that $p$ is the origin. If an isometry $(\widehat A,v)$ of $\R^{1,n-1}$ fixes (set-wise) all the hyperplanes $P$ going through the origin, it must also fix the origin itself, hence it must be linear (i.e. the translation part $v$ is trivial). Since it fixes timelike normal directions, then $\widehat A$ is either $\mathrm{id}$ or  $-\mathrm{id}$. In conclusion, the unique half-pipe reflection fixing $H$ is $(-\mathrm{id},0)$.

\section{Polytopes and cone-manifolds} \label{sec:poly-conemfds}

In this section, we provide some additional tools to prove Theorem \ref{teo: main}. We first introduce projective polytopes and simple projective cone-manifolds. Then, we describe the singularities of such cone-manifolds {in the hyperbolic, AdS, and HP cases} by means of the $(G,X)$-structures given by the links of points.

\subsection{Polytopes} \label{section polytopes}

We now introduce our main tool to provide examples of transition. Recall Section \ref{sec:half-spaces} about half-spaces of the projective sphere.

\begin{defi}
An $n$-dimensional \emph{projective polytope} is a finite intersection of half-spaces
$${\mathcal P}=\l H_1\cap\ldots\cap\l H_N\subset\SP^n$$
in the projective sphere such that the interior of ${\mathcal P}$ is non-empty.
\end{defi}

We will always assume that the set of half-spaces defining ${\mathcal P}$ is minimal, that is, there is no $\l H_i$ containing the intersection of the remaining half-spaces.
In this case, each $\partial\l H_i$ is called a \emph{bounding hyperplane} of ${\mathcal P}$.

The polytope ${\mathcal P}$ is naturally stratified into $k$-\emph{faces}, $k=0,\ldots,n$, as follows. The unique $n$-face is ${\mathcal P}$ itself. The $(n-1)$-faces, called \emph{facets}, are given by $\mathcal F_i={\mathcal P}\cap\partial\l H_i$, $i=1,\ldots,N$. Now, each facet $\mathcal F_i$ is a projective $(n-1)$-polytope. The $(n-2)$-faces of ${\mathcal P}$, called \emph{ridges}, are by definition the facets of each $\mathcal F_i$. One proceeds inductively in this way, until reaching the 1-faces, called \emph{edges}, and the 0-faces, called \emph{vertices}. The \emph{combinatorics} of a polytope ${\mathcal P}$ is (the isomorphism class of) the poset of the faces of ${\mathcal P}$, ordered by set inclusion.

Given a point $p\in\mathcal P$, we have $p=[v]$ for some $v\in\R^{n+1}\smallsetminus\{0\}$. Each (possibly none) half-space $\l H_i$ such that $p\in\partial\l H_i$ determines a half-space of the projective sphere $\SP^{n-1}$ over the quotient vector space $\R^{n+1}/{\langle v\rangle}$. The \emph{link} of the point $p$ is the $(n-1)$-dimensional projective polytope $\mathcal L_p\subset\SP^{n-1}$ defined as the intersection of such half-spaces. (If $p$ is in the interior of $\mathcal P$, then $\mathcal L_p$ is nothing but $\SP^{n-1}$.)

The polytope ${\mathcal P}$ is said to be \emph{simple} if every $k$-face is contained in exactly $(n-k)$ bounding hyperplanes. Equivalently, the link of each vertex is a simplex.

\begin{defi}
A \emph{hyperbolic} (resp. \emph{half-pipe} or \emph{anti-de Sitter}) \emph{polytope} is a non-empty subset
$$\hat {\mathcal P}={\mathcal P}\cap\Hyp^n\quad(\mbox{resp.\ }\hat {\mathcal P}={\mathcal P}\cap\HP^n\ \mbox{or}\ \hat {\mathcal P}={\mathcal P}\cap\AdS^n),$$
where ${\mathcal P}\subset\SP^n$ is a projective polytope.
\end{defi}

\subsection{Simple projective cone-manifolds} \label{sec cone-mfds}

The process of geometric transition typically involves a path of \emph{singular} geometric structures. In this section, we describe such structures in the special case of Theorem \ref{teo: main}.

We call \emph{stratified manifold} a topological $n$-manifold $\mathcal{X}$ together with a \emph{stratification}, that is a partition
$$\mathcal{X}=\mathcal{X}^{[0]}\sqcup\ldots\sqcup \mathcal{X}^{[n]}$$
such that $\mathcal{X}^{[k]}$ is an embedded $k$-manifold with empty boundary. The connected components of $\mathcal{X}^{[k]}$ are called $k$\emph{-strata}.

Let us fix $n$ half-spaces $\l H_1,\ldots,\l H_n\subset\SP^n$ such that the hyperplanes $\partial\l H_1,\ldots,\partial\l H_n$ are in general position. The intersection
$${\mathcal P}=\l H_1\cap\ldots\cap\l H_n\subset\SP^n$$
is a simple projective polytope (see Section \ref{section polytopes}).

Let us now consider the \emph{double} $D({\mathcal P})$ of ${\mathcal P}$, that is, the space obtained from two copies of ${\mathcal P}$ by identifying the two copies of $\partial {\mathcal P}$ through the map induced by the identity. Note that $D({\mathcal P})$ is homeomorphic to the $n$-sphere. By considering the natural stratification ${\mathcal P}^{[0]}\sqcup\ldots\sqcup {\mathcal P}^{[n]}$ of ${\mathcal P}$ induced by its faces, we define the following stratification of $D({\mathcal P})$:
\begin{itemize}
\item $D({\mathcal P})^{[k]}={\mathcal P}^{[k]}$ for $k\leq n-2$,
\item $D({\mathcal P})^{[n-1]}=\emptyset$, 
\item $D({\mathcal P})^{[n]}=D({\mathcal P}^{[n-1]}\cup {\mathcal P}^{[n]})$.
\end{itemize}

Recalling now Section \ref{sec: proj structures} about projective stuctures, we note that for $k\leq n-2$ each $k$-stratum is a projective manifold homeomorphic to $\R^k$, while the $n$-stratum $D({\mathcal P})^{[n]}$ does not have a preferred projective structure. Choosing for each $i\in\{1,\ldots,n\}$ a projective reflection $r_i$ (recall Definition \ref{def:reflection}) that fixes the hyperplane $\partial\l H_i$, we have a well defined projective structure also on $D({\mathcal P})^{[n]}$.

We call $\mathcal{D}$ the double $D({\mathcal P})$ together with its stratification and such an additional projective structure on each of its strata. This will be the local model for our cone-manifolds. 

\begin{defi} \label{def:proj_simple_cone-mfd}
A \emph{simple projective cone-manifold} is a stratified manifold $\mathcal{X}$ with an atlas of stratum-preserving charts, each with values in some $\mathcal D$ constructed as above, whose transition functions restrict to an isomorphism of projective manifolds on each stratum.
\end{defi}

\begin{remark}
One could give a much more general definition of ``projective cone-manifold'', by induction on the dimension. We do not need this here. However, our notion of simple projective cone-manifold includes the one introduced by Danciger in his works on geometric transition. In contrast with the projective cone-manifolds of Theorem \ref{teo: main}, the ones considered in \cite{danciger,dancigertransition,LMR} are stratified as $\mathcal{X}=\mathcal{X}^{[n]}\sqcup \mathcal{X}^{[n-2]}$. In other words, the singular locus $\Sigma\subset \mathcal{X}$ is a codimension-two submanifold.
\end{remark}

We note that each stratum of such a cone-manifold $\mathcal{X}$ is a real projective manifold. The set
$$\Sigma=\mathcal{X}^{[0]}\cup\ldots\cup \mathcal{X}^{[n-2]}$$
is called the \emph{singular locus} of $\mathcal{X}$, while $\mathcal{X}^{[n-1]}=\emptyset$ and $\mathcal{X}^{[n]}=\mathcal{X}\smallsetminus\Sigma$ (which is the unique $n$-stratum if $\mathcal{X}$ connected) is called the \emph{regular locus} of $\mathcal{X}$. The singular locus $\Sigma$ is an $(n-2)$-complex with generic singularities: $\Sigma$ is empty if $n=1$, a discrete set if $n=2$, and is locally modelled on the cone over the $(n-3)$-skeleton of an $(n-1)$-simplex if $n\geq3$. In particular, $\Sigma$ is a trivalent graph if $n=3$, and a so called \emph{foam} if $n=4$.

Given a point $p\in\mathcal X\smallsetminus\Sigma$, we define its \emph{link} $\mathcal L_p$ to be the sphere of directions at $p$ of the projective manifold $X\smallsetminus\Sigma$, that is, the projective sphere $\SP^{n-1}$ over the tangent space at $p$. If instead $p\in\Sigma$, we define its \emph{link} $\mathcal L_p$ to be the double of the link (see Section \ref{section polytopes}) of a corresponding point in $\partial\mathcal P$ through a chart. The link $\mathcal L_p$ is naturally a simple projective cone-manifold homeomorphic to the $(n-1)$-sphere.

For example, if $\mathcal X$ has dimension $n=2$, then $\mathcal L_p$ is a projective circle, which is equivalent to $\SP^1$ if and only if $p$ is non-singular. If $n=3$, then $\mathcal L_p$ is a cone 2-sphere with 0, 2, or 3 singular points, depending on whether $p$ belongs to a 3-, 1-, or 0-stratum, respectively. If $n=4$, the singular locus of the cone 3-sphere $\mathcal L_p$ is empty, an unknotted cicle, an unknotted theta graph, or the 1-skeleton of a tetrahedron, depending whether $p$ belongs to a 4-, 2-, 1-, or 0-stratum, respectively.

The \emph{holonomy representation} and \emph{developing map} of a projective cone-manifold $\mathcal{X}$ are by definition those of its regular locus. The holonomy $\rho(\gamma)$ of a meridian $\gamma\in\pi_1(\mathcal{X}\smallsetminus\Sigma)$ of an $(n-2)$-stratum is conjugated to a product of reflections $r_ir_j\in\Aut(\SP^n)$ that fix a common $(n-2)$-subspace of $\SP^n$.

\subsection{The hyperbolic, AdS, and HP case} \label{sec:HS}

We are interested in some special classes of simple projective cone-manifolds:

\begin{defi} \label{def: AdS cone-manifold}

A simple projective cone-manifold $\mathcal{X}$ is said to be \emph{hyperbolic} (resp. \emph{half-pipe} or \emph{anti-de Sitter}) if 
\begin{itemize}
\item each chart has values in some $D(\hat {\mathcal P})\subset D({\mathcal P})=\mathcal D$, where $\hat {\mathcal P}={\mathcal P}\cap\Hyp^n$ (resp. $\hat {\mathcal P}={\mathcal P}\cap\HP^n$ or $\hat {\mathcal P}={\mathcal P}\cap\AdS^n$);
\item for each $\mathcal D$, the reflections $r_1,\ldots,r_n$ belong to $\Isom(\Hyp^n)$ (resp. $G_{\HP^n}$ or $\Isom(\AdS^n)$);
\item the transition functions restrict to isometries (resp. $G_{\HP}$-isomorphisms, isometries) on the strata.
\end{itemize}

A simple AdS or HP cone-manifold has \emph{spacelike} \emph{singularities} if every bounding hyperplane of each $\hat{\mathcal P}$ is spacelike.
\end{defi}

We refer to \cite{BLP} (see also \cite{thurston_shapes,CHK,McM}) for the general definition of hyperbolic cone-manifold, and to \cite{bbsads} for the 3-dimensional AdS case.  

Recall now Section \ref{sec rotations} about rotations, boosts and their infinitesimal counterpart. Given an $n$-dimensional simple projective cone-manifold $\mathcal{X}$, the holonomy of a meridian of an $(n-2)$-stratum is conjugated to:
\begin{itemize}
\item a rotation in $\Isom(\Hyp^n)$ if $\mathcal{X}$ is hyperbolic,
\item a boost in $\Isom(\AdS^n)$ if $\mathcal{X}$ is anti-de Sitter and $\Sigma$ is spacelike,
\item their infinitesimal version in $G_{\HP^n}$ if $\mathcal{X}$ is half-pipe and $\Sigma$ is spacelike.
\end{itemize}
In such cases, to each $(n-2)$-stratum is thus associated a number: the angle of rotation is called \emph{cone angle}, the angle of the boost (with opposite sign) is called \emph{magnitude}, and that of the infinitesimal rotation (again with opposite sign) is called \emph{infinitesimal cone angle}, respectively. {The sign convention is consistent with \cite{danciger,dancigertransition}, where roughly speaking the negative sign corresponds to the fact that  the singularity is a ``defect'' with respect to the non-singular case. We refer to \cite[Sections 2.5 and 4.2]{dancigertransition}} for more details.

The transitional 3-dimensional AdS cone-manifolds in \cite{danciger,dancigertransition} have spacelike singularities --- said ``with tachyons'' \cite{bbsads}, being each 1-stratum a spacelike geodesic, i.e. a ``particle faster than light''. In this case, the holonomy representation at a meridian is a AdS boost which fixes pointwise a spacelike geodesic.

The local structure of a point $p\in\mathcal X$ is determined by its link $\mathcal L_p$. We now briefly describe the situation in the three cases of interest for us.

\subsubsection{In hyperbolic geometry}
Given a point $x\in\Hyp^n$, 
{we have} $\mathrm{Stab}_{\Hyp^n}(x)\cong\O(n)$, and the link of $x$ can be identified to the round sphere. Simple hyperbolic cone-manifolds can be defined as manifolds locally modelled on the hyperbolic cone \cite{BH} over a spherical cone-manifold one dimension less, which is the double of a spherical simplex.

\subsubsection{In anti-de Sitter geometry}
The analogous geometry for $\AdS^n$ has been introduced in \cite{bbsads}, where it is called \emph{HS geometry}. By identifying $T_x\AdS^n$ with $\R^{1,n-1}$, the link of a point $x\in\AdS^n$ is called $\HS^{n-1}$, and is a projective $(n-1)$-sphere identified with the space of rays in $\R^{1,n-1}$. The stabiliser $\mathrm{Stab}_{\AdS^n}(x)$ is identified to $\O(1,n-1)$. The sphere $\HS^{n-1}$ is partitioned into (see Figure \ref{fig:transition_stabilizers}--right):

\begin{itemize}
\item the region of timelike rays, which corresponds to two copies of $\Hyp^{n-1}$;
\item  the region of spacelike rays, which is a copy of de Sitter space $\dS^{n-1}$;
\item the set of lightlike rays, which is the common boundary of the two latter regions and is topologically the disjoint union of two $(n-2)$-spheres.
\end{itemize}

An \emph{HS-structure} is by definition an $(\O(1,n-1),\HS^{n-1})$-structure. An HS manifold is thus partitioned into a hyperbolic region, a de Sitter region, and a null locus. In analogy with the hyperboic case, simple AdS cone-manifolds are locally modelled on the AdS suspension \cite{bbsads} over the double of an HS simplex of one dimension less. In the AdS case with spacelike singularities, the facets of the HS simplex are contained in the de Sitter region of $\HS^{n-1}$ and are spacelike.

\begin{figure}
\includegraphics[scale=.26]{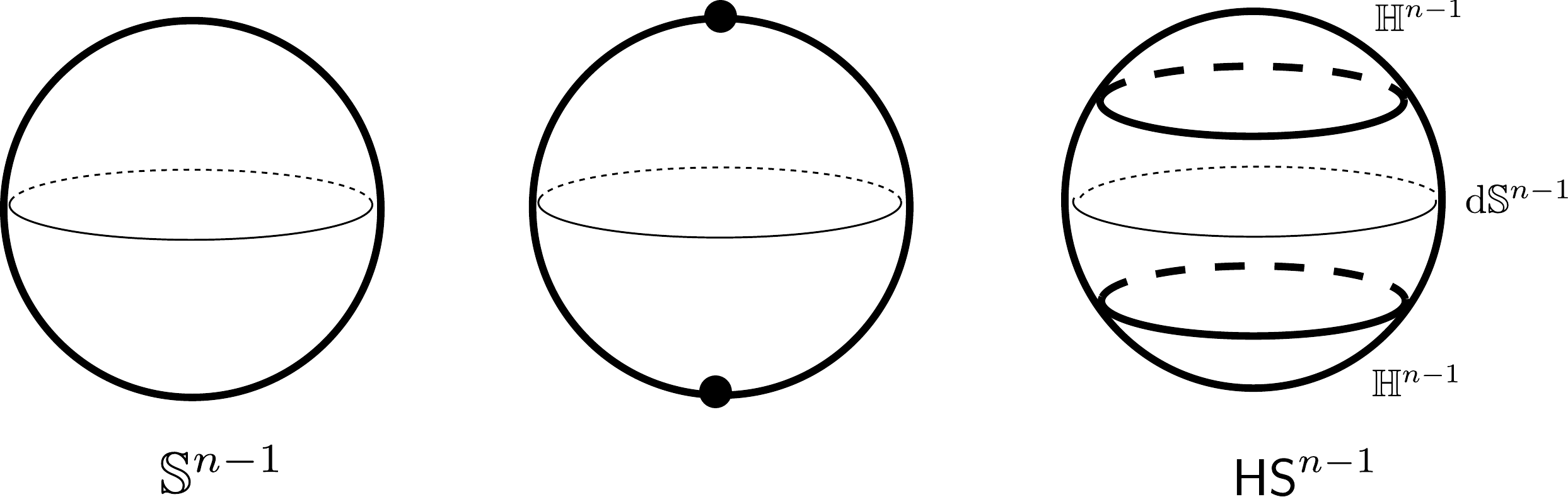}
\caption[The link of a point in a hyperbolic, half-pipe, or AdS manifold.]{\footnotesize The link of a point in a hyperbolic, half-pipe, or anti-de Sitter $n$-manifold. In the hyperbolic case (left), we have the round sphere $\mathbb S^{n-1}$. In the AdS case (right), the link is called $\HS^{n-1}$, and is subdivided into two timelike regions (copies of $\Hyp^{n-1}$), one spacelike region (copy of the de Sitter space $\dS^{n-1}$), and two lightlike $(n-2)$-spheres. In the HP case (centre), we have just a sphere with two marked antipodal points, corresponding to the degenerate direction. In the transition from AdS to hyperbolic geometry the two timelike regions shrink to points and then disappear.} \label{fig:transition_stabilizers}
\end{figure}

\subsubsection{In half-pipe geometry}

The stabiliser in $G_{\HP^n}$ of a point $x\in\HP^n$ is isomorphic to $\Isom(\R^{n-1})\times{\Z/2\Z}$. To see this, recall that by the usual duality (Section \ref{sec motivations hp}), the stabiliser of a point in $\HP^n$ is the same as the stabiliser of a spacelike hyperplane $P$ in $\Isom(\R^{1,n-1})$. The ${\Z/2\Z}$ factor is generated by a reflection in $P$, while  $\Isom(\R^{n-1})$ corresponds to Euclidean isometries of $H$ which extend to $\R^{1,n-1}$ by preserving setwise each component of $\R^{1,n-1}\smallsetminus P$. The link of a point in $\HP^n$ is thus endowed with a $(G,X)$-structure, where $X$ is an $(n-1)$-sphere identified to the set of rays in $T_x \HP^n$, and $G=\mathrm{Stab}_{\HP^n}(x)$ is the stabiliser of $x$ as described above. Note that such a $(G,X)$-structure has some distinguished points, which correspond to the degenerate rays exiting from $x$. These are well-defined since they are preserved by the action of $G$ (see Figure \ref{fig:transition_stabilizers}--centre).

Analogously to the hyperbolic and AdS case, one can visualise the link of a point in a simple half-pipe cone-manifold as the double of a simplex in the space of rays in $T_x \HP^n$. If the singularity $\Sigma$ is spacelike, then the facets of such simplex are spacelike.

\begin{remark} \label{rem: AdS-cone}
We omit the details here, but (similarly to Remark \ref{rem:trans-link-ideal}) it can be seen directly that the rescaled limits of the point stabilisers $\mathrm{Stab}_{\Hyp^n}(x)$ and $\mathrm{Stab}_{\AdS^n}(x)$ are point stabilisers in half-pipe geometry. Hence a geometric transition from hyperbolic to AdS geometry on simple projective cone-manifolds induces a geometric transition from spherical to HS cone structures, going through the analogous structure for half-pipe geometry. This can be visualised in Figure \ref{fig:transition_stabilizers} for non-singular points, and in Figures \ref{fig:transition link 1}, \ref{fig:transition link 2}, \ref{fig:transition link 3} for singular points in dimension four (as in Theorem \ref{teo: main}).
\end{remark}

\section{Warm up in dimension three}\label{sec:warm-up}

We are ready to build explicitly some examples of geometric transition. Section \ref{sec:warm-up} is a warm up in dimension three, while in Section \ref{sec 4dim} we prove Theorem \ref{teo: main}.

In this section, we describe two examples of 3-dimensional geometric transition from hyperbolic to an anti-de Sitter structures, going through a half-pipe structure. These will serve as a toy model for the 4-dimensional geometric transition constructed in Section \ref{sec 4dim}. In contrast with the deformations studied in \cite{dancigertransition} and \cite{dancigerideal}, where the 3-manifold is closed and the singular locus is a knot, our examples are cusped and the singular locus is either a link or a trivalent graph. We will not provide all the proofs in this section, since they will actually follow from the results we prove in dimension four; the reader can also see the survey paper \cite{TSG} for more details on some examples in dimension two and three.

\subsection{Singularity along a link}

\begin{figure}
\includegraphics[scale=.4]{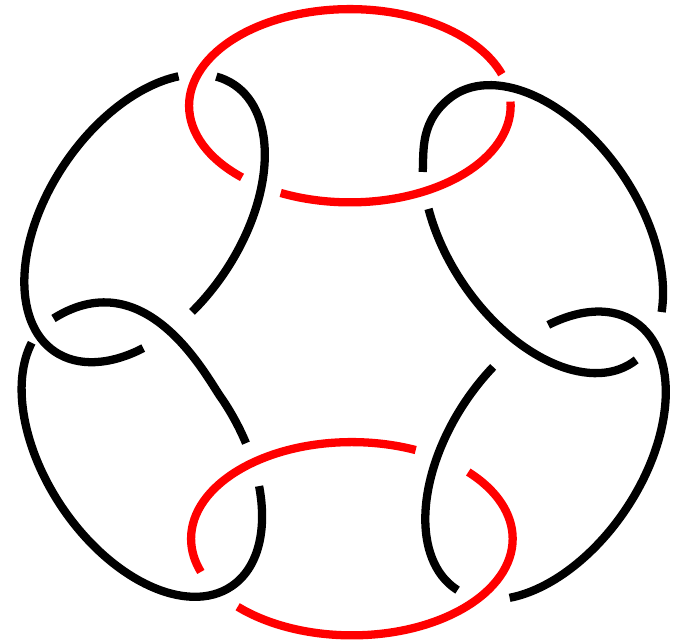}
\caption[The minimally twisted 6-chain link.]{\footnotesize The minimally twisted 6-chain link in the 3-sphere. Its complement $\mathcal{M}$ is hyperbolic, and can be tessellated by four ideal right-angled octahedra. By 0-surgery on the two red components, we get a Dehn filling ${\mathcal{X}}$ of $\mathcal{M}$ homeomorphic to  $\mathcal{S}_{0,4}\times S^1$, where $\mathcal{S}_{0,4}$ is a 4-times punctured sphere.}
\label{fig:6-chain-link}
\end{figure}

Consider an ideal right-angled octahedron ${\mathcal O}_0\subset\Hyp^3$, with the natural colouring of its facets in black and white in the chequerboard fashion. By doubling ${\mathcal O}_0$ along its white faces, and then doubling the resulting manifold with boundary, we get \cite[Figure 2]{KM} a well-known complete, finite-volume, hyperbolic 3-manifold $\mathcal{M}$ homeomorphic to the complement in $S^3$ of the minimally twisted 6-chain link (see Figure \ref{fig:6-chain-link}).
Note that $\mathcal{M}$ is tessellated by four copies of the octahedron ${\mathcal O}_0$.

Let us now deform ${\mathcal O}_0$ by a path $\theta\mapsto {\mathcal O}_\theta\subset\Hyp^3$ of polyhedra described in Figure \ref{fig:facet_L}, where the two red edges have varying dihedral angle $\theta$ and all the remaining edges are right-angled, while the white dots represent ideal vertices. Geometric models of this deformation are shown in Figure \ref{fig:collapse oct}. The polyhedron ${\mathcal O}_\theta$ exists for all $\theta\in(0,\pi)$ by Andreev's Theorem \cite{A1,A2}. As $\theta\to0$, the red edges shrink and go to infinity, and we have the original octahedron ${\mathcal O}_0$.

\begin{figure}
\includegraphics[scale=1]{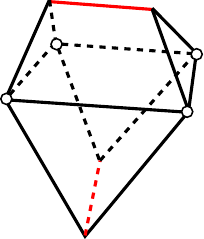}
\caption[A facet $\mathcal F_{\l X}$ of $\mathcal P_t$: Dehn filling the octahedron.]{\footnotesize The polyhedron $\mathcal O_\theta\subset\Hyp^3$. The white dots represent ideal vertices, the black edges are right-angled, and the red edges have dihedral angle $\theta\in(0,\pi)$. As $\theta\to0$, the red edges disappear, and we have the original ideal octahedron $\mathcal O_0$. As $\theta\to\pi$, the polyhedron collapses to the horizontal ideal quadrilateral ${\mathcal Q}$. By rescaling and continuing the path, we have similarly an AdS polyhedron with the same combinatorics and the same convention on the dihedral angles, where now the quadrliateral faces are spacelike, and the triangular ones are timelike. The polyhedron $\mathcal O_\theta$ (and its AdS version) is isometric to a facet ${\mathcal F}_{\l X}$, $\l X\in\{\l A,\ldots,\l F\}$, of the 4-polytope $\mathcal P_t$ of Section \ref{sec 4dim}. The quadrilateral faces are the ridges ${\mathcal R}_{\l X\p i}$ of $\mathcal P_t$, while the triangular faces are the ridges of type ${\mathcal R}_{\l X\m i}$ (see the end of Section \ref{sec: combinatorics} for the notation).}
\label{fig:facet_L}
\end{figure}

\begin{figure}
\centering
\begin{minipage}[c]{.25\textwidth}
\centering
\includegraphics[scale=0.2]{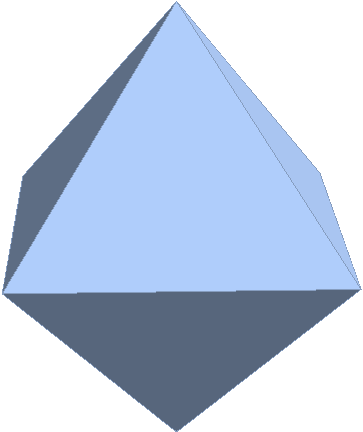}
\end{minipage}%
\begin{minipage}[c]{.25\textwidth}
\centering
\includegraphics[scale=0.2]{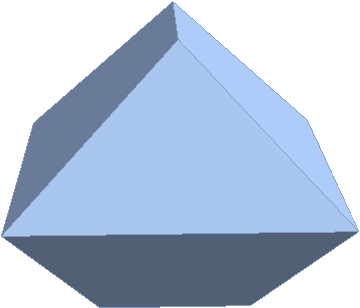}
\end{minipage}%
\begin{minipage}[c]{.25\textwidth}
\centering
\includegraphics[scale=0.2]{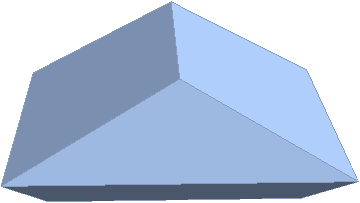}
\end{minipage}
\begin{minipage}[c]{.2\textwidth}
\vspace{0.35cm}
\centering
\includegraphics[scale=0.05]{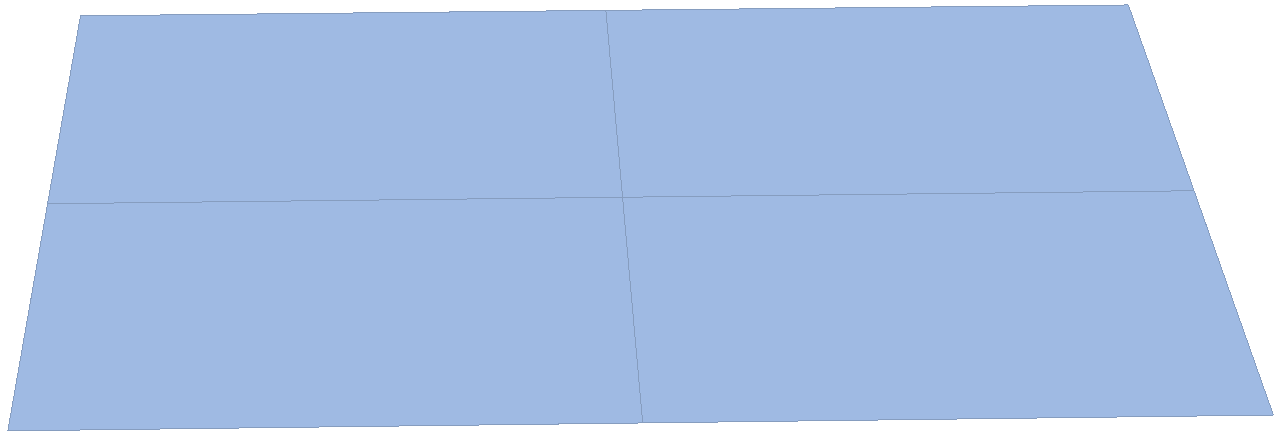}
\end{minipage}%
\caption[The collapse of $\mathcal F_{\l X}$.]{\footnotesize A movie of the collapse of the polyhedron ${\mathcal O}_\theta$ in an affine chart (Klein model of $\Hyp^3$), from the ideal octahedron $\mathcal O_0$ to the ideal quadrilateral $\mathcal Q$.} \label{fig:collapse oct}
\end{figure}

More concretely, ${\mathcal O}_\theta$ can be defined as the intersection in $\SP^3$ of the half-spaces in Table \ref{table:link}, for $t\in(0,1)$, using the notation introduced in Section \ref{subsec:proj sphere}. It can be checked directly that the correct orthogonality relations hold, and that the relation between the angle $\theta$  and the parameter $t\in(0,1)$ (by applying Lemma \ref{lemma angle hyp} to the first and third vector in the left column, for instance) is given by:
$$\cos\theta_t=\frac{3t^2-1}{1+t^2}~.$$

\begin{table}[htb]
\begin{eqnarray*}
\left( -|t|:-\sqrt{2}|t|:0:-1 \right) , & &
\left( -1:-\sqrt{2}:0:+t \right),\\
\left( -|t|:0:-\sqrt{2}|t|:+1 \right),& &
\left( -1:0:-\sqrt{2}:-t \right),\\
\left( -|t|:+\sqrt{2}|t|:0:-1 \right),& &
\left( -1:+\sqrt{2}:0:+t \right),\\
\left( -|t|:0:+\sqrt{2}|t|:+1 \right),& &
\left( -1:0:+\sqrt{2}:-t \right).\\
\end{eqnarray*}
\caption{\footnotesize The half-spaces defining the deformation of the right-angled ideal octahedron, expressed as elements of $\SP^{3,*}$. When $t<0$, we have an AdS polyhedron.}\label{table:link}
\end{table}

By deforming simultaneously in this way each copy of ${\mathcal O}_0$ in $\mathcal{M}$, we get a \emph{Dehn filling} ${\mathcal{X}}$ of the hyperbolic 3-manifold $\mathcal{M}$, i.e. ${\mathcal{X}}\smallsetminus \Sigma$ is homeomorphic to $\mathcal{M}$ for a link $\Sigma\subset{\mathcal{X}}$; see Figure \ref{fig:6-chain-link}. (It is standard to check that $\mathcal{X}$ is obtained from the link complement $\mathcal{M}$ by 0-surgery along the two red components in the picture. However, we will not use this fact here.) Each component of $\Sigma$ is the double of a red edge of $\mathcal O_\theta$. Moreover, the deformation describes a path of hyperbolic cone-structures on ${\mathcal{X}}$ with cone angles $2\theta$ along $\Sigma$, converging to the hyperbolic manifold $\mathcal{M}$ as $\theta\to0$.

The path of polyhedra is arranged in such a way that the four ideal vertices of ${\mathcal O}_\theta$ stay fixed, and belong to $\partial\Hyp^2\subset\partial\Hyp^3$ for the fixed hyperplane $\Hyp^2=\{x_3=0\}\subset\Hyp^3$ --- this arrangement is indeed used in Table \ref{table:link} and in Figure \ref{fig:collapse oct}. In particular, we have a fixed ideal quadrilateral ${\mathcal Q}={\mathcal O}_\theta\cap\Hyp^2$. As $\theta\to\pi$, the polyhedron ${\mathcal O}_\theta$ collapses to the polygon ${\mathcal Q}$.
The corresponding cone-manifolds collapse to the hyperbolic four-punctured sphere $\mathcal{S}_{0,4}$ obtained by doubling ${\mathcal Q}$. Note that there is a homeomorphism ${\mathcal O}_\theta\to {\mathcal Q}\times[-1,1]$ wich sends the quadrilateral faces to ${\mathcal Q}\times\{1,-1\}$ and the triangular faces to $\partial {\mathcal Q}\times[-1,1]$. It follows that ${\mathcal{X}}$ is homeomorphic to $\mathcal{S}_{0,4}\times S^1$. (In particular, ${\mathcal{X}}$ does not admit any complete hyperbolic structure.)

 The reader can check that when $t\in(-1,0)$, the polyhedron defined in Table \ref{table:link} is anti-de Sitter, its quadrliateral faces are spacelike, and the triangular ones are timelike. All the edges are right-angled, with the exception of the red ones (which are spacelike). By rescaling the polyhedron in the direction of collapse, i.e. orthogonally to the plane $\Hyp^2\subset\Hyp^3$, we get transition from hyperbolic to AdS polyhedra with constant combinatorics. In the limit half-pipe polyhedron, the triangular faces are degenerate. The path of rescaled polyhedra is easily computed by applying Lemma \ref{lemma rescale halfspace} and pictured (in an affine chart) in Figure \ref{fig:rescale oct}.

\begin{figure}
\centering
\begin{minipage}[c]{.33\textwidth}
\centering
\includegraphics[scale=0.3]{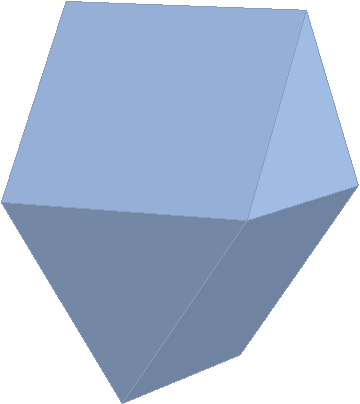}
\end{minipage}%
\begin{minipage}[c]{.33\textwidth}
\centering
\includegraphics[scale=0.28]{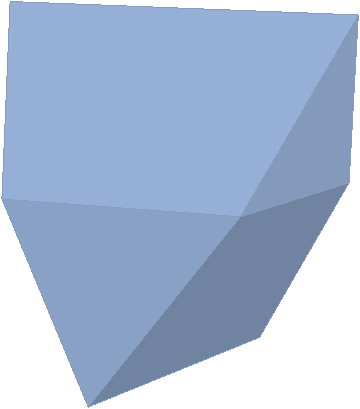}
\end{minipage}%
\begin{minipage}[c]{.33\textwidth}
\centering
\includegraphics[scale=0.33]{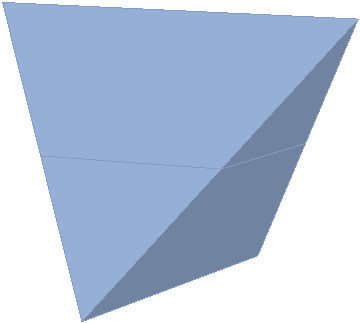}
\end{minipage}

\caption[Rescaling $\mathcal F_{\l X}$ when $t\leq 0$.]{\footnotesize In an affine chart, the rescaled path of polyhedra for $t\leq 0$. In the left figure (t=0), the polyhedron is in half-pipe space and the triangular faces are degenerate (vertical). In the middle picture, an AdS polyhedron with timelike triangular faces and spacelike quadrilateral faces. When $t=-1$ (in the right), the faces become lightlike and the polyhedron is inscribed in the one-sheeted hyperboloid which is the boundary at infinity of $\AdS^3$.}
\label{fig:rescale oct}
\end{figure}

The essential point to prove the transition at the level of geometric structures on $\mathcal M$ is to show that for every plane in the list of Table \ref{table:link} (which depends on the parameter $t$), if $r_t$ is the hyperbolic (for $t>0$) or AdS (for $t<0$) reflection in the given plane, then $\r_t r_t\r_t^{-1}$ converges to a half-pipe reflection when $t\to 0$ (with the same limit for $t\to0^+$ and $t\to0^-$). This essentially shows the convergence at the level of holonomy representations, since the holonomy of any element of $\pi_1(\mathcal M)$ is obtained by composition of a finite number of such reflections. This point is certainly non-trivial in general since, as we explained in Section \ref{remark:reflections}, a {(degenerate)} plane in $\HP^3$ does not determine uniquely a half-pipe reflection.
 
Although all the previous statements can be verified directly, we avoid detailed computations, since everything follows from the fact that, as $\theta\geq\frac\pi2$, the polyhedron ${\mathcal O}_\theta$ is a facet of the 4-polytope introduced in Section \ref{sec 4dim} (see also Remark \ref{rem:24-cell} for the case $\theta<\frac\pi2$), and meets the adjacent facets of the 4-polytope orthogonally (see Propositions \ref{prop: const-combinatorics}, \ref{prop: facets}, \ref{prop: hyperplanes polytope} and \ref{prop: ridges} in the sequel).

In particular, we get (all the details follow from the proof of Theorem \ref{teo: main} in Section \ref{sec 4dim}):

\begin{prop} \label{prop:example link}
There exists a $C^1$ family $\{\sigma_t\}_{t\in\left(-1,1\right]}$ of simple projective cone-manifold structures on the 3-manifold $\mathcal X=\mathcal{S}_{0,4}\times S^1$, singular along a link $\Sigma$ with two components, such that $\sigma_t$ is conjugated to a cusped, finite-volume
\begin{itemize}
\item hyperbolic cone structure with decreasing cone angles $2\theta_t\in(0,2\pi)$ as $t>0$,
\item half-pipe structure with spacelike singularity as $t=0$,
\item anti-de Sitter structure with spacelike singularity of increasing magnitude $\varphi_t\in(-\infty,0)$ as $t<0$.
\end{itemize}
As $t\to1$, we have {$2\theta_t\to0$} and the hyperbolic structures on $\mathcal{M}={\mathcal X}\smallsetminus\Sigma$ converge to the complete one. As $t\to0^+$ (resp. $t\to0^-$), we have {$2\theta_t\to2\pi$} (resp. $\varphi_t\to0$) and the hyperbolic (resp. AdS) structures degenerate to the hyperbolic structure of $\mathcal{S}_{0,4}=\mathrm{Double}(\mathcal Q)$.
\end{prop}

\subsection{Singularity along a graph} \label{subsec:theta graph}

Let now $\mathcal{S}_{0,3}$ be the hyperbolic thrice punctured sphere. Similarly to the previous section, we now provide an example of 3-dimensional transition where the singular locus is a theta-graph:

\begin{prop} \label{prop:example theta graph}
There exists a $C^1$ family $\{\sigma_s\}_{s\in\left(-1,\epsilon\right]}$ of simple projective cone-manifold structures on the 3-manifold $\mathcal X=\mathcal{S}_{0,3}\times S^1$, singular along a theta-graph $\Sigma$, such that $\sigma_s$ is conjugated to a cusped, finite-volume
\begin{itemize}
\item hyperbolic orbifold structure with cone angles $\pi$
as $s=\epsilon>0$,
\item hyperbolic cone structure with decreasing cone angles {$\vartheta_s\in[\pi,2\pi)$} as $s>0$,
\item half-pipe structure with spacelike singularity as $s=0$,
\item anti-de Sitter structure with spacelike singularity of increasing magnitude $\phi_s<0$ as $s<0$.
\end{itemize}
As $s\to0^+$ (resp. $s\to0^-$), we have $\vartheta_s\to2\pi$ (resp. $\phi_s\to0$) and the hyperbolic (resp. AdS) structures on $\mathcal X\smallsetminus\Sigma$ degenerate to the complete hyperbolic structure of $\mathcal{S}_{0,3}$.
\end{prop}

\begin{figure}
\includegraphics[scale=1]{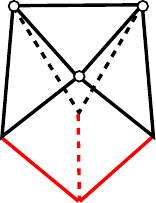}
\caption[A facet ${\mathcal F}_{\m i}$ of ${\mathcal P}_t$.]{\footnotesize A facet ${\mathcal F}_{\m i}$, $\m i\in\{\m0,\ldots,\m7\}$, of the 4-polytope of Section \ref{sec 4dim} (see the end of Section \ref{sec: combinatorics} for the notation). The white dots represent ideal vertices, the black edges are right-angled, and the red edges have dihedral angle $\theta\in[\frac\pi2,\pi)$. As $\theta\to\pi$, the polyhedron collapses to the horizontal ideal triangle. By rescaling and continuing the path, we have similarly an AdS polyhedron with the same combinatorics and the same convention on the dihedral angles, where now the three vertical triangular faces (ridges of the 4-polytope of type ${\mathcal R}_{\m i\l X}$) are timelike, while the horizontal ideal triangle (the ridge ${\mathcal R}_{\m i\p i}$ of the 4-polytope) and the remaining three quadrilateral faces (ridges of type ${\mathcal R}_{\m i\p j}$) are spacelike.}
\label{fig:facet_N}
\end{figure}

Consider indeed the polyhedron ${\mathcal F}_\vartheta$ in Figure \ref{fig:facet_N}, where the red edges have varying dihedral angle ${\tfrac\vartheta2}\in[\frac\pi2,\pi)$ and the black edges are right-angled. Again, the path $\vartheta\mapsto {\mathcal F}_\vartheta$ can be arranged so that the three ideal vertices stay fixed. As $\vartheta\to2\pi$, the polyhedron collapses to the horizontal ideal triangle ${\mathcal T}$, which is a face of ${\mathcal F}_\vartheta$ for all $\vartheta_s\in[\pi,2\pi)$. Figure \ref{fig:collapse theta} gives a geometric picture.

\begin{figure}
\centering
\begin{minipage}[c]{.25\textwidth}
\centering
\includegraphics[scale=0.2]{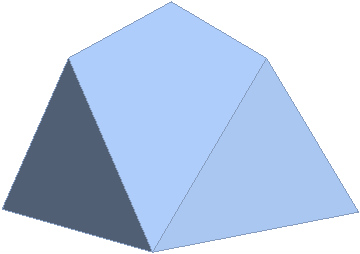}
\end{minipage}%
\begin{minipage}[c]{.25\textwidth}
\centering
\vspace{0.3cm}
\includegraphics[scale=0.2]{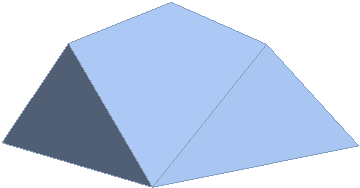}
\end{minipage}%
\begin{minipage}[c]{.25\textwidth}
\vspace{0.7cm}
\centering
\includegraphics[scale=0.2]{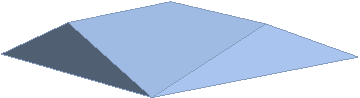}
\end{minipage}
\begin{minipage}[c]{.2\textwidth}
\vspace{1cm}
\centering
\includegraphics[scale=0.2]{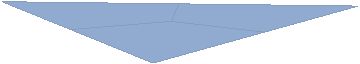}
\end{minipage}%
\caption[The collapse of $\mathcal F_{\m i}$.]{\footnotesize The collapse of the polyhedron ${\mathcal F}_\vartheta$ of Section \ref{subsec:theta graph}, in the Klein model of $\Hyp^3$. The leftmost polyhedron turns out to be also the rescaled limit, inside $\HP^3$.}
\label{fig:collapse theta}
\end{figure}

Similarly to the previous section, for $s<0$ we have a path of AdS polyhedra collapsing to ${\mathcal T}$, and by rescaling opportunely the two paths are joined by a half-pipe polyhedron with the same combinatorics. Again, the black edges are constantly right-angled. In the AdS side, the three vertical triangular faces are timelike, while the remaining three quadrilateral faces (and the ideal triangle ${\mathcal T}$) are spacelike. In the half-pipe limit, the triangular faces are degenerate.

By doubling ${\mathcal F}_\vartheta$ along the three vertical triangular faces, and then doubling the resulting manifold with boundary, we get the desired transition on a 3-manifold homeomorphic to $\mathcal{S}_{0,3}\times S^1$ (where $\mathcal{S}_{0,3}$ is the double of ${\mathcal T}$), with singularity along a theta-graph (which is the double of the red locus in ${\mathcal F}_\vartheta$).

Again, all these statements can be proved directly, but will follow from the fact that ${\mathcal F}_\vartheta$ is isometric to a facet of the 4-dimensional polytope of Section \ref{sec 4dim}. 

\section{Transition in dimension four} \label{sec 4dim}

In this section we prove Theorem \ref{teo: main}, giving as a byproduct also rigours to the assertions of the previous section about 3-dimensional transition.

\subsection{A deforming polytope}

Recall Section \ref{sec:half-spaces} about half-spaces of the projective sphere $\SP^n$, and Section \ref{section polytopes} about projective polytopes. We define
$$\overline {\mathcal P}_t\subset\SP^4$$
to be the intersection of the 22 half-spaces listed in Table \ref{table:walls}, depending on the time parameter
$$t\in I=\left(-1,\tfrac1{\sqrt3}\right].$$
We set
$$I^-=\Big(-1,0\Big),\quad I^+=\left(0,\tfrac1{\sqrt3}\right],$$
and
$${\mathcal P}_t=
\begin{cases}
\overline {\mathcal P}_t\cap\Hyp^4\mathrm{\ when\ }t\in I^+\cup\{0\},\\
\overline {\mathcal P}_t\cap\AdS^4\mathrm{\ when\ }t\in I^-\cup\{0\}.
\end{cases}$$
We will see that this last definition is well posed as $t=0$.
It will be clear later that $\overline {\mathcal P}_t$ is the closure in $\SP^4$ of ${\mathcal P}_t$.

\begin{table} 
\begin{eqnarray*}
\p{0} = \left( -\sqrt{2}\ |t|:+|t|:+|t|:+|t|:+1 \right) , & &
\m{0} = \left( -\sqrt{2}:+1:+1:+1:-t \right),\\
\p{1} = \left( -\sqrt{2}\ |t|:+|t|:-|t|:+|t|:-1 \right),& &
\m{1} = \left( -\sqrt{2}:+1:-1:+1:+t \right),\\
\p{2} = \left( -\sqrt{2}\ |t|:+|t|:-|t|:-|t|:+1 \right),& &
\m{2} = \left( -\sqrt{2}:+1:-1:-1:-t \right),\\
\p{3} = \left( -\sqrt{2}\ |t|:+|t|:+|t|:-|t|:-1 \right),& &
\m{3} = \left( -\sqrt{2}:+1:+1:-1:+t \right),\\
\p{4} = \left( -\sqrt{2}\ |t|:-|t|:+|t|:-|t|:+1 \right),& &
\m{4} = \left( -\sqrt{2}:-1:+1:-1:-t \right),\\
\p{5} = \left( -\sqrt{2}\ |t|:-|t|:+|t|:+|t|:-1 \right),& &
\m{5} = \left( -\sqrt{2}:-1:+1:+1:+t \right),\\
\p{6} = \left( -\sqrt{2}\ |t|:-|t|:-|t|:+|t|:+1 \right),& &
\m{6} = \left( -\sqrt{2}:-1:-1:+1:-t \right),\\
\p{7} = \left( -\sqrt{2}\ |t|:-|t|:-|t|:-|t|:-1 \right),& &
\m{7} = \left( -\sqrt{2}:-1:-1:-1:+t \right),\\
\l{A} = \left( -1:+\sqrt{2}:0:0:0 \right),& &
\l{B} = \left( -1:0:+\sqrt{2}:0:0 \right),\\
\l{C} = \left( -1:0:0:+\sqrt{2}:0 \right),& &
\l{D} = \left( -1:0:0:-\sqrt{2}:0 \right),\\
\l{E} = \left( -1:0:-\sqrt{2}:0:0 \right),& &
\l{F} = \left( -1:-\sqrt{2}:0:0:0 \right).
\end{eqnarray*}
\caption{\footnotesize The half-spaces in $\SP^4$ that define the projective polytope $\overline {\mathcal P}_t$ are given by these elements of $\SP^{4,*}$ and denoted by the same symbols. We will often omit the dependence in $t$ in the symbols $\p i$, $\m i$ and $\l X$, to simplify the notation.}\label{table:walls}
\end{table}

When $t\in I^+$, each element $(\alpha)\in\SP^{4,*}$ in the list of Table \ref{table:walls} satisfies $q_1(\alpha)>0$, and thus by Lemma \ref{lemma hyperplane hyp} defines a half-space of $\Hyp^4$. When $t\in I^+$, the set ${\mathcal P}_t\subset\Hyp^4$ is indeed a hyperbolic 4-polytope, first introduced in \cite{KS} and then studied in \cite{MR}. The set of half-spaces of Table \ref{table:walls} is minimal for $t\neq 0$; in other words, none of them contains any other --- this is shown in \cite[Proposition 3.3]{MR} for $t\in I^+$ and holds for $t\in I^-$ by a straightforward adaptation of the proof.

Throughout the remainder of the paper, we will denote
$$\iota(\Hyp^3)=\{x_4=0\}\subset\Hyp^4,\HP^4,\AdS^4,$$ 
simply as $\Hyp^3$, where the isometric embedding $\iota$ is defined by \eqref{eq: iota} in Section \ref{sec recipe}. As $t\to0^+$, the polytope ${\mathcal P}_t$ collapses to a 3-dimensional polytope in the hyperplane $\Hyp^3\subset\Hyp^4$.

\begin{prop}[\cite{KS,MR}] \label{prop:polytope_hyp}
When $t\in I^+$, the set ${\mathcal P}_t$ is a finite-volume hyperbolic 4-polytope, whose combinatorics does not depend on $t\in I^+$. The set ${\mathcal P}_0$ is a finite-volume 3-polytope in $\Hyp^3\subset\Hyp^4$.
\end{prop}

In the next sections, we will show that  when $t\in I^-$ the behaviour of ${\mathcal P}_t$ is analogue, in the AdS setting, to that when $t\in I^+$ given by Proposition \ref{prop:polytope_hyp}.

\begin{remark} \label{rem:24-cell} \label{rmk:orthogonality preserved}
The path $t\mapsto {\mathcal P}_t\subset\Hyp^4$ of hyperbolic polytopes extends beyond $t=\frac1{\sqrt3}$ to all $t\in(0,1]$. When $t>\frac1{\sqrt3}$ the combinatorics changes a few times, and moreover when $t>\frac1{\sqrt2}$ the volume becomes infinite. This whole path of hyperbolic polytopes was discovered by Kerckhoff and Storm \cite{KS}.

The smaller polytope ${\mathcal P}_1\cap\l G\cap\l H$ is the ideal right-angled 24-cell, where the half-spaces $\l G=(-1:0:0:0:-\sqrt2)$ and $\l H=(-1:0:0:0:\sqrt2)$ correspond to two opposite facets. The partition
$$\{\p0,\ldots,\p7\}\sqcup\{\m0,\ldots,\m7\}\sqcup\{\l A,\ldots,\l H\}$$
gives the standard 3-colouring of the facets of the 24-cell: if two hyperplanes belong to the same octet, then they are disjoint in $\Hyp^4$.

Moreover, it is easily checked that all the orthogonality relations that occur when $t=1$ between any two vectors in Table \ref{table:walls} are maintained for all $t\in(-1,1]$, with respect to the bilinear form of signature $-++++$ when $t\in[0,1)$, and $-+++-$ when $t\in(-1,0]$.
\end{remark}

\subsection{Combinatorics of the polytope} \label{sec: combinatorics}

The main goals here are to prove that when $t\in I^-$ the set ${\mathcal P}_t\subset\AdS^4$ is a deforming AdS 4-polytope, that the combinatorics of ${\mathcal P}_t$ is independent on $t\in I\smallsetminus\{0\}$, and that the rescaled limit (see Section \ref{sec:rescaled_limit})
$$\lim_{t\to0}\mathfrak r_{|t|}({\mathcal P}_t)\subset\HP^4$$
is a half-pipe polytope with the same combinatorics of ${\mathcal P}_t$ with $t\neq0$. In particular, we will show that $\{\mathfrak r_{|t|}({\mathcal P}_t)\}_{t\in I}$ is a path of projective 4-polytopes (extended to $t=0$ by continuity) whose combinatorics is constant. For easiness of the reader, we record the list of rescaled half-spaces defining $\mathfrak r_{|t|}({\mathcal P}_t)$ in Table \ref{table:walls rescaled}, calculated using Lemma \ref{lemma rescale halfspace}.

All these facts will follow from the following proposition (recall that we denote by $\Aff^4$ the affine chart $\{x_0>0\}\subset\SP^4$):

\begin{prop}\label{prop: const-combinatorics}
The set $\mathfrak r_{|t|}(\overline {\mathcal P}_t)$ is contained in $\overline{\mathbb X^4_t}\cap\Aff^4$, and is a 4-polytope whose combinatorics does not depend on $t\in I$.
\end{prop}

\begin{remark} \label{rem: transition interior polytope}
Recall Definition \ref{defi transition} of geometric transition. Proposition \ref{prop: const-combinatorics} implies that we already have a transition on the interior of ${\mathcal P}_t$.
\end{remark}

\begin{table} 
\begin{eqnarray*}
\r_{|t|}\p{0} = \left( -\sqrt{2}:+1:+1:+1:+1 \right) , & &
\r_{|t|}\m{0} = \left( -\sqrt{2}:+1:+1:+1:-t|t| \right),\\
\r_{|t|}\p{1} = \left( -\sqrt{2}:+1:-1:+1:-1 \right),& &
\r_{|t|}\m{1} = \left( -\sqrt{2}:+1:-1:+1:+t|t| \right),\\
\r_{|t|}\p{2} = \left( -\sqrt{2}:+1:-1:-1:+1 \right),& &
\r_{|t|}\m{2} = \left( -\sqrt{2}:+1:-1:-1:-t|t| \right),\\
\r_{|t|}\p{3} = \left( -\sqrt{2}:+1:+1:-1:-1 \right),& &
\r_{|t|}\m{3} = \left( -\sqrt{2}:+1:+1:-1:+t|t| \right),\\
\r_{|t|}\p{4} = \left( -\sqrt{2}:-1:+1:-1:+1 \right),& &
\r_{|t|}\m{4} = \left( -\sqrt{2}:-1:+1:-1:-t|t| \right),\\
\r_{|t|}\p{5} = \left( -\sqrt{2}:-1:+1:+1:-1 \right),& &
\r_{|t|}\m{5} = \left( -\sqrt{2}:-1:+1:+1:+t|t| \right),\\
\r_{|t|}\p{6} = \left( -\sqrt{2}:-1:-1:+1:+1 \right),& &
\r_{|t|}\m{6} = \left( -\sqrt{2}:-1:-1:+1:-t|t| \right),\\
\r_{|t|}\p{7} = \left( -\sqrt{2}:-1:-1:-1:-1 \right),& &
\r_{|t|}\m{7} = \left( -\sqrt{2}:-1:-1:-1:+t|t| \right),\\
\r_{|t|}\l{A} = \left( -1:+\sqrt{2}:0:0:0 \right),& &
\r_{|t|}\l{B} = \left( -1:0:+\sqrt{2}:0:0 \right),\\
\r_{|t|}\l{C} = \left( -1:0:0:+\sqrt{2}:0 \right),& &
\r_{|t|}\l{D} = \left( -1:0:0:-\sqrt{2}:0 \right),\\
\r_{|t|}\l{E} = \left( -1:0:-\sqrt{2}:0:0 \right),& &
\r_{|t|}\l{F} = \left( -1:-\sqrt{2}:0:0:0 \right).
\end{eqnarray*}
\caption{\footnotesize The half-spaces in $\SP^4$ defining $\mathfrak r_{|t|}({\mathcal P}_t)$, by a direct application of Lemma \ref{lemma rescale halfspace} to Table \ref{table:walls}.}\label{table:walls rescaled}
\end{table}

Before proving Proposition \ref{prop: const-combinatorics}, let us begin with a simple lemma.

\begin{lemma}
For all $t\in I$, the set $\mathfrak r_{|t|}(\overline {\mathcal P}_t)\subset\SP^4$ is a 4-polytope.
\end{lemma}

\begin{proof}
It sufficies to show that $\mathfrak r_{|t|}(\overline {\mathcal P}_t)$ has non-empty interior. From Table \ref{table:walls rescaled}, a small neighbourhood of the point $[1:0:0:0:0]\in\SP^4$ is contained in $\mathfrak r_{|t|}(\overline {\mathcal P}_t)$ because the first entry of each vector of Table \ref{table:walls rescaled} is negative.
\end{proof}

As in \cite{KS,MR}, we now descibe the symmetries of the polytope $\overline {\mathcal P}_t$ and of its rescaled $\mathfrak r_{|t|}(\overline {\mathcal P}_t)$ which are useful to reduce the number of computations. We refer to \cite[Section 4]{KS} and \cite[Section 3.2]{MR} for details in the hyperbolic case.

Let us introduce three auxiliary half-spaces $\l L$, $\l M$ and $\l N$, defined in Equation \eqref{table:walls_aux}. Observe that they are all left invariant by $\r_{|t|}$. 
\begin{equation} \label{table:walls_aux}
\begin{aligned} 
\l L  &=  \left( 0:-1:+1:0:0 \right), \\
\l M  &= \left( 0:0:-1:+1:0 \right), \\
\l N  &=\left( 0:0:-1:-1:0 \right).
\end{aligned}
\end{equation}
The following projective involutions of $\SP^4$
\begin{align*}
r_{\l L}\colon & [x_0:x_1:x_2:x_3:x_4] \longmapsto [x_0:x_2:x_1:x_3:x_4], \\
r_{\l M}\colon & [x_0:x_1:x_2:x_3:x_4] \longmapsto [x_0:x_1:x_3:x_2:x_4], \\
r_{\l N}\colon & [x_0:x_1:x_2:x_3:x_4] \longmapsto [x_0:x_1:-x_3:-x_2:x_4], \\
R\colon & [x_0:x_1:x_2:x_3:x_4] \longmapsto [x_0:x_1:x_2:-x_3:-x_4]
\end{align*}
commute with $\mathfrak r_{|t|}$, preserve the hyperplane $\Hyp^3$, and all belong to $\Isom(\Hyp^4)$, $G_{\HP^4}$ and $\Isom(\AdS^4)$. As the notation suggests, $r_{\l L},r_{\l M}$ and $r_{\l N}$ are reflections along the hyperplanes $\partial\l L$, $\partial \l M$ and $\partial \l N$, respectively. The map $R$, instead, is a rotation {(called ``roll symmetry'' in \cite{KS})}. The group $\langle r_{\l L},r_{\l M},r_{\l N}\rangle<\Aut(\SP^4)$ is isomorphic to the symmetric group $\mathfrak S_4$ on 4 elements.

\begin{lemma} \label{lem:symmetry}
The maps $r_{\l L},r_{\l M},r_{\l N}$ and $R$ preserve $\overline {\mathcal P}_t$ and $\mathfrak r_{|t|}(\overline {\mathcal P}_t)$ for all $t\in I$.
Moreover, the set
$$\overline {\mathcal Q}_t=\p0\cap\m0\cap\p3\cap\m3\cap\l A\cap\l L\cap\l M\cap\l N\subset\SP^4$$
is a fundamental domain for the action of the group $\langle r_{\l L},r_{\l M},r_{\l N}\rangle<\Aut(\SP^4)$ on $\overline {\mathcal P}_t$. Moreover, $\mathfrak r_{|t|}(\overline {\mathcal Q}_t)$ is a fundamental domain for the action of the group $\langle r_{\l L},r_{\l M},r_{\l N}\rangle<\Aut(\SP^4)$ on $\mathfrak r_{|t|}(\overline {\mathcal P}_t)$.
\end{lemma}

\begin{proof}
This is already proved for $t\in I^+$ in  \cite[Section 4]{KS} and \cite[Section 3.2]{MR}. To conclude, it suffices to observe that the action of $\langle r_{\l L},r_{\l M},r_{\l N}\rangle$ on the set of vectors in Table \ref{table:walls} does not depend on $t\in I$. The second statement follows as a consequence, using that $\mathfrak r_{|t|}$ commutes with $r_{\l L},r_{\l M}$ and $r_{\l N}$.
\end{proof}

\begin{lemma} \label{lem:subset A^4}
For all $t\in I^-$, the sets $\overline {\mathcal P}_t$ and $\mathfrak r_{|t|}(\overline {\mathcal P}_t)$ are contained in the affine chart $\Aff^4$.
\end{lemma}

\begin{proof}
Since the maps $r_{\l L},r_{\l M},r_{\l N},\mathfrak r_{|t|}\in\Aut(\SP^4)$ preserve the affine chart $\Aff^4$, it sufficies to show that $\mathfrak r_{|t|}(\overline {\mathcal Q}_t)\subset\Aff^4$.
By looking at Table \ref{table:walls rescaled} and Equation \eqref{table:walls_aux}, $\mathfrak r_{|t|}(\overline {\mathcal Q}_t)$ is defined by the following inequalities:
\begin{eqnarray}
\label{eq:Qt_1} -\sqrt{2} x_{0} + x_{1} + x_{2} + x_{3} + x_{4} \leq 0,\quad -\sqrt{2} x_{0} + x_{1} + x_{2} - x_{3} - x_{4} \leq 0,\\%
\label{eq:Qt_2} - \sqrt{2} x_{0} + x_{1} + x_{2} + x_{3} + t^{2} x_{4} \leq 0,\quad - \sqrt{2} x_{0} + x_{1} + x_{2} - x_{3} -t^{2} x_{4} \leq 0,\\%
\label{eq:Qt_3} - x_{0} + \sqrt{2} x_{1} \leq 0,\quad -x_{1} + x_{2} \leq 0,\quad -x_{2} + x_{3} \leq 0,\quad -x_{2} - x_{3} \leq 0.
\end{eqnarray}
Suppose by contradiction that $x_0\leq0$.
By \eqref{eq:Qt_3}, we would also have $x_1,x_2,x_3\leq0$.
Together with the last inequality of \eqref{eq:Qt_3}, this gives $x_2=x_3=0$.
By the second inequality of \eqref{eq:Qt_3}, we have also $x_1=0$.
By the first inequality of \eqref{eq:Qt_3}, we have also $x_0=0$.
Substituting $x_0=x_1=x_2=x_3=0$ in \eqref{eq:Qt_1}, we have also $x_4=0$, and this is absurd.
\end{proof}

We will thus be free to use the affine coordinates $y_1,\ldots,y_4$ of $\Aff^4$, where $y_i=x_i/x_0$. Let us now analyse the vertices of $\overline {\mathcal P}_t$. 

\begin{lemma} \label{lem: vertices}
For all $t\in I$, the polytope $\mathfrak r_{|t|}(\overline {\mathcal P}_t)$ has 46 vertices, of which 12 belong to $\partial\mathbb X_t^4$ and 34 belong to $\mathbb X_t^4$.
\end{lemma}

\begin{proof}
This is already proven in \cite[Proposition 3.16]{MR} for $t\in I^+$, by applying to ${\mathcal Q}_t\subset\Hyp^4$ Vinberg's theory of acute-angled hyperbolic polytopes \cite{V} and then by letting the group $\langle r_{\l L},r_{\l M},r_{\l N}\rangle$ act. We cannot do the same for $t\in I^-\cup\{0\}$, being now in the AdS (or HP) setting, so in this case we proceed as follows.
\begin{enumerate}
\item For every $k=4,5,6,7,8$ and every set $\{\l H_1,\ldots,\l H_k\}$ of bounding hyperplanes of $\mathfrak r_{|t|}(\overline {\mathcal Q}_t)$, we consider the linear system in $\Aff^4$ defining $\bigcap_i \partial\l H_i$.
\item Every time such linear system has a unique solution, we check if the solution belongs to $\mathfrak r_{|t|}(\overline {\mathcal Q}_t)$.
\item We collect all such points, which are the 13 vertices of $\mathfrak r_{|t|}(\overline {\mathcal Q}_t)$.
\item We check that one vertex belongs to $\partial\mathbb X_t^4$, while the remaing 12 belong to $\mathbb X_t^4$.
\item We select the vertices of $\mathfrak r_{|t|}(\overline {\mathcal Q}_t)$ which are vertices of $\overline {\mathcal P}_t$.
\item We let the group $\langle r_{\l L},r_{\l M},r_{\l N}\rangle$ act on these latter points, to finally find all the vertices of $\mathfrak r_{|t|}(\overline {\mathcal P}_t)$.
\end{enumerate}
The final output is reported in Table \ref{table:vertices}. Although this procedure is very simple, the number of computations is terribly big, so we omit the complete proof. The details can be checked through a computer (see \cite{worksheet}). 
\end{proof}

\begin{table}
\begin{eqnarray*}
{\mathcal V}_{\p0\p3\m0\m3\l A\l L} =&  \left(\tfrac{\sqrt2}2,\ \tfrac{\sqrt2}2,\ 0,\ 0\right),\\
{\mathcal V}_{\p0\m0\l A\l M} =&  \left(\tfrac{\sqrt2}2,\ \tfrac{\sqrt2}4,\ \tfrac{\sqrt2}4,\ 0\right),\\
{\mathcal V}_{\p0\m0\l L\l M} =&  \left(\tfrac{\sqrt2}3,\ \tfrac{\sqrt2}3,\ \tfrac{\sqrt2}3,\ 0\right), \\
{\mathcal V}_{\p0\m3\l A\l N} =&  \left(\tfrac{\sqrt2}2,\ \tfrac{\sqrt2}4(t^2 + 1),\ -\tfrac{\sqrt2}4(t^2 + 1),\ \tfrac{\sqrt2}4\right),\\
{\mathcal V}_{\p0\m3\l L\l N} =&  \left(\sqrt2\tfrac{t^2 + 1}{t^2 + 3},\ \sqrt2\tfrac{t^2 + 1}{t^2 + 3},\ -\sqrt2\tfrac{t^2 + 1}{t^2 + 3},\ \tfrac{2\sqrt2}{t^2 + 3}\right),\\
{\mathcal V}_{\p0\l A\l M\l N} =&  \left(\tfrac{\sqrt2}2,\ 0,\ 0,\ \tfrac{\sqrt2}2\right),\\
{\mathcal V}_{\p0\l L\l M\l N} =&  \left(0,\ 0,\ 0,\ \sqrt2\right),\\
{\mathcal V}_{\p3\m 0\l A\l M} =&  \left(\tfrac{\sqrt2}2,\ \tfrac{\sqrt2}4(t^2 + 1),\ \tfrac{\sqrt2}4(t^2 + 1),\ -\tfrac{\sqrt2}2\right),\\
{\mathcal V}_{\p3\m0\l L\l M} =&  \left(\sqrt2\tfrac{t^2 + 1}{t^2 + 3},\ \sqrt2\tfrac{t^2 + 1}{t^2 + 3},\ \sqrt2\tfrac{t^2 + 1}{t^2 + 3},\ -\tfrac{2\sqrt2}{t^2 + 3}\right),\\
{\mathcal V}_{\p3\m3\l A\l N} =&  \left(\tfrac{\sqrt2}2,\ \tfrac{\sqrt2}4,\ -\tfrac{\sqrt2}4,\ 0\right),\\
{\mathcal V}_{\p3\m3\l L\l N} =&  \left(\tfrac{\sqrt2}3,\ \tfrac{\sqrt2}3,\ -\tfrac{\sqrt2}3,\ 0\right),\\
{\mathcal V}_{\p3\l A\l M\l N} =&  \left(\tfrac{\sqrt2}2,\ 0,\ 0,\ -\tfrac{\sqrt2}2\right),\\
{\mathcal V}_{\p3\l L\l M\l N} =&  \left(0,\ 0,\ 0,\ -\sqrt2\right).
\end{eqnarray*}

\vspace{.3cm}

\caption{\footnotesize The vertices of $\mathfrak r_{|t|}(\overline {\mathcal Q}_t)$ in affine coordinates.} \label{table:vertices}
\end{table} 

Note that the previous lemma implies that $\mathfrak r_{|t|}(\overline {\mathcal P}_t)\subset\overline{\mathbb X_t^4}\cap\Aff^4$ when $t\in I^+\cup\{0\}$, since in that case $\overline{\mathbb X_t^4}\subset\Aff^4$ is convex. We cannot directly conclude in the same way when $t\in I^-$, as $\overline{\mathbb X_t^4}\cap \Aff^4$ is not convex when $t<0$.

\begin{lemma}\label{lem:subset X^4_t}
For all $t\in I$, we have
$$\mathfrak r_{|t|}(\overline {\mathcal Q}_t)\cap\partial\mathbb X_t^4=\{[2:\sqrt2:\sqrt2:0:0]\}.$$
\end{lemma}

\begin{proof}
This is already proven in \cite{MR} when $t\in I^+$. So, let us assume that $t\in I^-\cup\{0\}$. In affine coordinates, \eqref{eq:Qt_1} and \eqref{eq:Qt_3} read as:
$$-\sqrt{2} + y_{1} + y_{2} + y_{3} + y_{4} \leq 0,\quad -\sqrt{2} + y_{1} + y_{2} - y_{3} - y_{4} \leq 0,$$
$$-y_2\leq y_3\leq y_2\leq y_1\leq \tfrac{\sqrt2}{2}.$$
By summing the first two equations and using the third, we get $y_1=y_2=\sqrt{2}/2$. This implies $y_3+y_4=0$. Now, the affine coordinates of a point in $\partial\mathbb X^4_t$ satisfy
$$y_1^2+y_2^2+y_3^2-t^2y_4^2=1.$$
This implies $y_3^2-t^2y_4^2=0$. Together with $y_3+y_4=0$, we get $y_3=y_4=0$ since $t^2\neq 1$. This concludes the proof.
\end{proof}

We are finally ready to prove Proposition \ref{prop: const-combinatorics}.

\begin{proof}[Proof of Proposition \ref{prop: const-combinatorics}]
By Lemma \ref{lem:subset A^4}, $\mathfrak r_{|t|}(\overline {\mathcal P}_t)\subset\Aff^4$. Recall that $[1:0:0:0:0]\in\mathfrak r_{|t|}({\mathcal P}_t)\cap\mathbb X_t^4$. By Lemma \ref{lem:subset X^4_t}, the intersection $\mathfrak r_{|t|}(\overline {\mathcal P}_t)\cap \partial{\mathbb X_t^4}$ consists solely of vertices of $\mathfrak r_{|t|}(\overline {\mathcal P}_t)$, hence we have also $\mathfrak r_{|t|}(\overline {\mathcal P}_t)\subset\overline{\mathbb X_t^4}$. Thus, $\mathfrak r_{|t|}(\overline {\mathcal P}_t)\subset\overline{\mathbb X_t^4}\cap\Aff^4$. 
The combinatorics of $\mathfrak r_{|t|}(\overline {\mathcal P}_t)$ is constant by Lemma \ref{lem: vertices}. The proof is complete.
\end{proof}

In contrast with $\overline {\mathcal P}_t$, the polytope ${\mathcal P}_t$ is simple \cite[Proposition 3.12]{MR}. Moreover, each ideal vertex of ${\mathcal P}_t$ belongs to exactly 6 facets of $\overline {\mathcal P}_t$ \cite[Proposition 3.16]{MR}. (We applied Proposition \ref{prop: const-combinatorics} to conclude when $t\in I^-$.) We adopt the following notation for the faces of ${\mathcal P}_t$ when $t\neq0$ (and similarly for the rescaled limit $\lim_{t\to0}\mathfrak r_{|t|}{\mathcal P}_t$):
\begin{itemize}
\item facets: ${\mathcal F}_{\l H}=\partial\l H\cap {\mathcal P}_t$, 
\item ridges: ${\mathcal R}_{\l H_1\l H_2}=\partial\l H_1\cap\partial\l H_2\cap {\mathcal P}_t$,
\item edges: ${\mathcal E}_{\l H_1\l H_2\l H_3}=\partial\l H_1\cap\partial\l H_2\cap\partial\l H_3\cap {\mathcal P}_t$,
\item finite vertices: ${\mathcal V}_{\l H_1\ldots\l H_4}=\partial\l H_1\cap\ldots\cap\partial\l H_4\cap {\mathcal P}_t$,
\item ideal vertices: ${\mathcal V}_{\l H_1\ldots\l H_6}=\partial\l H_1\cap\ldots\cap\partial\l H_6\cap\overline {\mathcal P}_t$,
\end{itemize}
where $\l H,\l H_i\subset\SP^4$ are half-spaces from the list in Table \ref{table:walls}.

We conclude the section with a combinatorial description of the facets of the polytope ${\mathcal P}_t$. This follows by applying Proposition \ref{prop: const-combinatorics} to \cite[Proposition 3.16]{MR}, where the combinatorics was studied for $t\in I^+$.

\begin{prop} \label{prop: facets}
For all $t\in I$, the combinatorics of each of the 22 facets ${\mathcal F}_{\l X}$, ${\mathcal F}_{\m i}$ and ${\mathcal F}_{\p i}$ of $\r_{|t|}{\mathcal P}_t$ where $\l i\in\{\l0,\ldots,\l7\}$ and $\l X\in\{\l A,\ldots,\l F\}$, is described in Figures \ref{fig:facet_L}, \ref{fig:facet_N} and \ref{fig:facet_P}, respectively.
\end{prop}

\begin{figure}
\includegraphics[scale=1]{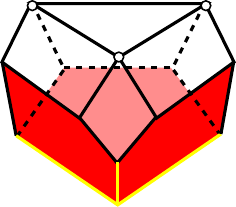}
\caption[A facet ${\mathcal F}_{\p i}$ of ${\mathcal P}_t$.]{\footnotesize A facet ${\mathcal F}_{\p i}$, $\p i\in\{\p0,\ldots,\p7\}$, of the 4-polytope ${\mathcal P}_t$. The white dots represent ideal vertices, the black edges are right-angled, and the yellow edges have some other varying dihedral angle. The three red pentagons are ridges of ${\mathcal P}_t$ of type ${\mathcal R}_{\p i\p j}$, and have varying dihedral angle $\theta_t$ or $\varphi_t$ (see Proposition \ref{prop: ridges}). Each of the three quadrilaterals with one ideal vertex (resp. two ideal vertices) is a ridge of ${\mathcal P}_t$ of type ${\mathcal R}_{\p i\m j}$ (resp. ${\mathcal R}_{\p i\l X}$). The horizontal ideal triangle is the ridge ${\mathcal R}_{\p i\m i}$.}
\label{fig:facet_P}
\end{figure}

\subsection{Geometry of the polytope} \label{sec: geometry polytope}
We continue to describe the polytope ${\mathcal P}_t$. Recall Sections \ref{sec geom halfspaces} and \ref{lemma hyperplane hp} about hyperplanes and angles in $\Hyp^n$, $\AdS^n$ and $\HP^n$. By applying Lemmas \ref{lemma hyperplane ads} and \ref{lemma hyperplane hp} to the list of vectors in Tables \ref{table:walls} and \ref{table:walls rescaled}, we get:

\begin{prop} \label{prop: hyperplanes polytope}
When $t\in I^-$, each hyperplane $\partial\p i\cap\AdS^4$ is spacelike, while each hyperplane $\partial\m i\cap\AdS^4$ and $\partial\l X\cap\AdS^4$ is timelike, for all $\l i\in\{\l0,\ldots,\l7\}$ and $\l X\in\{\l A,\ldots,\l F\}$.

Similarly when $t=0$, the rescaled limits of $\partial\p i$ are non-degenerate hyperplanes in $\HP^4$, while the rescaled limits of $\partial\m i$ and $\partial\l X$ are degenerate hyperplanes.
\end{prop}

We will also need the following:

\begin{prop} \label{prop: ridges}
The constant dihedral angles of ${\mathcal P}_t$ are right. The non-constant ones equal $\theta_t\in[\frac\pi2,\pi)$ when $t\in I^+$, and $\varphi_t\in(0,+\infty)$ when $t\in I^-$, where
$$\cos\theta_t=\frac{3t^2-1}{1+t^2}\quad\mbox{and}\quad\cosh\varphi_t=\frac{3t^2+1}{1-t^2}.$$
A ridge is compact if and only if its dihedral angle is non-constant, and such ridges consist precisely of ${\mathcal R}_{\p i\p j}$ for all distinct $\p i,\p j\in\lbrace\p0,\ldots,\p7\rbrace$ such that $i\equiv j\ (\mathrm{mod}\ 2)$.
\end{prop}

\begin{proof}
This is proven for $t\in I^+$ in \cite[Proposition 3.10]{MR}. Such ridges are compact also when $t\in I^-$ by Proposition \ref{prop: const-combinatorics}. By applying Lemma \ref{lem: AdS angle} to the list of vectors in Table \ref{table:walls}, we conclude also for $t\in I^-$. (In fact, we have already observed in Remark \ref{rmk:orthogonality preserved} that the orthogonality between the vectors of Table \ref{table:walls} is maintained when $t\in(-1,0)$ for the bilinear form of signature $-+++-$.)
\end{proof}

We now describe the links of the vertices of ${\mathcal P}_t$, whose geometric structures have been described in Section \ref{sec:cusps} for ideal vertices and Section \ref{sec:poly-conemfds} for finite vertices. These are depicted in Figures \ref{fig:links}, \ref{fig: cusp transition} and \ref{fig: cusp transition2} and described by the following proposition:

\begin{figure}
\includegraphics[scale=.14]{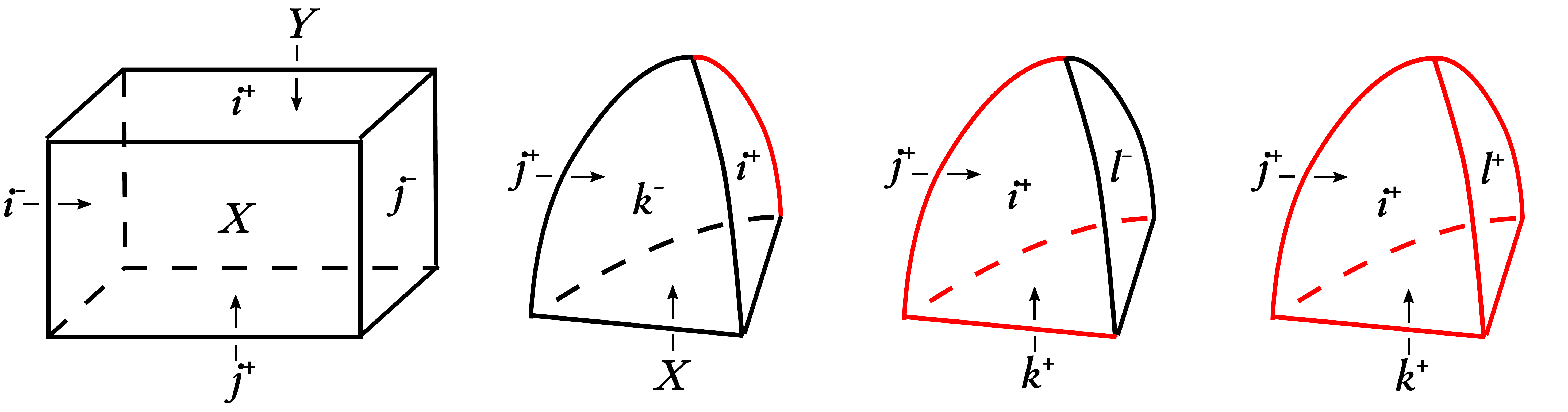}
\caption[The links of the vertices of ${\mathcal P}_t$.]{\footnotesize The links of the vertices of ${\mathcal P}_t$ (see Proposition \ref{prop: vertices}). When $t\in I^+$ (resp. $t\in I^-$) the link of an ideal vertex is a Euclidean (resp. Minkowski) right paralleleped, and the link of a finite vertex is a spherical tetrahedron (resp. de Sitter tetrahedron with spacelike facets). The black edges are right-angled, and the red edges have varying dihedral angle.}
\label{fig:links}
\end{figure}

\begin{prop} \label{prop: vertices}
The 46 vertices of ${\mathcal P}_t$ are divided by the similarity class of their link in:
\begin{itemize}
\item 12 ideal vertices of type ${\mathcal V}_{\p i\m i\p j\m j\l X\l Y}$, 
\item 24 finite vertices of type ${\mathcal V}_{\p i\p j\m k\l X}$, 
\item 8 finite vertices of type ${\mathcal V}_{\p i\p j\p k\m \ell}$, 
\item 2 finite vertices of type ${\mathcal V}_{\p i\p j\p k\p\ell}$, 
\end{itemize}
The link of each ideal vertex is a rectangular parallelepiped (which in a horospherical section is Euclidean when $t\in I^+$, and Minkowski when $t\in I^-$), while the link each finite vertex is a tetrahedron (which is spherical when $t\in I^+$, and HS when $t\in I^-$). A similar statement holds for the rescaled limit $\lim_{t\to0}\mathfrak r_{|t|}{\mathcal P}_t$ in the half-pipe setting.
\end{prop}

\begin{proof}
These facts are proven in \cite[Proposition 3.16]{MR} for $t\in I^+$. By applying Propositions \ref{prop: const-combinatorics}, \ref{prop: hyperplanes polytope} and \ref{prop: ridges}, we conclude also for $t\in I^-$ and for the rescaled limit.
\end{proof}


\begin{figure}
\centering
\begin{minipage}[c]{.33\textwidth}
\centering
\includegraphics[scale=0.5]{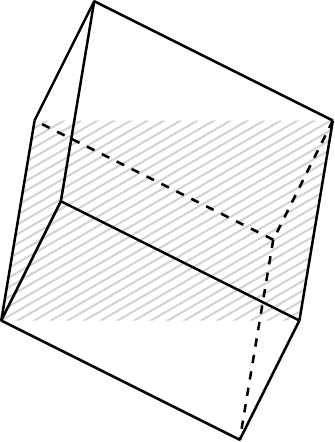}
\end{minipage}%
\begin{minipage}[c]{.33\textwidth}
\centering
\includegraphics[scale=0.5]{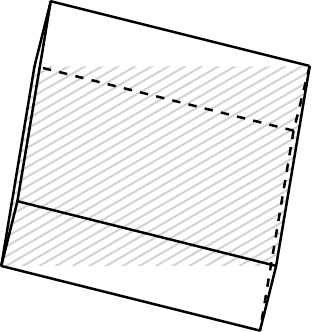}
\end{minipage}%
\begin{minipage}[c]{.33\textwidth}
\centering
\includegraphics[scale=0.5]{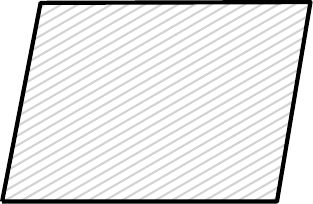}
\end{minipage}
\caption[The link of an ideal vertex of ${\mathcal P}_t$.]{\footnotesize The link of an ideal vertex of ${\mathcal P}_t$, obtained by intersecting ${\mathcal P}_t$ with a horosphere. The further intersection with $\Hyp^3$, which is constant in $t$, {is a rectangle (shaded in the picture)}. When $t\to 0$, the rectangular parallepiped collapses to this rectangle.} 
\label{fig: cusp transition}
\end{figure}

\begin{figure}
\centering
\begin{minipage}[c]{.33\textwidth}
\centering
\includegraphics[scale=0.5]{euclideancusp1.pdf}
\end{minipage}%
\begin{minipage}[c]{.33\textwidth}
\centering
\includegraphics[scale=0.47]{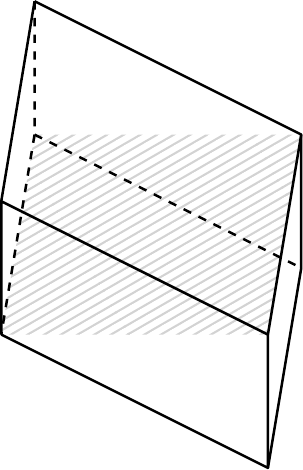}
\end{minipage}%
\begin{minipage}[c]{.33\textwidth}
\centering
\includegraphics[scale=0.45]{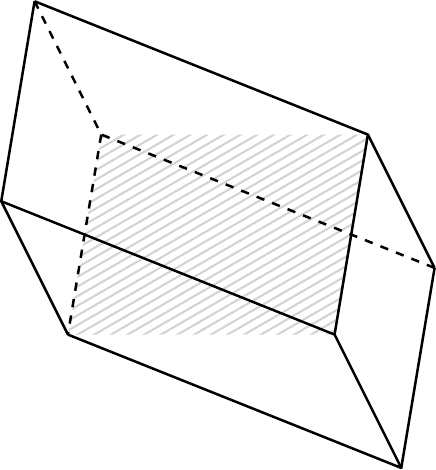}
\end{minipage}
\caption[Geometric transition on the link of an ideal vertex.]{\footnotesize After rescaling, the geometry of the link of an ideal vertex transitions from Euclidean (left) to Minkoskian (right), via Galilean geometry (centre). The intersection with the fixed copy of $\Hyp^3$ is shaded. This is an example of the transition explained in Section \ref{sec transition horospheres}.} 
\label{fig: cusp transition2}
\end{figure}

As a consequence of Proposition \ref{prop: vertices}, we get:

\begin{cor} \label{cor: finite volume polytope}
When $t\in I^-$, the anti-de Sitter polytope ${\mathcal P}_t\subset\AdS^4$ has finite volume.
The same holds for the half-pipe polytope $\lim_{t\to0}\mathfrak r_{|t|}({\mathcal P}_t)\subset\HP^4$.
\end{cor}

\begin{proof}
By Proposition \ref{prop:polytope_hyp}, when $t\in I^+$ the hyperbolic polytope ${\mathcal P}_t$ has finite volume. Equivalently, each edge of ${\mathcal P}_t$ joins two (finite or ideal) vertices of ${\mathcal P}_t$. 
By Proposition \ref{prop: const-combinatorics}, this last fact also holds in $\AdS^4$ when $t\in I^-$ and in $\HP^4$ for the rescaled limit. By truncating the ends of ${\mathcal P}_t$ (resp. $\lim_{t\to0}\mathfrak r_{|t|}({\mathcal P}_t)$) with horospheres (see Section \ref{sec: horospheres}), we have decomposed the polytope in a compact part plus 12 regions, each isometric to the fundamental domain of a cusp in an anti-de Sitter (resp. half-pipe) 4-manifold (see Definition \ref{defi cusp}). Therefore (see Remark \ref{rem: finite volume}), the polytope has finite volume.
\end{proof}

\subsection{Orbifold transition} \label{subsec orbifold transition}

Roughly speaking, a $(G,X)$-\emph{orbifold} is a space locally modelled on quotients of $X$ by finite subgroups of $G$. 
We refer to \cite{thurstonnotes,choi} for the details.

The geometry of ${\mathcal P}_t$ gives the complement of the ridges with non-constant dihedral angle (see Proposition \ref{prop: ridges}) a natural structure of orbifold (this is different, but somehow related, to the concept of ``mirror polytope'' \cite{marquis,CLM}), and this fact will be convenient in the sequel. 

Let us first prove a preliminary lemma. Recall Section \ref{sec: reflections} about reflections and their limits. For the sake of clarity, we will make explicit the dependence of the half-spaces of Table \ref{table:walls} in $t$, by a subscript $\l H_t$.

\begin{lemma}\label{lemma reflections C1}
For every half-space $\l H$ of Table \ref{table:walls}, let $r_{\l H}=r_{\l H}(t)$ be the reflection in $\Isom(\Hyp^4)$ for $t\in I^+$ and in $\Isom(\AdS^4)$ for $t\in I^-$ which fixes $\partial \l H$. Then $\mathfrak r_{|t|}r_{\l H}(t)\mathfrak r_{|t|}^{-1}$ extends to a $C^1$ path in $\Aut(\SP^4)$ for $t\in I^-\cup\{0\}\cup I^+$.
\end{lemma}

Observe that Lemma \ref{lemma reflections C1} does not follow immediately from the convergence of the hyperplanes $\mathfrak r_{|t|}(\partial\l H_t)$ to a half-pipe hyperplane because, as remarked in Section \ref{remark:reflections}, in half-pipe geometry there is \emph{not} a uniquely determined reflection in a hyperplane.

\begin{proof}
Let us start by the case $\l H=\m i\in\l i\in\{\m 0,\ldots,\m 7\}$. For $t\in I^+$, the hyperbolic reflection $r_{\l i^-}$ (for which we will make explicit the dependence on $t$) can be written as the matrix
\begin{equation} \label{eq: reflection hyp}
r_{\l i^-}(t)=\mathrm{id}-2J_1\alpha_i(t)\alpha_i(t)^T~,
\end{equation}
where 
$J_1=\mathrm{diag}(-1,1,1,1,1)$ and
$$\alpha_i(t)^T=\frac{1}{\sqrt{1+t^2}}\left( -\sqrt{2},\pm1,\pm1,\pm1,\pm t \right)$$
is obtained from the vector defining $\m i$ (see Table \ref{table:walls}) by normalising with respect to the Minkowski product of $\R^{1,4}$.
(The signs in the $\pm$ symbols are fixed once and forever according to the choice of $\m i$.)
Indeed, one can check (using that $\alpha_i(t)^TJ_1\alpha_i(t)=1$) that the expression of Equation \eqref{eq: reflection hyp} maps $J_1\alpha_i(t)$ to its opposite, whereas it fixes every $v\in\partial \l i^-$, since $v$ satisfies $\alpha_i(t)^T v=0$.

Similarly, when $t\in I^-$ the AdS reflection can be expressed as
\begin{equation} \label{eq: reflection ads}
r_{\l i^-}(t)=\mathrm{id}-2J_{-1}\alpha_i(t)\alpha_i(t)^T~,
\end{equation}
where $J_{-1}=\mathrm{diag}(-1,1,1,1,-1)$ and 
$$\alpha_i(t)=\frac{1}{\sqrt{1-t^2}}\left( -\sqrt{2}:\pm1:\pm1:\pm1:\pm t \right)~.$$
Hence we get the expression (for $t\neq 0)$:
$$\r_{|t|}r_{\l i^-}(t)\r_{|t|}^{-1}=\mathrm{id}-2\left[J_{\mathrm{sign}(t)}\r_{|t|}\alpha_i(t)\right]\left[\r_{|t|}^{-1}\alpha_i(t)^T\right]~.$$
The term in the first square bracket thus reads for both $t>0$ and $t<0$ as the column vector:
$$\frac{1}{\sqrt{1+t|t|}}\left(\sqrt{2}:\pm1:\pm1:\pm1:\pm 1 \right)$$
while the second square bracket has the form (horizontal vector):
$$\frac{1}{\sqrt{1+t|t|}}\left( -\sqrt{2}:\pm1:\pm1:\pm1:\pm t|t| \right)~.$$
Since both extend $C^1$ to $t=0$, so does $\r_{|t|}r_{\l i^-}(t)\r_{|t|}^{-1}$.

For $\l H=\l X\in \l X\in\{\l A,\ldots,\l F\}$, the path $\r_{|t|}r_{\l X}\r_{|t|}^{-1}$ is actually constant, since $\l X$ does not depend on $t$ and it can be easily checked that the hyperbolic and AdS reflections, expressed as in Equations \eqref{eq: reflection hyp} and \eqref{eq: reflection ads},  coincide and commute with $\r_{|t|}$.
Finally, for the case $\l H=\p i\in\{\p 0,\ldots,\p 7\}$, when $t\in I^-$ there is a small difference in the formula of Equation \eqref{eq: reflection hyp}, which now becomes
\begin{equation} \label{eq: reflection ads 2}
r_{\l i^+}(t)=\mathrm{id}+2J_{-1}\alpha_i(t)\alpha_i(t)^T~,
\end{equation}
due to the fact that the $\m i$ are timelike while the $\p i$ are spacelike, and again the rescaled limit is the same as for $t\in I^+$. With this caveat, it can be checked directly that all the entries in $\r_{|t|}r_{\l i^+}(t)\r_{|t|}^{-1}$ are, up to constants, either of the form $1/\sqrt{1+t|t|}$ or of the form $t|t|/\sqrt{1+t|t|}$, and thus the path is $C^1$ in $\Aut(\SP^4)$.
\end{proof}

From the proof of Lemma \ref{lemma reflections C1}, we see also that the convergence is not $C^2$.

\begin{remark} \label{remark rescaled holonomy}
The proof of Lemma \ref{lemma reflections C1} enables us to compute also the limits of $\r_{|t|}r_{\l H}(t)\r_{|t|}^{-1}$, as {$t\to0$} in the half-pipe group $G_{\HP^4}$. In fact, for the reflections along the hyperplanes $\m i$, we obtain immediately
\begin{equation}\label{eq:colomba}
\lim_{t\to 0}\r_{|t|}r_{\l i^-}(t)\r_{|t|}^{-1}=\left(
\begin{array}{ccc|c}
  &&& 0 \\
  
   & \mathrm{id}-2Jv_iv_i^T & & \vdots \\
  &&& 0 \\
    \hline  
   \ldots &\pm 2v_i^T & \ldots  & 1
\end{array}
\right)=\phi(r_i,\mp 2Jv_i)~,
\end{equation}
where we put
{$v_i^T=( -\sqrt{2},\pm 1,\pm 1,\pm 1)$} (namely, the first four terms of the corresponding vector in Table \ref{table:walls}, and $r_i=\mathrm{id}-2Jv_iv_i^T$ is the reflection in $\Hyp^3$ in the plane determined by $v_i$, for $J=\mathrm{diag}(-1,1,1,1)$. The sign in Equation \eqref{eq:colomba} depends on the oddity of $i$, since it follows from the sign in the last entry of $\alpha_i(t)$, and in fact the correct sign is $(-1)^{i+1}$. In the last equality, we applied the isomorphism $\phi$ of Lemma \ref{lemma isomorphism halfpipe and minkowski}. 

For any half-space $\l X\in\{\l A,\ldots,\l F\}$, the same computation shows easily that 
$$\lim_{t\to 0}\r_{|t|}r_{\l X}\r_{|t|}^{-1}=\phi(r_{\l X},0)~,$$
where $r_{\l X}$ is now interpreted as the reflection in $\Hyp^3$ associated to the plane $\Hyp^3\cap{\partial}\l X$.

Finally, for the half-spaces of the form $\p i$ the computation is again similar following Lemma \ref{lemma reflections C1}. One obtains
$$\lim_{t\to 0}\r_{|t|}r_{\p i}(t)\r_{|t|}^{-1}=\left(\begin{array}{ccc|c}
  &&& 0 \\
  
   & \mathrm{id} & & \vdots \\
  &&& 0 \\
    \hline  
   \ldots &\mp 2v_i^T & \ldots  & -1
\end{array}
\right)=\phi(-\mathrm{id},\mp 2Jv_i)~,$$
where $v_i$ is defined as above.
\end{remark}

We are ready to describe the natural orbifold structure on a subset of $\mathcal P_t$.

\begin{prop} \label{prop: transiz orbifold}
The set
$${\mathcal P}^\times_t= {\mathcal P}_t\smallsetminus\bigcup_{i\neq j}{\mathcal R}_{\p i\p j}$$
is isometric to a hyperbolic orbifold when $t\in I^+$, and to an anti-de Sitter orbifold when $t\in I^-$. Similarly, the rescaled limit $\lim_{t\to0}\mathfrak r_{|t|}{\mathcal P}^\times_t$ has a natural structure of half-pipe orbifold.
\end{prop}

\begin{proof}
When $t\in I^+$ (resp. $t\in I^-$), we associate to each facet ${\mathcal F}_{\l H}$ of ${\mathcal P}_t$ the unique hyperbolic (resp. AdS) reflection $r_{\l H}$ that fixes the bounding hyperplane $\partial\l H$. By Lemma \ref{lemma reflections C1}, when $t\to 0^\pm$, the rescaled reflections $\lim_{t\to0}\mathfrak r_{|t|}r_{\l H}(t)\mathfrak r_{|t|}^{-1}$ converge to a half-pipe reflection.

Note that by Proposition \ref{prop: vertices} ${\mathcal P}^\times_t$ does not contain any finite vertex of ${\mathcal P}_t$, hence it only remains to check the orbifold structure at the edges. Now, note that by Proposition \ref{prop: facets} each edge of ${\mathcal P}_t$ disjoint from each of the ridges ${\mathcal R}_{\p i\p j}$ is of type ${\mathcal E}_{\p i\m j\l X}$. Moreover, since the hyperbolic (resp. AdS) hyperplanes ${\partial}\p0$, ${\partial}\m1$ and ${\partial}\l A$ are pairwise orthogonal for all $t$, the corresponding hyperbolic (resp. AdS) reflections commute. So we have
$$\Delta_t=\langle r_{\p0},r_{\m1},r_{\l A}\rangle\cong(\Z/2\Z)^3,$$
and similarly for the rescaled limit $\lim_{t\to0}\mathfrak r_{|t|}\Delta_t\mathfrak r_{|t|}^{-1}$.

By symmetry and Proposition \ref{prop: ridges} (and also Proposition \ref{prop: hyperplanes polytope} in the AdS case) the conjugacy class of $\Delta_t$ in $\mathrm{Isom}(\Hyp^4)$ or $\mathrm{Isom}(\AdS^4)$ does not depend on the chosen triple of reflections $r_{\p i},r_{\m j},r_{\l X}$ such that there is an edge ${\mathcal E}_{\p i\m j\l X}$ of the polytope.

In this way, ${\mathcal P}^\times_t$ (together with the associated reflections) is locally modelled on
$$\Hyp^4/_{\Delta_t}\quad\mbox{and}\quad\AdS^4/_{\Delta_t}$$  
when $t\in I^+$ and $t\in I^-$, respectively. Similarly, $\lim_{t\to0}\mathfrak r_{|t|}{\mathcal P}^\times_t$  is locally modelled on
$$\HP^4/_{\lim_{t\to0}\mathfrak r_{|t|}\Delta_t\mathfrak r_{|t|}^{-1}}.$$
The proof is complete.
\end{proof}

\begin{remark} \label{rem: transiz orbifold}
By an opportune orbifold version of Definition \ref{defi transition}, the path $t\mapsto {\mathcal P}^\times_t$ defines a geometric transition on an orbifold. Moreover, the transition is $C^1${. Indeed,} the holonomy representation depends $C^1$ on the parameter $t$ as a consequence of Lemma \ref{lemma reflections C1}. The developing map also depends $C^1$ essentially because the vectors defining the rescaled polytope $\mathfrak r_{|t|}({\mathcal P}_t)$ depend $C^1$ on $t$.
\end{remark}

\begin{remark} \label{rem: right-angled}
By Proposition \ref{prop: ridges}, the polytope ${\mathcal P}_{\nicefrac1{\sqrt3}}$ is right-angled. In particular, it can be thought as a hyperbolic 4-orbifold. In contrast with ${\mathcal P}_t^\times$, the orbifold ${\mathcal P}_{\nicefrac1{\sqrt3}}$ is complete (and clearly ${\mathcal P}_{\nicefrac1{\sqrt3}}$ is the metric completion of ${\mathcal P}_{\nicefrac1{\sqrt3}}^\times$).
\end{remark}

\begin{remark} \label{rmk: HP holonomy cusps}
Let us briefly elucidate the geometric structure of the cusp sections of the orbifold $\mathcal P^\times_t$ and of its recaled limit. In Figure \ref{fig: cusp transition2} we showed a horospherical section of {an ideal vertex of} the polytope $\mathfrak r_{|t|}({\mathcal P}_t)$, for $t<0$, $t=0$, $t>0$. The subgroup of the orbifold fundamental group of $\mathcal P_t^\times$ 
preserving a cusp is isomorphic to the Coxeter group $\Gamma_{\mathrm{cube}}$ generated by reflections in the sides of a Euclidean cube --- see 
\cite{transition_char_var} for more details. In the hyperbolic and AdS case, the restriction of the holonomy representation of the orbifold $\mathcal P_t^\times$ to this peripheral subgroup $\Gamma_{\mathrm{cube}}$ maps each generator to a Euclidean or Minkowski reflection in a face of the rectangular parallelepiped (as in Figures \ref{fig: cusp transition} and \ref{fig: cusp transition2}). 
\end{remark}

\subsection{The cuboctahedron} \label{sec: cuboctahedron}

If a bounded Euclidean polytope $\overline {\mathcal P}\subset\R^n$ is \emph{vertex-transitive}, i.e. its symmetry group acts transitively on the set of the vertices, then $\overline {\mathcal P}$ is inscribed in a closed ball $\overline  {\mathcal B}$. Let us identify $\R^n$ with our favourite affine chart $\Aff^n\subset\SP^n$ of the projective sphere. Up to similarity, we can put $\overline  {\mathcal B}=\overline{\Hyp^n}\subset\SP^n$, so that ${\mathcal P}=\overline {\mathcal P}\cap\Hyp^n$ is an \emph{ideal} hyperbolic polytope, i.e. all the vertices of ${\mathcal P}$ are ideal. The polytope ${\mathcal P}$ is unique up to isometry of $\Hyp^n$.

A \emph{Euclidean cuboctahedron} $\overline {\mathcal C}\subset\R^3$ (see Figure \ref{fig:cuboct}) is the convex envelop of the midpoints of the edges of a regular cube (or, equivalently, of a regular octahedron). The polyhedron $\overline {\mathcal C}$ is vertex-transitive, and has 14 facets, consisting of 6 squares and 8 triangles.

Let now ${\mathcal C}\subset\Hyp^3$ be the \emph{ideal hyperbolic cuboctahedron}.
We shall identify the cuboctahedron ${\mathcal C}$ with ${\mathcal P}_0\subset\Hyp^3\subset\Hyp^4,\AdS^4$ thanks to the following (see Figure \ref{fig:cuboctahedron}):

\begin{figure}
\includegraphics[scale=.35]{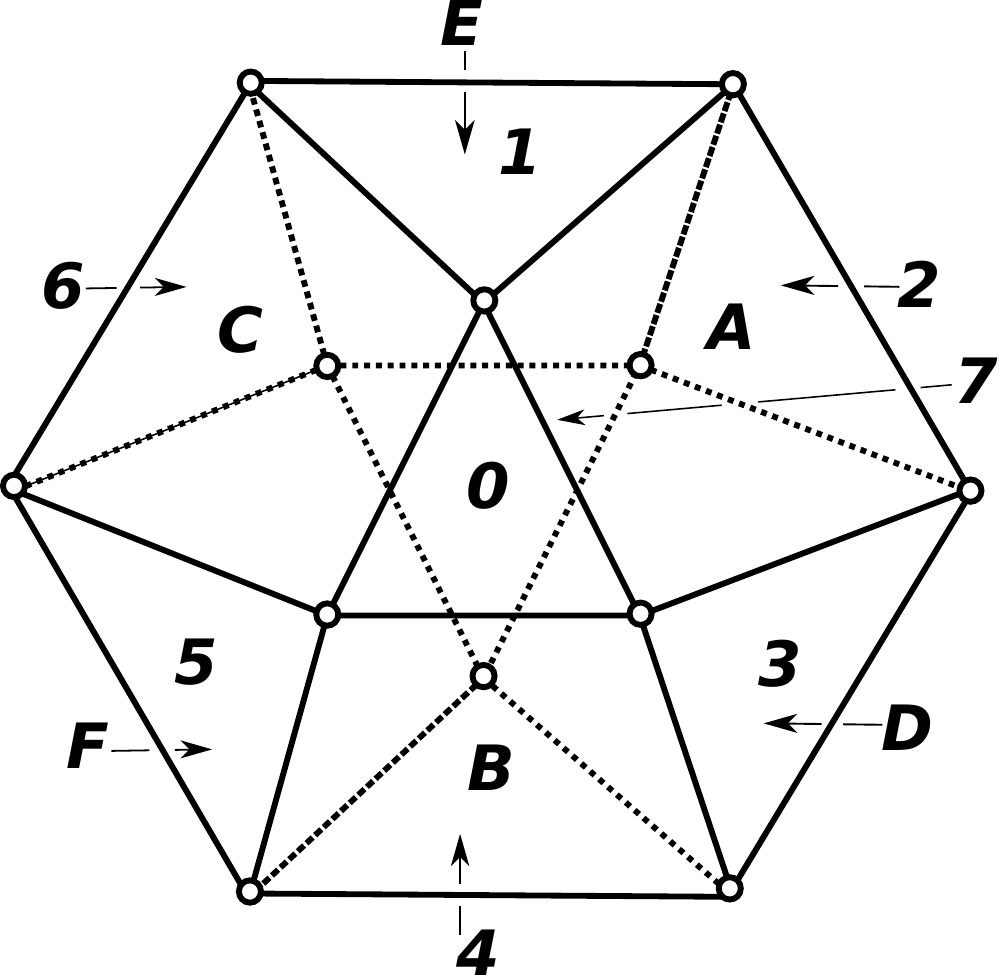}
\caption[The cuboctahedral slice $\mathcal P_0$ of $\mathcal P_t$.]{\footnotesize The ideal right-angled cuboctahedron ${\mathcal C}={\mathcal P}_t\cap\Hyp^3={\mathcal P}_0$. A quadrilateral face with label $\l X\in\{\l A,\ldots,\l F\}$ coincides with ${\mathcal F}_{\l X}\cap\Hyp^3$, while a triangular face with label $\l i\in\{\l0,\ldots,\l7\}$ coincides with the ridge ${\mathcal R}_{\p i\m i}$ of ${\mathcal P}_t$.}
\label{fig:cuboctahedron}
\end{figure}

\begin{prop} \label{prop: cuboctahedron}
The set ${\mathcal P}_t\cap \Hyp^3$ does not depend on $t\in I$ and is isometric to ${\mathcal C}$.
Its 6 quadrilateral faces are given by ${\mathcal F}_{\l X}\cap\Hyp^3$ for all $\l X\in\lbrace\l A,\ldots,\l F\rbrace$, while the 8 triangular faces are the ridges of ${\mathcal P}_t$ of type ${\mathcal R}_{\p i\m i}$ for all $\l i\in\lbrace\l0,\ldots,\l7\rbrace$.
Moreover, we have $${\mathcal P}_t\cap \Hyp^3=\bigcap_{s\in I}{\mathcal P}_s={\mathcal P}_0.$$
\end{prop}

\begin{proof}
It is straightforward to check that the same proof of \cite[Proposition 3.19]{MR} for $t\in I^+\cup\{0\}$ applies also when $t\in I^-$ (recall the isometric embedding $\iota\colon\Hyp^3\hookrightarrow\X_t^4$ defined by \eqref{eq: iota} in Section \ref{sec recipe}).
\end{proof}

A nice feature of the hyperbolic polyhedron ${\mathcal C}$ is that it is right-angled. In particular, there is a unique hyperbolic  orbifold $\Hyp^3/_{\Gamma_{\mathrm{co}}}$ isometric to ${\mathcal C}$, where the discrete group $\Gamma_{\mathrm{co}}<\Isom(\Hyp^3)$ is generated by reflections through the bounding hyperplanes of ${\mathcal C}$.
We shall thus interpret ${\mathcal C}$ as an orbifold.

\subsection{From polytopes to manifolds} \label{sec: construction}
We now build the cone-manifolds of Theorem \ref{teo: main}.

Let $\mathcal{N}\to {\mathcal C}$ be an \emph{orbifold covering} for some 3-manifold $\mathcal{N}$; in other words, we can assume to have a torsion-free subgroup $\Gamma<\Gamma_{\mathrm{co}}$ and
$$\mathcal{N}=\Hyp^3/_\Gamma\to\Hyp^3/_{\Gamma_{\mathrm{co}}}={\mathcal C}.$$

\begin{remark} \label{rem: same proof}
The first two points of Theorem \ref{teo: main} (that is, when $t\in I^+$) were proven in \cite[Theorem 1.2]{MR} for a particular manifold $\mathcal{N}$ such that $\Gamma<\Gamma_{\mathrm{co}}$ is normal and $\Gamma/\Gamma_{\mathrm{co}}\cong\Z/{2\Z}\times\Z/{2\Z}$.
Following our arguments, the proof given there can be indeed extended to every $\mathcal{N}$ that orbifold-covers ${\mathcal C}$, as in our hypothesis. Then the main content of our Theorem \ref{teo: main} is extending the deformation to half-pipe and anti-de Sitter geometry, for a rather general choice of $\mathcal{N}$ (see also the discussion of Remark \ref{rem: cuboctahedral_manifolds} for this point).
\end{remark}

The covering $\mathcal{N}\to {\mathcal C}$ induces a tessellation of the hyperbolic 3-manifold $\mathcal{N}$ into copies of ${\mathcal C}$. One can think of $\mathcal{N}$ as obtained by pairing the facets of such copies of ${\mathcal C}$ through the maps induced by the identity. The existence (and abundance) of such orbifold-covers from a manifold to ${\mathcal C}$ is a consequence of Selberg's Lemma (and Malcev's Theorem).

Now, we pick a copy of ${\mathcal P}_t$ for each copy of ${\mathcal C}$ in $\mathcal{N}$. Recall that by Propsition \ref{prop: cuboctahedron} we put ${\mathcal C}={\mathcal P}_0\subset {\mathcal P}_t$. If two copies of ${\mathcal C}$ in $\mathcal{N}$ are adjacent along a quadrilateral face ${\mathcal F}_{\l X}\cap\Hyp^3$, we glue the corresponding two copies of ${\mathcal P}_t$ along the facet ${\mathcal F}_{\l X}$ through the map induced by the identity. If two copies of ${\mathcal C}$ in $\mathcal{N}$ are adjacent along a triangular face ${\mathcal R}_{\p i\m i}$, we glue the corresponding two copies of ${\mathcal P}_t$ along the facet ${\mathcal F}_{\m i}$ through the map induced by the identity. We call $\mathcal{X}'_t$ the resulting space. Note that we have paired all the facets of the copies of ${\mathcal P}_t$, with the exception of those of type ${\mathcal F}_{\p0},\ldots,{\mathcal F}_{\p7}$.

Finally, let $\mathcal{X}_t$ be the space obtained by doubling $\mathcal{X}'_t$ along the unpaired facets. We call also $\mathfrak r_{|t|}(\mathcal{X}_t)$ the space obtained similarly to $\mathcal{X}_t$, by taking copies of the rescaled polytope $\mathfrak r_{|t|}({\mathcal P}_t)$ in place of copies of ${\mathcal P}_t$. We have:

\begin{prop} \label{prop: product}
For all $t\in I\smallsetminus\{0\}$, the space $\mathcal{X}_t$ is homeomorphic to $\mathcal{N}\times S^1$. The same holds for the rescaled $\mathfrak r_{|t|}(\mathcal{X}_t)$ for all $t\in I$.
\end{prop}

\begin{proof}
By the proof of \cite[Proposition 4.13]{MR} and by Proposition \ref{prop: const-combinatorics}, as $t\neq0$ there is a homeomorphism ${\mathcal P}_t\to {\mathcal C}\times[-1,1]$ which restricts to
\begin{align*}
{\mathcal P}_0\to\ & {\mathcal C}\times\{0\} & \mbox{(see\ Figure\ \ref{fig:cuboctahedron}),}\\
{\mathcal F}_{\p0}\cup {\mathcal F}_{\p2}\cup {\mathcal F}_{\p4}\cup {\mathcal F}_{\p6}\to\ & {\mathcal C}\times\{-1\} & \mbox{(see Figure \ref{fig:altogether}),}\\
{\mathcal F}_{\p1}\cup {\mathcal F}_{\p3}\cup {\mathcal F}_{\p5}\cup {\mathcal F}_{\p7}\to\ & {\mathcal C}\times\{1\} & \mbox{(see Figure \ref{fig:altogether}),}\\
{\mathcal F}_{\l X}\to\ & {\mathcal Q}\times[-1,1] & \mbox{(see Figure \ref{fig:facet_L}),}\\
\mbox{and}\quad {\mathcal F}_{\m i}\to\ & \mathcal T\times[-1,1] & \mbox{(see Figure \ref{fig:facet_N}),}
\end{align*}
where for each $\l X\in\{\l A,\ldots,\l F\}$ (resp. $\m i\in\{\m0,\ldots,\m7\}$) there is a quadrilateral (resp. triangular) face ${\mathcal Q}\subset\partial {\mathcal C}$ (resp. $\mathcal T\subset\partial {\mathcal C}$) of ${\mathcal C}$ (recall Proposition \ref{prop: cuboctahedron}). This extends to a homeomorphism $\mathcal{X}'_t\to \mathcal{N}\times[-1,1]$. By doubling, the proof is complete.
\end{proof}

\begin{figure}
\includegraphics[scale=1.2, angle=270]{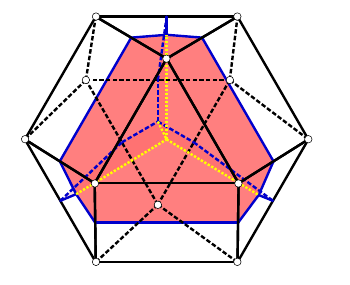}
\caption[The roof of $\mathcal P_t$ is a bent cuboctahedron.]{\footnotesize 
The union ${\mathcal C}_0={\mathcal F}_{\p0}\cup {\mathcal F}_{\p2}\cup {\mathcal F}_{\p4}\cup {\mathcal F}_{\p6}$ (resp. $\mathcal C_1={\mathcal F}_{\p1}\cup {\mathcal F}_{\p3}\cup {\mathcal F}_{\p5}\cup {\mathcal F}_{\p7}$) is an ideal right-angled cuboctahedron, pleated along the 6 red pentagons ${\mathcal R}_{\p i\p j}$ (each with 3 blue edges and 2 yellow edges in the picture). The facets of ${\mathcal C}_0$ are divided as follows: 4 ideal triangles ${\mathcal R}_{\p i\m i}$; 4 ideal triangles, each subdivided by the pleats as ${\mathcal R}_{\p i\m\ell}\cup {\mathcal R}_{\p j\m\ell}\cup {\mathcal R}_{\p k\m\ell}$; 6 ideal quadrilaterals, each subdivided by a pleat as ${\mathcal R}_{\p i \l X}\cup {\mathcal R}_{\p j\l X}$. The black edges are edges of 
the cuboctahedron, while
the blue edges are contained in its facets
, and the yellow edges intersect in the barycentre ${\mathcal V}_{\p0\p2\p4\p6}$ (resp.${\mathcal V}_{\p1\p3\p5\p7}$)
.}
\label{fig:altogether}
\end{figure}

\subsection{Transition and cone structures}

In this section, we give the promised cone-manifold structure to the space $\mathcal{X}_t$ constructed above and conclude the proof of Theorem \ref{teo: main}.

For $t\neq0$, we put
$$\Sigma_t\ =\bigcup_{{\mathcal P}_t\mbox{\tiny\ in\ }\mathcal{X}_t}\ \bigcup_{i\neq j}{\mathcal R}_{\p i\p j}\ \subset\ \mathcal{X}_t,$$
where the union runs over all the copies of ${\mathcal P}_t$ in $\mathcal{X}_t$. In other words, $\Sigma_t\subset \mathcal{X}_t$ is the union of the ridges with non-constant dihedral angle of the copies of ${\mathcal P}_t$ in $\mathcal{X}_t$ (see Proposition \ref{prop: ridges}), and we have (see Proposition \ref{prop: transiz orbifold})
$$\mathcal{X}_t\smallsetminus\Sigma_t\ =\bigcup_{{\mathcal P}_t\mbox{\tiny\ in\ }\mathcal{X}_t}{\mathcal P}_t^\times.$$
The couple $(\mathcal{X}_t,\Sigma_t)$ is homeomorphic to $(\mathcal{N}\times S^1,\Sigma)$ by Proposition \ref{prop: product}, where $\Sigma\subset \mathcal{N}\times S^1$ is a foam by Proposition \ref{prop: vertices}. If the covering $\mathcal{N}\to {\mathcal C}$ is finite, the foam $\Sigma$ is compact by Proposition \ref{prop: ridges}.

Recall from Proposition \ref{prop: transiz orbifold} that ${\mathcal P}^\times_t$ has a natural structure of orbifold. We have:

\begin{prop} \label{prop: orbifold covering}
The natural map $\mathcal{X}_t\smallsetminus\Sigma_t\to {\mathcal P}_t^\times$ is an orbifold covering, and similarly for the rescaled limits.
\end{prop}

\begin{proof}
We continue to refer to \cite{thurstonnotes,choi} for details about orbifolds and their coverings.

By the proof of Proposition \ref{prop: transiz orbifold}, it suffices to check that locally, near a $k$-stratum of the orbifold ${\mathcal P}^\times_t$, the map $\mathcal{X}_t\smallsetminus\Sigma_t\to {\mathcal P}_t^\times$ is modelled on the quotient map $\R^4\to\R^4/_{(\Z/2\Z)^{4-k}}$. Here, the $i$-th factor of $(\Z/2\Z)^{4-k}<(\Z/2\Z)^{4}$ is generated by the reflection $r_i\in\O(4)$ along the hyperplane $\{x_i=0\}\subset\R^4$. Note that by Proposition \ref{prop: vertices}, each stratum of the orbifold ${\mathcal P}^\times_t$ is non-compact and has an ideal vertex (in particular, ${\mathcal P}^\times_t$ has no 0-strata). This implies that it suffices to consider the case $k=1$ only.

By symmetry (see Lemma \ref{lem:symmetry}), we can fix a horosection $H$ of the ideal vertex ${\mathcal V}={\mathcal V}_{\p0\m0\p3\m3\l A\l B}$ of ${\mathcal P}_t$ and look at the effect of the gluing on the copies of the link $\mathcal L_{\mathcal V}=H\cap {\mathcal P}_t$. We also know by Proposition \ref{prop: vertices} that the orbifold structure on ${\mathcal L}_{\mathcal V}$ is that of a right parallelepiped. Again by symmetry, we can fix the vertex $v=v_{\m0\p3\l A}$ of ${\mathcal L}_{\mathcal V}$ and look at the effect of the gluing of the copies of its link $\ell_v$. We refer to Figure \ref{fig: gluing_link_of_link}. Note that the orbifold structure on $\ell_v$ is that of a mirror triangle $\Delta(2,2,2)=S^2/_{(\Z/2\Z)^3}$.

\begin{figure}
\includegraphics[scale=.4]{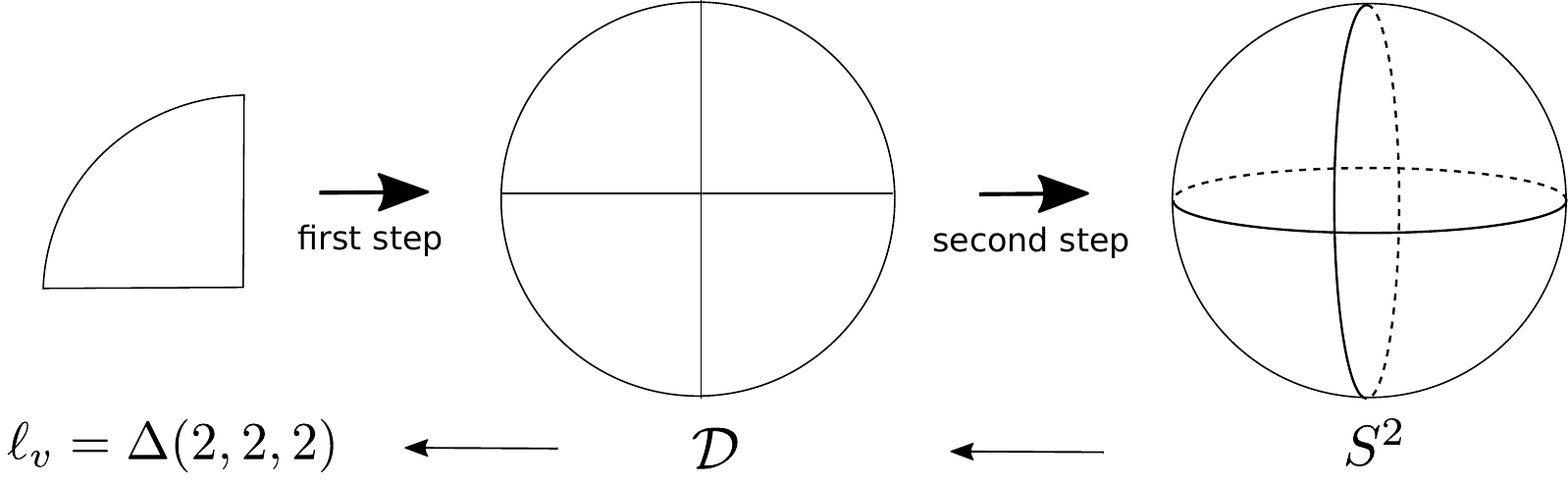}
\caption[Gluing the links of the ideal vertices.]{\footnotesize The effect of the gluing on the link $\ell_v$ of a vertex $v$ of the link ${\mathcal L}_{\mathcal V}$ (which is depicted in Figure \ref{fig:links}--left) of an ideal vertex ${\mathcal V}$ of ${\mathcal P}_t$ (see the proof of Proposition \ref{prop: orbifold covering}). The two top arrows represent the two steps of the construction of $\mathcal{X}_t$, while the bottom ones indicate orbifold coverings.}
\label{fig: gluing_link_of_link}
\end{figure}

Recall from Section \ref{sec: construction} that $\mathcal{X}_t$ is built from some copies of ${\mathcal P}_t$ in two steps. When we pair the facets of ${\mathcal P}_t$ of type ${\mathcal F}_{\l X}$ and ${\mathcal F}_{\m i}$, we glue four copies of $\ell_v$ around its vertex of type $\m0\l A$ and get a disc ${\mathcal D}=S^2/_{\Z/2\Z}$ with mirror boundary. The reason is that ${\mathcal C}$ is right-angled and $\mathcal{N}$ is a hyperbolic manifold, so each edge of its tessellation into copies of ${\mathcal C}$ has valence 4. By doubling along the unpaired facets, we double ${\mathcal D}$ and get the sphere $S^2$. Thus, the map $\mathcal{X}_t\smallsetminus\Sigma_t\to {\mathcal P}_t^\times$ induces at $v$ the orbifold covering $S^2\to\Delta(2,2,2)$, and therefore it is locally modelled on $\R^4\to\R^4/_{(\Z/2\Z)^3}$ near the 1-strata of ${\mathcal P}_t^\times$. The proof is complete.
\end{proof}

\begin{remark}
Recalling Remark \ref{rem: right-angled}, we have also that $\mathcal{X}_{\nicefrac1{\sqrt3}}\to {\mathcal P}_{\nicefrac1{\sqrt3}}$ is an orbifold covering.
\end{remark}

Recall now Remark \ref{rem: transiz orbifold}. By lifting to $\mathcal{X}_t\smallsetminus\Sigma_t$ the geometric structures of the orbifold ${\mathcal P}^\times_t$, and similarly for the rescaled limits, we immediately get:

\begin{cor} \label{cor: transition manifold}
The family $\{\mathcal{X}_t\smallsetminus\Sigma_t\}_{t\in I}$ defines a $C^1$ geometric transition on $\mathcal{N}\times S^1\smallsetminus\Sigma$.
\end{cor}

\begin{remark} \label{rem:cusp_transition}
Recall Sections \ref{sec: horospheres}, \ref{sec transition horospheres} and \ref{sec upper halfspace} about horospheres and transition. From the proof of Proposition \ref{prop: orbifold covering} one can recover the geometric transition on each cusp section of $\mathcal{X}_t$ (see Figure \ref{fig: cusp transition} and \ref{fig: cusp transition2}). We have a path of (non-singular) Euclidean structures (on the 3-torus or $\mathcal K\times S^1$, where $\mathcal K$ is the Klein bottle) collapsing to a Euclidean surface (a flat 2-torus or Klein bottle, which is a cusp section of $\mathcal{N}$), such that by rescaling in the direction of collapse the path extends to Minkowskian structures, via a transitional Galilean structure.
\end{remark}

To conclude the proof of Theorem \ref{teo: main}, it remains to understand what happens near $\Sigma_t$. Recall Definition \ref{def: AdS cone-manifold} of simple hyperbolic, anti-de Sitter, or half-pipe cone-manifold. Recall also Proposition \ref{prop: ridges} where the explicit expressions of the dihedral angles $\theta_t$ and $\varphi_t$ of ${\mathcal P}_t$ are given. We have:

\begin{prop} \label{prop: near_Sigma}
When $t\in I^-$, the space $\mathcal{X}_t$ is a simple anti-de Sitter cone-manifold with spacelike singularity along $\Sigma_t$, whose 2-strata have all the same magnitude $\beta_t=-2\cdot\varphi_t$.

Similarly, $\mathcal{X}_t$ is a simple hyperbolic cone-manifold with cone angles $\alpha_t=2\cdot\theta_t$ when $t\in I^+$, and the rescaled limit $\lim_{t\to0}\mathfrak r_{|t|}\mathcal{X}_t$ is a simple half-pipe cone-manifold.
\end{prop}

\begin{proof}
Let us fix $t\in I^-$. We will show that $\mathcal{X}_t$ is locally modelled on 
$$\mathcal D=D(\p1\cap\p3\cap\p5\cap\p7)$$
in the sense of Definition \ref{def:proj_simple_cone-mfd}, where for each bounding hyperplane of the polytope $\p1\cap\p3\cap\p5\cap\p7\subset\SP^4$ we choose the unique AdS reflection that fixes it (recall Section \ref{sec: reflections} about reflections).

To this purpose, it suffices to look at the effect of the gluing of the copies of ${\mathcal P}_t$ on the link ${\mathcal L}_{\mathcal V}$ of each finite vertex ${\mathcal V}$. By symmetry (Lemma \ref{lem:symmetry}) and Proposition \ref{prop: vertices}, it suffices to consider the vertices ${\mathcal V}_{\p1\p3\m0\l A}$, ${\mathcal V}_{\p1\p3\p5\m0}$ and ${\mathcal V}_{\p1\p3\p5\p7}$ only. Recall from Section \ref{sec: construction} that $\mathcal{X}_t$ is built from some copies of ${\mathcal P}_t$ in two steps.  We refer to Figure \ref{fig: gluing_link}.

\begin{figure}
\includegraphics[scale=.1]{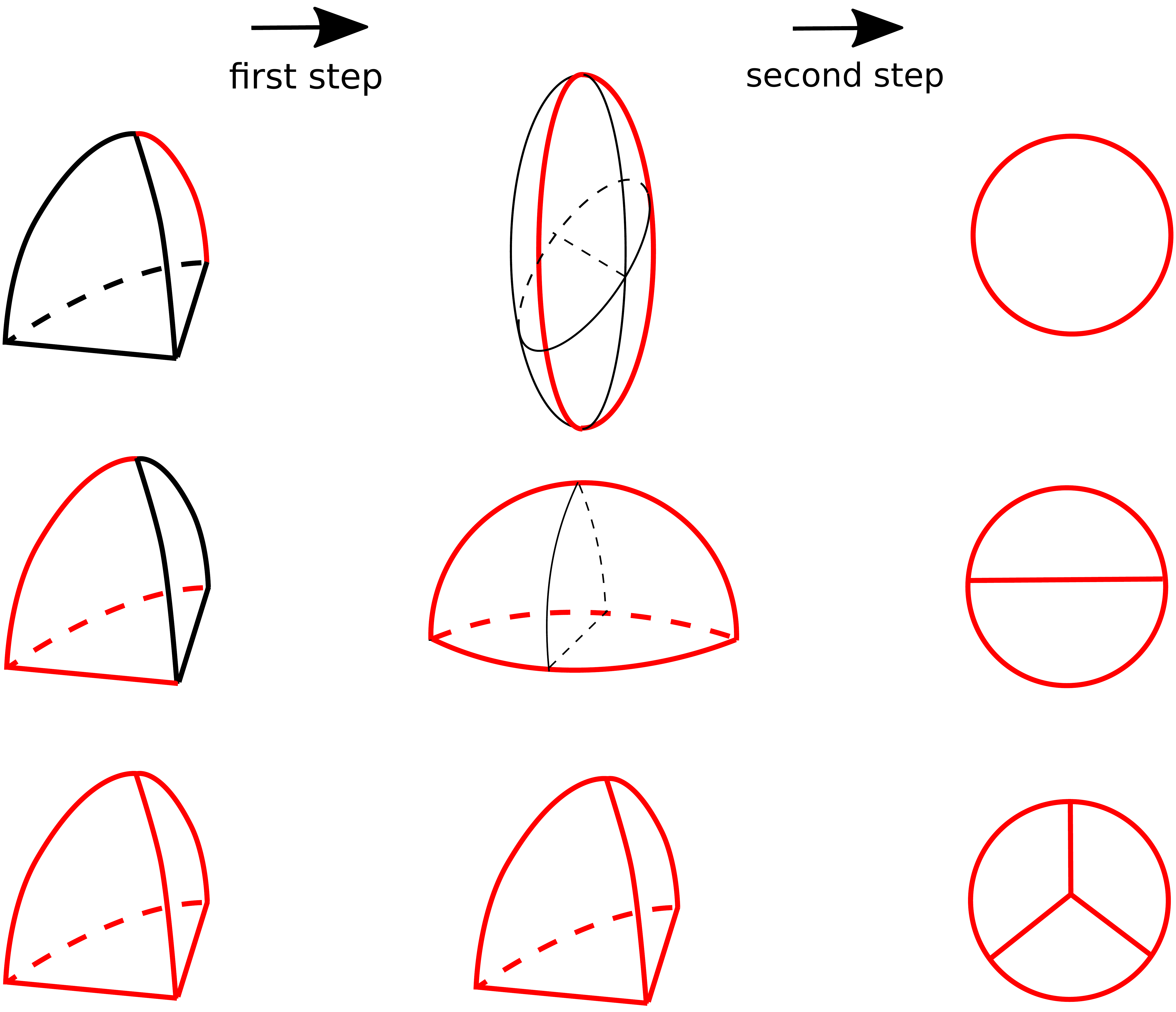}
\caption[Gluing the links of the finite vertices.]{\footnotesize The effect of the gluing on the link ${\mathcal L}_{\mathcal V}$ of a finite vertex ${\mathcal V}$ of ${\mathcal P}_t$ (see the proof of Proposition \ref{prop: near_Sigma}). At each line we see the effect on a different class of vertices (see Figure \ref{fig:links}). The two arrows represent the two steps of the construction of $\mathcal{X}_t$. In the first two columns we have some polyhedra, while in the third column we have cone-manifolds homeomorphic to $S^3$ (the singular locus is a red graph in $S^3$). These polyhedra and cone-manifolds are spherical when $t\in I^+$, and HS (with spacelike red locus contained in the de Sitter region) when $t\in I^-$. Each of the three cone-manifolds in the third column is the link of a point in a 2- (top), 1- (centre), and 0-stratum (bottom) of the cone-manifold $\mathcal{X}_t$.}
\label{fig: gluing_link}
\end{figure}

\begin{itemize}
\item If ${\mathcal V}={\mathcal V}_{\p1\p3\m0\l A}$, the same argument in the proof of Proposition \ref{prop: orbifold covering} implies that at the first step we glue 4 copies of ${\mathcal L}_{\mathcal V}$ around its edge of type $\m0\l A$. The resulting space is a polyhedron in $\HS^3$ (see Section \ref{sec:HS}) obtained as the intersection of two spacelike half-spaces. At the second step this polyhedron is doubled, and we get an HS cone 3-sphere with singular locus a spacelike unknotted circle in the de Sitter region. This is the link of a point in a 2-stratum of $\mathcal D$ (corresponding to $\partial \p 1\cap\partial \p 3$).
\item If ${\mathcal V}={\mathcal V}_{\p1\p3\p5\m0}$, at the first step we just double ${\mathcal L}_{\mathcal V}$ along its $\m0$-face, and then double the resulting polyhedron, to get an HS cone-sphere with singular locus a spacelike unknotted theta-graph in the de Sitter region. This is the link of a point in a 1-stratum of $\mathcal D$ (corresponding to $\partial \p 1\cap\partial \p 3\cap\partial\p5$).
\item If ${\mathcal V}={\mathcal V}_{\p1\p3\p5\p7}$ the link ${\mathcal L}_{\mathcal V}$ is doubled, to get an HS cone-sphere with singular locus a spacelike unknotted complete graph with four vertices in the de Sitter region. This is the link of a vertex of $\mathcal D$ (corresponding to $\partial \p 1\cap\partial \p 3\cap\partial\p5\cap\p7$).
\end{itemize}
In particular, giving $\mathcal{X}_t$ the naturally induced stratification, it is locally modelled on $\mathcal D$.
The proof for $t\in I^-$ is complete.

We omit the details for the hyperbolic and half-pipe case, since the proof goes exactly as in the AdS case with the obvious modifications (see Remark \ref{remark rescaled holonomy} regarding the choice of an HP reflection along each bounding hyperplane for the half-pipe case).
\end{proof}

\begin{figure}[b]
\begin{minipage}[c]{.35\textwidth}
\centering
\includegraphics[scale=.25]{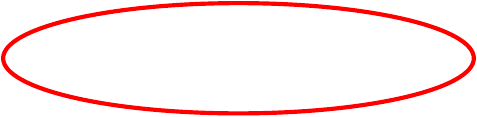}
\end{minipage}
\begin{minipage}[c]{.45\textwidth}
\centering
\includegraphics[scale=.25]{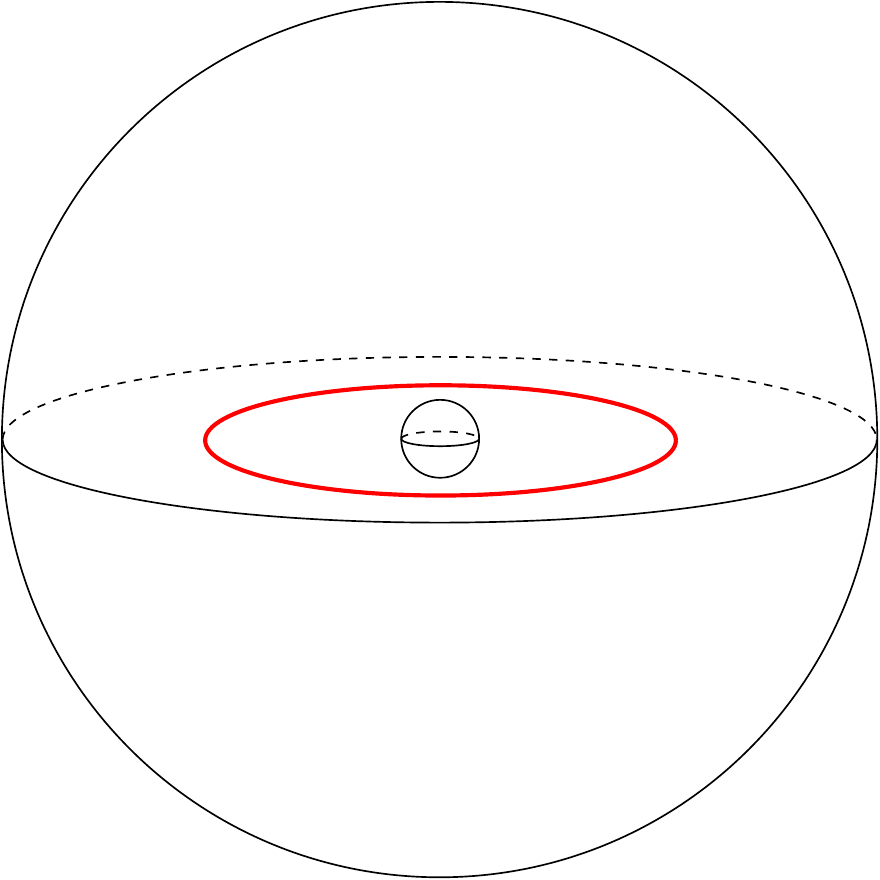}
\end{minipage}
\caption[The link of a point in a 2-stratum of $\Sigma_t$.]{\footnotesize The link of a point in a 2-stratum of $\Sigma_t\subset\mathcal X_t$ is a cone 3-sphere with singular locus an unknotted circle (drawn in red). The geometry is spherical when $t\in I^+$ (left), and HS when $t\in I^-$ (right). In the HS case, the two balls (one internal and one external) are copies of $\Hyp^3$ and represent the timelike directions --- see Section \ref{def: AdS cone-manifold}.}
\label{fig:transition link 1}
\end{figure}

\begin{figure}[b]
\begin{minipage}[c]{.35\textwidth}
\centering
\includegraphics[scale=.25]{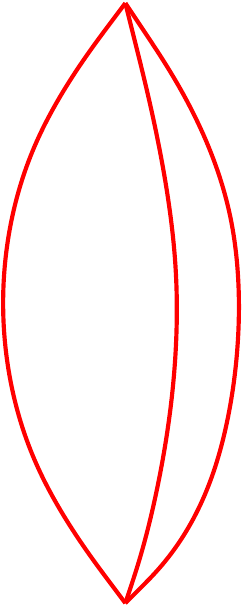}
\end{minipage}
\begin{minipage}[c]{.45\textwidth}
\centering
\includegraphics[scale=.25]{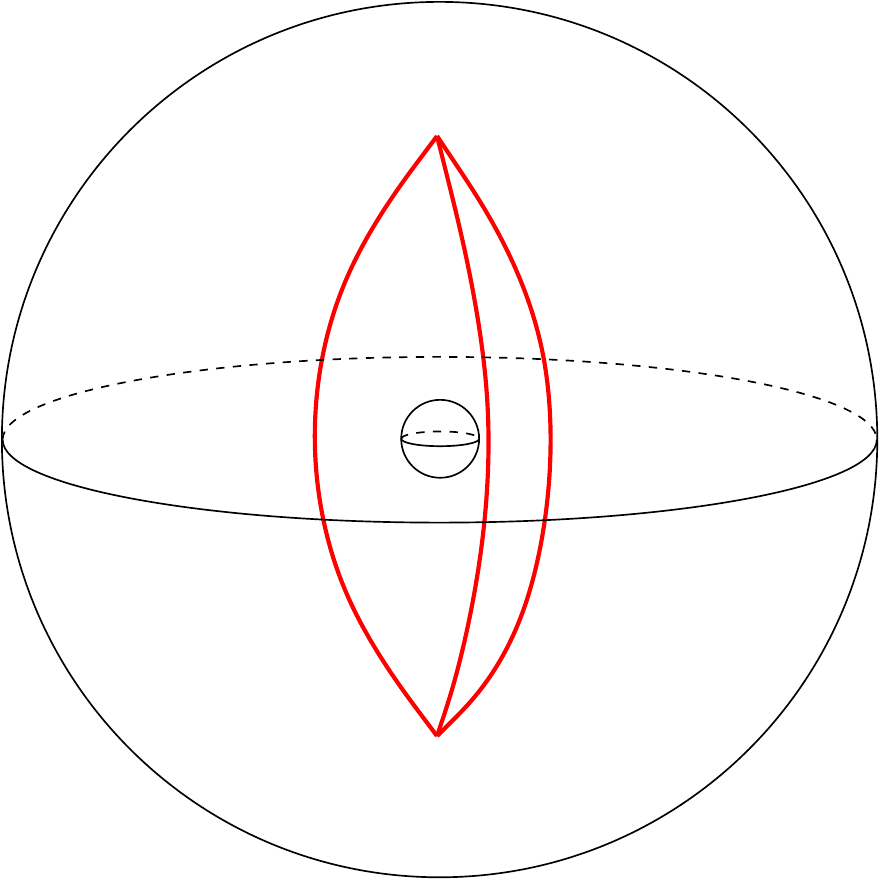}
\end{minipage}
\caption[The link of a point in an edge of $\Sigma_t$.]{\footnotesize
The link of a point in an edge (i.e. a 1-stratum) of $\Sigma_t\subset\mathcal X_t$, $t\neq0$, is a cone 3-sphere with singular locus an unknotted theta-graph.}
\label{fig:transition link 2}
\end{figure}

\begin{figure}[b]
\begin{minipage}[c]{.35\textwidth}
\centering
\vspace{0.22cm}
\includegraphics[scale=.25]{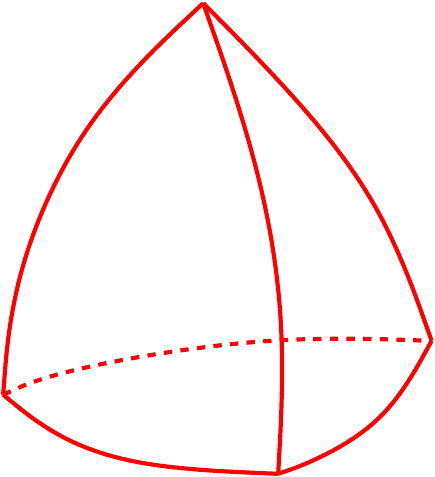}
\end{minipage}
\begin{minipage}[c]{.45\textwidth}
\centering
\includegraphics[scale=.25]{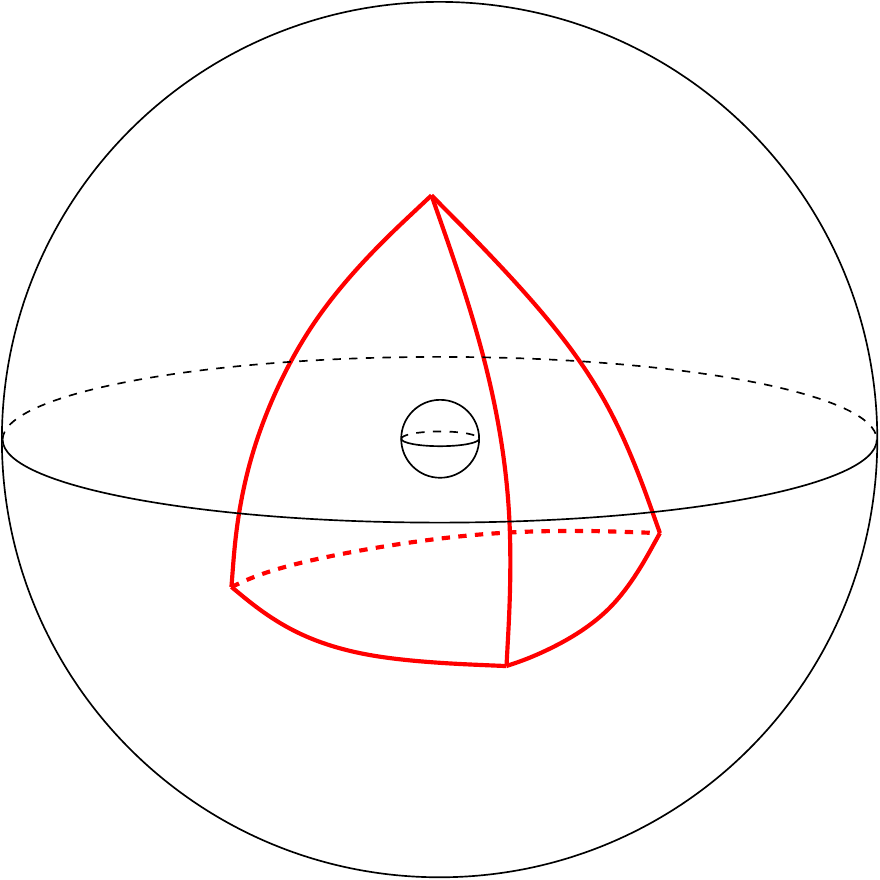}
\end{minipage}
\caption[The link of a vertex of $\Sigma_t$.]{\footnotesize The link of a vertex (i.e. a 0-stratum) of $\Sigma_t\subset\mathcal X_t$, $t\neq0$, is a cone 3-sphere with singular locus an unknotted complete graph on four vertices.}
\label{fig:transition link 3}
\end{figure}

The cone structure on the links of points of {$\Sigma_t$} is drawn in Figures \ref{fig:transition link 1}, \ref{fig:transition link 2} and \ref{fig:transition link 3}. As mentioned in Remark \ref{rem: AdS-cone}, there is indeed a geometric transition from spherical to HS cone structures, as a singular version of the transition that one can visualise in Figure \ref{fig:transition_stabilizers}.

By noticing that as $t\to0$ the cone-manifold $\mathcal{X}_t$ collapses to the hyperbolic 3-manifold $\mathcal{N}$, the proof of Theorem \ref{teo: main} is complete. We conclude with a last observation.

\begin{remark} \label{rem: cuboctahedral_manifolds}
Theorem \ref{teo: main} can be extended as follows. A \emph{cuboctahedral manifold} is a hyperbolic 3-manifold $\mathcal{N}$ that can be tessellated by some copies of the ideal right-angled cuboctahedron ${\mathcal C}$. Note that ${\mathcal C}$ has octahedral symmetry $\Isom({\mathcal C})\cong\Z/2\Z\times\mathfrak S_4$. Note also that every isometry between two faces of ${\mathcal C}$ is the restriction of a symmetry of ${\mathcal C}$. Moreover, as Figure \ref{fig:cuboctahedron} suggests, we have $\Isom({\mathcal P}_t)\cong\Isom({\mathcal C})$ in such a way that every symmetry of ${\mathcal C}={\mathcal P}_0\subset {\mathcal P}_t$ (see Proposition \ref{prop: cuboctahedron}) is the restriction of a symmetry of ${\mathcal P}_t$. (To show this, the same argument of \cite[Proposition 2.4]{RS} applies, by substituting ``upper tetrahedral facet'' with the link of the vertex ${\mathcal V}_{\p1\p3\p5\p7}$.)

Consider now the natural chequerboard colouring of the triangular faces of ${\mathcal C}$, inherited from that of the octahedron. It is easy to check that a symmetry of ${\mathcal C}$ preserves the chequerboard colouring if and only if its $\Z/2\Z$-factor is trivial. Moreover, this holds if and only if the corresponding symmetry of ${\mathcal P}_t$ preserves the half-space $\{x_4\geq0\}$ (and thus fixes the vertex ${\mathcal V}_{\p1\p3\p5\p7}$).

Thanks to this, it is not difficult to conclude that Theorem \ref{teo: main} holds for every cuboctahedral manifold $\mathcal{N}$ with a tessellation such that every pairing map between the copies of ${\mathcal C}$ is induced by a symmetry of ${\mathcal C}$ which preserves the chequerboard colouring. More generally, for every cuboctahedral manifold $\mathcal{N}$ one can find a 4-manifold (such that itself or a double covering is homeomorphic to $\mathcal N\times S^1$) supporting a geometric transition as in the conclusions of Theorem \ref{teo: main}.
\end{remark}

\bibliographystyle{alpha}
\bibliographystyle{ieeetr}
\bibliography{sr-bibliography}

\end{document}